\documentclass[11pt, leqno]{article}%
\usepackage{bm}
\usepackage{xcolor}
\usepackage{eurosym}
\usepackage{amsmath}
\usepackage{amsfonts}
\usepackage{amssymb}
\usepackage{graphicx}
\usepackage{fullpage}
\usepackage{url}
\usepackage{amsthm}
\usepackage{caption}
\usepackage{subcaption}
\usepackage{comment}%
\usepackage[title]{appendix}
\usepackage{cite}
\usepackage[colorlinks=true, citecolor=blue]{hyperref}
\theoremstyle{definition}
\newtheorem{definition}{Definition}[section]
\date{}
\newcommand{\bu}{{\bf u}}
\newcommand{\bP}{\mathbb{P}}
\newcommand{\bE}{\mathbb{E}}
\newcommand{\cred}{\color{red}}
 
\usepackage[margin=1in]{geometry}
\providecommand{\U}[1]{\protect\rule{.1in}{.1in}}
\newtheorem{theorem}{Theorem}

\newtheorem{lemma}{Lemma}
\newtheorem{proposition}{Proposition}

\numberwithin{equation}{section}

\begin{document}
\title{Global Well-posedness of the 2D Stochastic Self-consistent Keller-Segel-Navier-Stokes System with Subcritical Cellular Mass }

\author{Fanze Kong\footnote{Department of Applied Mathematics, University of Washington, Seattle, 98195, WA, USA, email address: {fzkong@uw.edu}.}, Chen-Chih Lai \footnote{email address: {cclai.math@gmail.com}} and Krutika Tawri\footnote{Department of Applied Mathematics, University of Washington, Seattle, 98195, WA, USA, email address: {ktawri@uw.edu}.} }

\maketitle

 \begin{abstract}

 We consider a stochastic Keller-Segel-Navier-Stokes system in $\mathbb R^2$ driven by a multiplicative white-in-time stochastic forcing.
  The model describes the collective motion of cells in an ambient stochastic fluid flow, where the cells are attracted by a chemical substance and transported by the ambient fluid velocity, and the fluid motion is self-consistently driven by forces induced by the cells. We prove the  existence of a unique mild solution globally-in-time to the two-dimensional stochastic Keller-Segel-Navier-Stokes system with subcritical mass. 
\end{abstract}

\section{Introduction}
In this paper, we investigate the following two-dimensional stochastic Keller-Segel-Navier-Stokes (KS-NS) system:
\begin{align}\label{PKSNS-time-dependent}
\left\{\,
\begin{array}{ll}
dn= \Delta ndt - \nabla\cdot(n\nabla c)dt - {\bf u}\cdot\nabla ndt,&x\in\mathbb R^2,t>0,\\
0= \Delta c+ n, &x \in\mathbb R^2,t>0,\\
d{\bf u}+(({\bf u}\cdot\nabla ){\bf u} + \nabla P)dt = \Delta {\bf u}dt +  n\nabla cdt+{\bf f}({\bf u})dW_t,\quad \nabla\cdot {\bf u} = 0,&x\in\mathbb R^2,t>0,\\
(n,{\bf u})(\cdot,0)=(n_0,{\bf u}_0),&x\in\mathbb R^2,
\end{array}
\right.
\end{align}
where $n = n(x,t) : \mathbb{R}^2 \times \mathbb{R}^+ \to \mathbb{R}^+$, 
$c = c(x,t) : \mathbb{R}^2 \times \mathbb{R}^+ \to \mathbb{R}^+$, 
${\bf u} = {\bf u}(x,t) : \mathbb{R}^2 \times \mathbb{R}^+ \to \mathbb{R}^2$, and 
$P = P(x,t) : \mathbb{R}^2 \times \mathbb{R}^+ \to \mathbb{R}$ represent the cell density, the  chemical concentration, the fluid velocity field, and the pressure, respectively. We supplement the problem with initial data $(n_0,{\bf u}_0)$ that is assumed to satisfy $\nabla \cdot {\bf u}_0 = 0$.  Here $\{W_t,\, t \ge 0\}$ is a Wiener process and ${\bf f}({\bf u})dW_t$ represents an external random driving force. The model \eqref{PKSNS-time-dependent} describes the collective biased motion of chemotactic cells along chemical concentration gradients in a noisy fluid environment, where the coupling $n\nabla c$ represents the friction exerted by the fluid on the moving cells,
and acts on the fluid as an anti-force.

System (\ref{PKSNS-time-dependent}) is  a coupled system consisting of the stochastic incompressible Navier-Stokes (NS) equation and the Keller-Segel model. When the transport term ${\bf u}\cdot \nabla n$ is equal to zero  identically, 
the  evolution equations for $n$ and $c$ in \eqref{PKSNS-time-dependent} are reduced to the classical 
Patlak-Keller-Segel (PKS) system \cite{patlak1953random,Keller1970,Keller1971}, which serves as a paradigm for describing the traveling bands of 
\textit{E.~coli}.  The deterministic PKS system and its variants have been 
extensively studied in the literature; we refer the reader to the 
representative survey articles \cite{hillen2009user,horstmann2003,horstmann2004}.  One of the most celebrated results in the study of the PKS model is the so-called ``chemotactic collapse'' phenomenon.  More precisely, there exists a critical threshold $8\pi$ such that if the initial 
cellular mass  $M:=\int_{\mathbb{R}^{2}} n(x,0)\,dx < 8\pi$, then the free-energy solution to the PKS system is globally well-posed \cite{blanchet2006two,dolbeault2004optimal}. Otherwise, if the initial mass
$M>8\pi$, then the PKS system admits solutions that can blow-up in finite time 
\cite{nanjundiah1973,childress1981,herrero1996,senba2000some,wang2002steady}.   For the critical case $M=8\pi$, global-in-time existence of solutions to the PKS system is established in \cite{biler20068pi} for the radial case and \cite{velazquez2004point} for the general case.  In particular, Blanchet et al. \cite{blanchet2008infinite} constructed an infinite-time blow-up free-energy solution with finite second moment.  Recently, refined blow-up profile and stability properties were further investigated in \cite{davila2020existence} by Davila et al. 
The study of the stochastic NS equations also remains an active and significant area of research in fluid mechanics, with an extensive body of literature devoted to the subject. For results concerning the existence and uniqueness of solutions to the stochastic NS equations, we refer the reader to the relevant works
\cite{flandoli1995martingale,mikulevicius2004stochastic,da2003ergodicity,hairer2008spectral}.

{ The KS-NS model is more challenging to study than the PKS model due to the presence of stronger nonlinearities.}
The deterministic KS-NS system (counterpart to \eqref{PKSNS-time-dependent} obtained by setting ${\bf f}={\bf 0}$) was proposed by Y. Gong and S. He \cite{SimingHe},  and has been extensively studied in recent years; see \cite{SimingHe,lai2021global,kong2024global,he2023enhanced} for related works.  In particular, under the subcritical cellular mass condition, global-in-time existence of classical solutions, was established on the torus $\mathbb T^2$, as well as in $\mathbb{R}^2$, in \cite{SimingHe}.  Moreover,  Lai et al. in \cite{lai2021global} show global existence of the solution to the KS-NS equation with critical cellular mass. In addition, Huang and Shen \cite{huang2023bound} proposed a structure-preserving numerical scheme to compute the numerical solutions of  the deterministic KS-NS equation.  We remark that the external force in the ${\bf u}$-equation of (\ref{PKSNS-time-dependent}) is given by $n\nabla c,$ which represents the friction exerted by the cells.  Consequently, the model is referred to as the self-consistent KS-NS equation.   If the external force in the NS equation is given by $n\nabla\phi$ for some known function $\phi$, {which typically represents the effects of gravity on the fluid equations}, then the corresponding model is non-self-consistent.  For {results regarding} the solvability of the non-self-consistent KS-NS model, we refer the reader to \cite{duan2010global,duan2014note,kozono2016existence,winkler2016global,winkler2014stabilization}.

When modeling bacterial dynamics, it is more realistic to account for environmental randomness and stochastic external influences, which naturally leads to the study of the stochastic KS–NS system. This perspective is particularly important because such random effects can play a significant role in the system’s behavior. However, incorporating multiplicative noise introduces substantial analytical challenges, especially in establishing global well-posedness.
  For the investigation of stochastic {\it non-self-consistent} KS-NS equation, we refer the reader to   \cite{zhai20202d,zhang2025keller}. To our knowledge, there have been no results devoted to the study of stochastic {\it self-consistent} KS-NS equations. Motivated by this, we investigate the Cauchy problem for the two-dimensional KS-NS system (\ref{PKSNS-time-dependent}), where the coupling force is self-consistent and in the NS equation the force is induced by the friction between the fluid flow and the cell.

The main goal of this paper is to establish the global existence of a unique { smooth mild solution, preserving the free-energy structure,} to the stochastic KS-NS system (\ref{PKSNS-time-dependent}) under the subcritical mass condition $M<8\pi$.

In contrast to \cite{zhai20202d}, we consider the influence of the frictional forces  $n\nabla c$ which leads to a stronger nonlinearity and thus a stronger singular behavior in the analysis of global well-posedness for (\ref{PKSNS-time-dependent}). Our techniques rely on first proving the existence of a local-in-time solution and then extending this solution globally-in-time. In our proof we introduce an approximation system using a cut-off function. We prove the existence of a unique solution to the approximation system via a fixed-point argument that utilizes Ito calculus in $L^p$-spaces. The cut-off is eventually resolved by employing a stopping time argument, leading to a maximal local pathwise solution. We then prove that the local-in-time solution enjoys a higher regularity using the semi-group estimates. To obtain global-in-time existence, we derive delicate entropy bounds. The use of entropy estimates is novel, compared to the other works such as \cite{zhai20202d}, due to the nonlinearities in the present setting. This argument is a stochastic analogue of the fact that the deterministic KS-NS system admits a free-energy functional given by
\begin{align}\label{freeenergy}
\mathcal{J}[n,u]
=
\int_{\mathbb R^2}
\left(
n\ln n
-
\frac{1}{2}nc
+
\frac{1}{2}|{\bf u}|^2
\right)dx,
\end{align}
where the first term is the entropy of the cell density, the second term represents the potential energy, and the third term denotes the kinetic energy of the fluid velocity field.  However, due to the fact that we pose our problem on the unbounded spatial domain $\mathbb{R}^2$, the traditional free-energy does not yield the desired
bounds (almost surely) due to the absence of a lower bound for $\int_{\mathbb{R}^2}n\ln n$. Hence, we define a modification of the free energy \eqref{freeenergy} which cuts off the growth of $\ln n$. Then, we bootstrap the bounds obtained using the modified free energy to obtain global-in-time bounds for more regular norms. Finally we show that the global solution we constructed has the desired free-energy \eqref{freeenergy} structure.

\subsection{Definitions and Main Results}
In this section, we introduce the definition of a (pathwise) mild solution and state our main results concerning well-posedness of \eqref{PKSNS-time-dependent}. We begin by setting up the mathematical framework and give a few basic preliminaries.

Let $A$ be the realization of the Stokes operator $-\mathcal {P}\Delta$, where 
$\mathcal {P}$ denotes the Helmholtz projection from $L^{p}(\mathbb{R}^{2})$, with $1<p<\infty$, onto
\[
L^p_\sigma(\mathbb{R}^2) := \{ {\bf u} \in L^p(\mathbb{R}^2) : \text{div}\, {\bf u} = 0 \}:=\overline{C^{\infty}_{0,\sigma}(\mathbb{R}^2)}^{\Vert\cdot\Vert_{L^p(\mathbb{R}^2)} },
\]
where $C_{0,\sigma}^{\infty}(\mathbb{R}^2):=\{{\bf u} \in C_0^{\infty}(\mathbb{R}^2) : \text{div}\, {\bf u} = 0\}.$  We denote the domain of this operator by,
\[
D(A^\beta;p) =\big( \dot{W}^{2\beta,p}(\mathbb{R}^2 ) \;\cap\; L^p(\mathbb{R}^2)\big)_{\sigma},\qquad \forall \beta\in[0,1].
\]
For simplicity, we set $\|\cdot\|_{\beta} := \|\cdot\|_{D(A^{\beta};p)}.$   In what follows, $\{e^{t\Delta}\}_{t \ge 0}$ and $\{e^{-tA}\}_{t \ge 0}$ will denote the heat semigroup and the Stokes semigroup on $L^p(\mathbb{R}^2)$ and $L^p_\sigma(\mathbb{R}^2)$, respectively. 

To define the stochastic integral appearing on the right hand side of \eqref{PKSNS-time-dependent}, we consider the stochastic basis $(\Omega, \mathcal{F}, \{\mathcal{F}_t\}_{t \ge 0}, \mathbb{P})$, where the filtration $\{\mathcal{F}_t\}_{t \ge 0}$ satisfies the so-called normal/usual conditions (see e.g. \cite{PR07}).
Let $U$ be a real  separable Hilbert space with complete orthonormal basis denoted by $\{e_j\}_{j\ge1}$. We assume that $\{W_t,\, t \ge 0\}$ is a $U$-valued  $ \{\mathcal{F}_t\}_{t \ge 0}$-Wiener process.

Let $S$ be any Banach space and $\mathcal{N}(0,T; S)$ be the space of (equivalence classes of) functions $\xi : [0,T] \times \Omega \to S$ which are progressively measurable.  For $q \in [1, \infty)$, we set

\begin{equation*}
\mathcal{M}^q(0,T; S)
=
\left\{
\xi \in \mathcal{N}(0,T; S) :
\mathbb{E} \int_0^T \|\xi(s)\|^q \, ds < \infty
\right\}.
\tag{2.7}
\end{equation*}

Denote by $\gamma\left(U, X\right)$ the space of $\gamma$-radonifying operators 
from $U$ to $X$ with Banach space $X$, equipped with the norm  
$$\|{\bf R}\|_{\gamma(U; X)}
=\|\Big(\sum\limits_{j \ge 1}|{\bf R}e_j|^2\Big)^{1/2}\|_{X}.$$ 
For example, if  $X=D(A^\beta;p)$, then
$$\|{\bf R}\|_{\gamma(U; D(A^{\beta};p))}
=\|\Big(\sum\limits_{j \ge 1}|{\bf R}e_j|^2\Big)^{1/2}\|_{L^p}+\|\Big(\sum\limits_{j \ge 1}|A^\beta({\bf R}e_j)|^2\Big)^{1/2}\|_{L^p}.$$

Now we describe the assumptions we impose on the noise coefficient ${\bf f}: D(A^\beta;p) \rightarrow \mathcal \gamma(U,D(A^\beta;p))$ appearing in the stochastic integral in \eqref{PKSNS-time-dependent}.  Denoting \[
{\bf f}_j({\bf u}) := {\bf f}({\bf u})e_j\in D(A^\beta;p),
\]
we assume that ${\bf f}$ satisfies the following conditions:

\medskip

\noindent{(H1).}  
There exists a positive constant $K$ such that for ${\bf u} \in D(A^\beta;p)$,
\[
\|\Big(\sum\limits_{j \ge 1}|{\bf f}_j({\bf u})|^2\Big)^{1/2}\|^q_{L^p}\leq K\big(1+\|{\bf u}\|^q_{L^p}\big), \text{ and } \|\Big(\sum\limits_{j \ge 1}|A^{\beta}{\bf f}_j({\bf u})|^2\Big)^{1/2}\|^q_{L^p}\leq K\big(1+\|A^{\beta}{\bf u}\|^q_{L^p}\big), 
\]
where constant $K$ depends on
 $\beta\in [0,1]$, $q>1$ and $p>1.$

\noindent{(H2).}  
There exists a positive constant $K$ such that for all ${\bf u}_1, {\bf u}_2\in D(A^\beta;p)$, 
\[
\|\Big(\sum_{j\geq 1}|{\bf f}_j({\bf u}_1)-{\bf f}_j({\bf u}_2)|^2\Big)^{1/2}\|^q_{L^p}
\le K \|{\bf u}_1-{\bf u}_2\|^q_{L^p},
\]
and 
\[ \|\Big(\sum_{j\geq 1}|A^{\beta}({\bf f}_j({\bf u}_1)-{\bf f}_j({\bf u}_2))|^2\Big)^{1/2}\|^q_{L^p}\leq K\|A^{\beta}({\bf u}_1-{\bf u}_2)\|^q_{L^p}, 
\]
where constant $K$ may depend on
 $\beta\in [0,1]$, $q>1$ and $p>1.$

\medskip

\medskip

We are now in a position to give define different notions of solutions.

\begin{definition}\label{def31}
Consider a filtered probability space $(\Omega, \mathcal{F}, \{\mathcal{F}_t\}_{t \ge 0}, \mathbb{P})$ with {initial data} $(n_0,{\bf u}_0)\in L^p(\mathbb R^2)\times (L_{\sigma}^s(\mathbb R^2)\cap \dot{W}^{1,r}(\mathbb R^2))$ for some $p>1$, $s> 2$ and  $r>2$.
We say that $(n, {\bf u})$ is a \textbf{mild solution} of system (\ref{PKSNS-time-dependent}) if
$(n, {\bf u})\in \big(\mathcal M^{q}(0,T;L^p(\mathbb R^2))\cap C(0,T; L^p(\mathbb R^2))\big)\times  \big(\mathcal M^{q}(0,T;L^p(\mathbb R^2)\cap\dot{W}^{1,r}(\mathbb R^2))\cap C(0,T; L^s_{\sigma}(\mathbb R^2) \cap\dot{W}^{1,r}(\mathbb R^2))\big)$ with some $q\geq 2$  satisfies
\begin{align}\label{3p1mar}
n(t) &= e^{t\Delta} n_0
      - \int_0^t e^{(t-s)\Delta} \big({  \nabla \cdot(n(s){\bf u}(s))}\big)\,ds
      - \int_0^t e^{(t-s)\Delta} \nabla\cdot\big(\,n(s)\nabla c(s)\big)\,ds, 
\end{align}
\begin{align}\label{3p2mar}
{\bf u}(t) &= e^{-t A} {\bf u}_0
      - \int_0^t e^{-(t-s) A} \mathcal{P}\big(({\bf u}(s)\cdot\nabla){\bf u}(s)\big)\,ds
      + \int_0^t e^{-(t-s) A} \mathcal{P}\big(n(s)\nabla c(s)\big)\,ds\nonumber \\
     &\quad + \int_0^t e^{-(t-s) A} \mathcal P{\bf f}({\bf u}(s))\,dW_s,
\end{align}
 and $-\Delta c=n,$ $\mathbb P$-a.s.
\end{definition}
In this article, we will prove the existence of a local-in-time mild solution to the system \eqref{PKSNS-time-dependent}.

Before stating our main result, we remark that the cellular density $n\in L^1(\mathbb R^2)$ for all $t>0$ if the initial density $n_0\in L^1(\mathbb R^2).$  Indeed, this follows from the fact that since the velocity field  ${\bf u}$ is divergence-free, the cellular mass, given below, is formally conserved in KS-NS system (\ref{PKSNS-time-dependent}), 
\[
\int_{\mathbb R^2} n(x,t)\,dx
=
\int_{\mathbb R^2} n_0(x)\,dx
=: M,\text{ for all }t\in(0,\infty).
\]
We will now state the main result of this paper.
\begin{theorem}\label{thm1}
   Assume that the deterministic initial condition satisfies
\begin{align}\label{regularic}
& n_0 \in  L^1(\mathbb R^2)\cap W^{2,\hat p}(\mathbb R^2)\cap C^2(\mathbb R^2), ~n_0\ln(1+|x|)\in L^1(\mathbb R^2), ~ n_0>0 \ \text{in}\ \mathbb R^2, \nonumber\\
& {\bf u}_0 \in W^{2,\hat p}(\mathbb R^2),~\text{div~} {\bf u}_0 = 0
\end{align}
for any $\hat p>1$.
Suppose that the noise coefficient ${\bf f}$ satisfies the two conditions (H1)-(H2) on growth and Lipschitz continuity. If the initial cellular mass is subcritical, that is we have
\begin{align}\label{smallness_condition_n0}
M:=\int_{\mathbb R^2}n_0dx<8
\pi,
\end{align}
then there exists a unique global-in-time mild solution
\begin{align*}
  n \in L^{2}(\Omega;C^1([0,\infty); C^2(\mathbb R^2 ))), ~~ {\bf u}\in  L^{2}(\Omega; L^\infty(0,\infty; W^{2,2+\epsilon}(\mathbb R^2))),
\end{align*}
with $\epsilon>0$ sufficiently small, 
to the stochastic system (\ref{PKSNS-time-dependent}) in the sense of Definitions \ref{def31}. Moreover, this solution also agrees with the stochastic counterpart of the so-called `free-energy' solution as it satisfies the following bounds: 
\begin{align}\label{energybounds}
   & \mathbb E\left[\sup_{t\leq T}  |\mathcal{J}[n,{\bf u}](t) |^k\right]
+ \bigg[\int_0^T \int_{\mathbb{R}^2}
\left(
2n|\nabla(\ln n - c)|^2
+ \frac12 |\nabla {\bf u}|^2
\right)
dx\,ds\bigg]^k
\le
C_{k,T},
\end{align}
for $k\geq 1$, $\mathcal J[n,{\bf u}]$ is defined in (\ref{freeenergy}) and $C_{k,T }>0$ is a constant depending on $k$, $T$ and   the initial data.
\end{theorem}
Our aim for the following sections is to prove Theorem \ref{thm1}.  The paper is organized as follows.  In Section \ref{sect2}, we present the necessary preliminaries, including semi-group estimates which are foundational to our existence proof. Section \ref{sect3} is devoted to the local-in-time existence of a mild solution to (\ref{PKSNS-time-dependent}). This solution is constructed via a fixed-point argument. In Section \ref{sect4}, we extend the local solution to a global-in-time mild solution. The extension is time relies on the use of the entropy estimates and a stopping-time argument.

\section{Preliminaries}\label{sect2}
In this section, we state well-known inequalities that will be used in deriving appropriate estimates for the fluid velocity and the cell density in Sections \ref{sect3} and \ref{sect4}.  
Firstly, we introduce the following notation: the inequality
\[
A \lesssim B,
\] 
represents the existence of a universal positive constant $C$ such that $A \le C B$. 

Our first observation is that second equation in \eqref{PKSNS-time-dependent} is a Poisson equation. We thus consider the solution which is determined by the Newtonian potential. In particular, we choose
\begin{align}\label{csolution_convolution}
c=-(\Delta^{-1})n:=-
\frac{1}{2\pi}
 \ln |\cdot|\ast n.
 \end{align}
 
 For this solution, by using Hardy-Littlewood-Sobolev inequality,  we obtain the following bounds:  
\begin{lemma}[\hspace{-0.12cm}\text{\cite[Lemma 2.5]{Nagai-2011}}]\label{lemma28}
For any $\hat q\in (2,\infty),$ one has the following a-priori estimate:
\begin{equation}
\|\nabla \tfrac{1}{2\pi} \ln \frac{1}{|\cdot|} * f\|_{L^{\hat q}}
\le C \|f\|_{L^{\frac{2\hat q}{2+\hat q}}}
\label{eq:3.3},  
\end{equation}
where $C>0$ is a constant depending on $\hat q.$

In addition, let $2 < q \le \infty$. Then for any $f \in L^1 \cap L^{q}$, we have
\[
\|\nabla \tfrac{1}{2\pi} \ln \frac{1}{|\cdot|} * f\|_{L^\infty}
\le C_q \|f\|_{L^1}^{\frac{q-2}{2(q-1)}} \|f\|_{L^q}^{\frac{q}{2(q-1)}},
\]
where
\[
C_q = (2\pi)^{1/2} 
\left(\frac{q-1}{q-2}\right)^{1/2} \bigg[
\left(\frac{q}{q-1}\right)^{\frac{q-2}{2(q-1)}}+ \left(\frac{q}{q-1}\right)^{-\frac{q}{2(q-1)}}\bigg].
\]
\end{lemma}

Our local-in-time existence result is based on an application of a fixed point theorem which relies on obtaining appropriate estimates. These estimates rely on the following semigroup estimates of the Laplacian and the Stokes operator:
\begin{lemma}[\hspace{-0.12cm}\text{\cite[Theorem 6.13]{Pazy-book1983};\cite[Lemma 1.3]{winkler2010aggregation}}]\label{lemma:semigroup-Lp-Lq}
Let $1 \le q \le p \le \infty$ with $p \neq 1$ and $q \neq \infty$. Then there exists a constant $C>0$ such that for all $t>0$ the following estimates of the heat and Stokes semigroups hold:
\begin{align}
\| e^{t\Delta} f \|_{L^p} + \|e^{-tA} (\mathcal {P} f )\|_{L^p} &\le C\, t^{-\frac{1}{2}\left(\frac{2}{q} - \frac{2}{p}\right)} \| f\|_{L^q}, \label{eq:3.4a} \\
\| e^{t\Delta} \nabla \cdot {\bf F} \|_{L^p} + \| e^{-tA} (\mathcal{P}\nabla \cdot {\bf F}) \|_{L^p} &\le C\, t^{-\frac{1}{2} - \frac{1}{2}\left(\frac{2}{q} - \frac{2}{p}\right)} \|{\bf F}\|_{L^q} \label{eq:3.4b},\\
\| \nabla(e^{t\Delta} f) \|_{L^p}  &\le C\, t^{-\frac{1}{2}-\frac{1}{2}\left(\frac{2}{q} - \frac{2}{p}\right)} \|f\|_{L^q}.\label{eq:3.4a1}
\end{align}
Moreover, suppose that $0\leq\beta_2\leq \beta_1<+\infty$, then we have
\begin{align}
\| A^{\beta_1}(e^{t\Delta} f) \|_{L^p} + \| A^{\beta_1}(e^{-tA} \mathcal {P} f )\|_{L^p} &\le C\, t^{-\beta_1+\beta_2-\frac{1}{2}\left(\frac{2}{q} - \frac{2}{p}\right)} \|A^{\beta_2}f\|_{L^q}, \label{eq:3.4afractional} \\
\| A^{\beta_1}(e^{t\Delta} \nabla \cdot {\bf F}) \|_{L^p} + \| A^{\beta_1}(e^{-tA} (\mathcal{P}\nabla \cdot {\bf F})) \|_{L^p} &\le C\, t^{-\beta_1+\beta_2-\frac{1}{2} - \frac{1}{2}\left(\frac{2}{q} - \frac{2}{p}\right)} \|A^{\beta_2}{\bf F}\|_{L^q}. \label{eq:3.4bfractional}
\end{align}
\end{lemma}

Next, we recall the following estimate which will be used to prove lower bounds for the entropy of the system:
\begin{lemma}[\text{Control of the negative part of the entropy \cite[Lemma 2.9]{lai2021global}}]\label{lemm2mar9}
For any $g$ such that $(1+\ln(1+|x|^2))g \in L^1_+(\mathbb{R}^2)$, we have 
$g\ln^- g \in L^1(\mathbb{R}^2)$ and that
\[
\int_{\mathbb{R}^2} g(x)\ln^- g(x)\,dx
\le
2\int_{\mathbb{R}^2} g(x)\ln(1+|x|^2)\,dx
+ \ln\pi \int_{\mathbb{R}^2} g(x)\,dx
+ \frac{1}{e}.
\]
\end{lemma}

We shall use the following logarithmic Hardy-Littlewood-Sobolev (HLS) inequality:
\begin{lemma}[\hspace{-0.12cm}\text{\cite{Beckner1993,CL1992}}]\label{logHLSlemma}
 Let $n \in L^1(\mathbb{R}^2)$ with $n \ge 0$ and 
$\int_{\mathbb{R}^2} n \, dx = M>0$, then there exists a constant $C(M)>0$ depending on $M$ such that
\begin{align}\label{logarHLS}
\frac{1}{2\pi}\int_{\mathbb{R}^2} \int_{\mathbb{R}^2} n(x) \ln \frac{1}{|x-y|} n(y) \, dx \, dy
\le \frac{M}{4\pi} \int_{\mathbb{R}^2} n(x) \ln n(x) \, dx + C (M).
\end{align}
\end{lemma}

Finally, we recall the Nash inequality:
\begin{lemma}[\hspace{-0.12cm}\text{\cite{nash1958continuity}}]\label{Nashinequality}
For any $g \in L^{1}(\mathbb{R}^{2}) \cap H^{1}(\mathbb{R}^{2})$, there exists a constant $C_N>0$ such that
\begin{equation}\label{nash_formula}
\|g\|_{L^2}^{4}
\le
C_N \|\nabla g\|_{L^2}^{2}\,\|g\|_{L^1}^{2}.
\end{equation}
\end{lemma}
With these preliminary results, we proceed to discuss the local-in-time existence of a mild solution to (\ref{PKSNS-time-dependent}) in the sense of Definition \ref{def31}.
\section{Local-in-time Existence}\label{sect3}

In this section, we discuss the local existence and regularity of mild solutions to (\ref{PKSNS-time-dependent}).  We begin with the definition of local mild solutions to (\ref{PKSNS-time-dependent}) i.e. a local-in-time version of  Definition \ref{def31}:
\begin{definition}[Local and maximal local mild solutions]\label{def:local_solution}
We say that $(n, {\bf u}, \tau)$ is a \textbf{local mild solution} of system~(\ref{PKSNS-time-dependent}) if
\begin{itemize}
\item[(1)] $\tau$ is an {$\{\mathcal{F}_t\}_{t\geq 0}$}-stopping time and $(n, {\bf u})$ is an { $\{\mathcal{F}_t\}_{t\geq 0}$}-progressively measurable stochastic process with values in 
    $ L^{p}(\mathbb R^2) \times (L_{\sigma}^s(\mathbb R^2)\cap \dot{W}^{1,r}(\mathbb R^2))$ for some $p>1$, $s>2$  and  $r >2$;
    \item[(2)] there exists a non-decreasing sequence of stopping times $\{\tau_l,\, l \ge 1\}$ with $\tau_l \uparrow \tau$ a.s. as $l \to \infty$, such that 
    $\{(n(t \wedge \tau_l),\, \, {\bf u}(t \wedge \tau_l)),\, t \ge 0\}$ is a mild solution to system~(\ref{PKSNS-time-dependent}) in the sense of Definition \ref{def31}.
\end{itemize}
Moreover, if 
\[
\limsup_{t \uparrow \tau} \big( \|n\|_{X_n^t}  + \|{\bf u}\|_{X_{\bf u}^t} \big) = \infty 
\quad \text{on } \{\omega, \tau < \infty\} \text{ a.s.},
\]
then the local solution $(n, {\bf u}, \tau)$ is called a \textbf{maximal local solution}.
\end{definition}

Our proof for the main result begins with proving the existence of a local-in-time mild solution in the sense of Definition \ref{def:local_solution}.
For that purpose, we introduce the following functional spaces where we seek a solution. 
\begin{align*}
X_T
&:= \Big\{(n,{\bf u})  :
\; n\in C(0,T;L^{p}(\mathbb R^2)),
{\bf u}\in C(0,T;(L^{\frac{pq}{p-q}}(\mathbb R^2)\cap \dot{W}^{1,r}(\mathbb R^2))_{\sigma})
\Big\},
\end{align*}
endowed with the norm:
\begin{align*}
    \|(n,{\bf u})\|_{X_T}
&:= \sup_{t\in(0,T)}\Big( \|n(t)\|_{L^{p}(\mathbb R^2)}
       + \|{\bf u}(t)\|_{L^{\frac{pq}{p-q}}(\mathbb R^2)}+ \|\nabla {\bf u}(t)\|_{L^{r}(\mathbb R^2)}\Big).
\end{align*}
Here, for an appropriately chosen $\epsilon>0$, we consider $p,q,r$ of the form:
\begin{align}\label{pqr}
    p=2-\epsilon,\qquad q=\frac{2(2-\epsilon)}{2+\epsilon},\qquad r=2+\epsilon.
\end{align} 
For this choice, note that 
\begin{align}\label{quotient}
    \frac{pq}{p-q}= \frac4{\epsilon}-2,
\end{align}
can be made arbitrarily large depending on our choice of $\epsilon>0$.  In what follows, we fix a sufficiently small $\epsilon>0$.

Next, we notate the individual spaces for the cell density and the fluid velocity as ${X_n^T}=C(0,T;L^p(\mathbb{R}^2))$ and ${X_{\bf u}^T}= C(0,T;(L^{\frac{pq}{p-q}}(\mathbb{R}^2)\cap\dot{W}^{1,r}(\mathbb{R}^2))_{\sigma})$ respectively. These spaces are equipped with the norms, 
\begin{align}\label{Xtspace}
    \Vert n\Vert_{X^t_n}:=\sup_{s\in(0,t)}\Big( \|n(\cdot,s)\|_{L^{p}}\Big), \qquad 
\|{\bf u}\|_{X_{\bf u}^t}:=\sup_{s\in(0,t)}\Big(  \|{\bf u}(\cdot,s)\|_{L^{\frac{pq}{p-q}}}+\|\nabla {\bf u}(\cdot,s)\|_{L^{r}}\Big).
\end{align}
{{We note that $(L^{\frac{pq}{p-q}}(\mathbb{R}^2)\cap\dot{W}^{1,r}(\mathbb{R}^2))_{\sigma}$ is a Banach space equipped with the norm $ \|\cdot\|_{L^{\frac{pq}{p-q}}}+\|\nabla \cdot\|_{L^{r}}.$}}  Finally, for $p,q,r$ defined in \eqref{pqr}, we introduce the solution space 
\begin{align*}
&S_T = \{  \{\mathcal{F}_t\}_{t\in[0,T]}\text{-adapted, }L^p(\mathbb R^2) \times (L^{\frac{pq}{p-q}} (\mathbb R^2) \cap \dot{W}^{1,r}(\mathbb R^2))_{\sigma}-\text{valued processes } (n(t),{\bf u}(t))_{t\in [0,T]}
\}, 
\end{align*}
{\text{equipped with the norm}}
\begin{align}\label{norm-S_T}
&\qquad\|(n,{\bf u})\|_{S_T}:=\bigg[\mathbb E\sup_{t\in(0,T)}\Big( \|n(\cdot,t)\|^{\frac{pq}{p-q}}_{L^{p}}
       + \|{\bf u}(\cdot,t)\|^{\frac{pq}{p-q}}_{L^{\frac{pq}{p-q}}}+\|\nabla {\bf u}(\cdot,t)\|^{\frac{pq}{p-q}}_{L^{r}}\Big)\bigg]^{\frac{p-q}{pq}}.
\end{align}

Next, we state the main result of this section.
\begin{theorem}\label{thm3mar8}
  Assume that the initial data $(n_0,{\bf u}_0)$ satisfies \eqref{regularic} and that the noise coefficient satisfied the assumptions (H1)-(H2). Then, there exists an almost surely positive stopping time $\tau$ and {$(n,{\bf u})$} such that $(n,{\bf u},\tau)$ is a solution to the Cauchy
problem (\ref{PKSNS-time-dependent}) in the sense of Definition \ref{def:local_solution}, with $p = 2-\epsilon$, $s=\frac{4}{\epsilon}-2$  and  $r = 2+\epsilon$ (see \eqref{pqr})
for some appropriately small $\epsilon>0$, where $\tau$ is an
almost surely positive $\{\mathcal{F}_t\}_{t\geq0}-$stopping time.

     \end{theorem}

\begin{proof}
Our proof relies on a cut-off argument that has appeared in earlier works such as \cite{zhai20202d, flandoli1995martingale}. By introducing a cut-off function to approximate the system \eqref{PKSNS-time-dependent}, we are able to control the nonlinearities in the system and derive estimates that enable the application of a fixed point theorem.

More precisely, we define $\theta\in C^2([0,\infty),[0,1])$ that satisfies the following conditions:
\begin{enumerate} 
  \item $\theta(r)=1$ for $r\in[0,1]$,
  \item $\theta(r)=0$ for $r>2$,
  \item $\sup_{r\in[0,\infty)}|\theta'(r)|\le C<\infty$.
\end{enumerate}
Next, for any $m \geq 1$, we set $\theta_m(\cdot)=\theta(\cdot/m)$. Using $\theta_m$, we then introduce the following set of equations that approximate the original system \eqref{PKSNS-time-dependent}:
\begin{equation}\label{3.1}
\begin{aligned}
dn + \theta_m\big(\|(n,{\bf u})\|_{X_t}\big)\,{\bf u}\cdot\nabla n\,dt
&=\Delta n\,dt
   - \theta_m\big(\|n\|_{X^t_n}\big) \nabla\cdot\big(n\nabla c\big)\,dt,\\
d{\bf u} + \theta_m\big(\|{\bf u}\|_{X_{\bf u}^t}\big)({\bf u}\cdot\nabla){\bf u}\,dt + \nabla P\,dt
&= \Delta {\bf u}\,dt
   - \theta_m\big(\|n\|_{X_n^t}\big)n\nabla c\,dt + {\bf f}({\bf u})\,dW_t,\\
\nabla\cdot {\bf u} &= 0,\qquad t>0,\; x\in\mathbb R^2.
\end{aligned}
\end{equation}

\smallskip

\smallskip
 Our proof of Theorem \ref{thm3mar8} is then split into the following two steps:\\
{\bf Step 1}: For a fixed $m\geq1$, establish the existence of a {\it global-in-time mild solution} to the approximation system \eqref{3.1}. \\
{\bf Step 2}: Resolve the cut-off approximation by restricting the approximate solution to a local-in-time interval determined by a stopping time. This gives us a local-in-time solution to the original system \eqref{PKSNS-time-dependent}.\\

We proceed with {\bf Step 1}. This will be done by the aid of the Banach fixed point theorem applied in the space ${S}_T$ defined in \eqref{norm-S_T}. 
For that purpose, we define the mapping $\Phi=(\Phi_1,\Phi_2)$ on $S_T$ as follows:
\begin{align}
\Phi_1(n,{\bf u})(t)
&:= e^{t\Delta}n_0
   - \int_0^t e^{(t-s)\Delta}\Big[
       \theta_m\big(\|(n,{\bf u})\|_{X_s}\big){\bf u}\cdot\nabla n\notag\\
&\qquad\qquad\qquad\qquad
     +\;\theta_m\big(\|n\|_{X^s_n}\big)\nabla\cdot(n\nabla c)
     \Big](s)\,ds, \label{eq220}
\\[6pt]
\Phi_2(n,{\bf u})(t)
&:= e^{-tA}{\bf u}_0
   - \int_0^t e^{-(t-s)A}\theta_m\big(\|{\bf u}\|_{X_{\bf u}^s}\big)\,\mathcal P\big(({\bf u}\cdot\nabla){\bf u}\big)(s)\,ds\notag\\
&\qquad +\int_0^t e^{-(t-s)A}\theta_m\big(\|n\|_{X_n^s}\big)\,\mathcal P\big(n\nabla c\big)(s)\,ds
   + \int_0^t e^{-(t-s)A}\mathcal P{\bf f}({\bf u}(s))\,dW_s. \label{eq221}
\end{align}
In what follows, we derive suitable estimates for the terms $\Phi_1$ and $\Phi_2$ and show that they define contractions on the space $S_T$ for some $T>0$.

First, we claim that the following estimates hold true for any $m$ and  $p$, $q$ in  \eqref{pqr}, there exists a constant $C_T>0$ such that
\begin{equation}
\begin{split}\label{bounds_phi}
    &\mathbb{E}\big( \sup_{t\leq T}\|\Phi_1(n,{\bf u})(t)\|_{L^{p}}\big) \lesssim \Vert n_0\Vert_{L^p}+m^2
   \, C_T\\
&\mathbb{E}\big( \sup_{t\leq T}\|\Phi_2(n,{\bf u})\|_{L^\frac{pq}{p-q}}^{\frac{pq}{p-q}} \big)
\lesssim \|{\bf u}_0\|_{L^\frac{pq}{p-q}}^\frac{pq}{p-q} + C_T\Big[m^{\frac{2pq}{p-q}}
+  \mathbb{E}\left(\sup_{t\leq T}\|{\bf u}(t )\|_{L^\frac{pq}{p-q}}^{\frac{pq}{p-q}}\right) +  1\Big].
\end{split}
\end{equation}

We begin with estimates for $\Phi_1$. By using (\ref{eq:3.3})--(\ref{eq:3.4a1}), we obtain for any $t\in(0,T)$ that there exists some positive constant $C_T$ depending only on $T$, such that
\begin{align}\label{phi1estimatemar}
\|\Phi_1(n,{\bf u})(t)\|_{L^{p}}
&\lesssim  \Vert e^{t\Delta }n_0\Vert_{L^{p}}+\int_0^t (t-s)^{-\frac{1}{2}+\frac{1}{p}-\frac{1}{q}}\Big\|\,\theta_m(\|(n,{\bf u})\|_{X_s}) 
     (n\nabla c+{\bf u}n)\Big\|_{L^q}\,ds\nonumber\\
&\lesssim  \Vert n_0\Vert_{L^p}+\int_0^t (t-s)^{-\frac{1}{2}+\frac{1}{p}-\frac{1}{q}} \,\theta_m(\|(n,{\bf u})\|_{X_s})
   (\|n(s)\|_{L^{p}} \|\nabla c(s) + {\bf u}(s)\|_{L^{\frac{pq}{p-q}}})\,ds\nonumber \\[4pt]
&\lesssim  \Vert n_0\Vert_{L^{p}}+
   \int_0^t (t-s)^{-\frac{1}{2}+\frac{1}{p}-\frac{1}{q}}\,
   \,\theta_m(\|(n,{\bf u})\|_{X_s})\| (n,{\bf u}) \|^2_{X_T}\, \,ds \nonumber \\[4pt]
&=\Vert n_0\Vert_{L^p}+\,\theta_m(\|(n,{\bf u})\|_{X_T})
\|(n,{\bf u}) \|^2_{X_T}
   \int_0^t (t-s)^{-\frac{1}{2}+\frac{1}{p}-\frac{1}{q}}\,ds\nonumber \\[4pt]
&\lesssim \Vert n_0\Vert_{L^p}+m^2
   \, C_T,
\end{align}
where $\epsilon\in(0,1)$ is chosen so that $-\frac{1}{2}+\frac{1}{p}-\frac{1}{q}>-1.$

Next, we will find bounds for the the term $\Phi_2$. However, this term is more challenging to treat due to the presence of the stochastic integral in its definition \eqref{eq221}. Therefore, we must first find appropriate bounds for the stochastic integral. 
For that purpose, we first consider the following It\^{o} equation
\begin{align}\label{itoequation}
    d{\bf Z}(t) = -A {\bf Z}(t)\, dt +  \mathcal P{\bf f}({\bf u}(t))\, dW_t, 
\quad {\bf Z}(0) = {\bf 0},
\end{align}
{ where we recall that $A$ is the Stokes operator.}

By applying the It\^{o} product rule to \eqref{itoequation}, we obtain for any $\bar q\geq 1$, that 
{
\begin{align*}
d|{\bf Z}|^{\bar q}(t)=&(-\bar q|{\bf Z}|^{\bar q-2} {\bf Z}\cdot A {\bf Z} +\frac{\bar q(\bar q-2)|{\bf Z}|^{\bar q-4}}{2}| {\bf Z}\cdot \mathcal P{\bf f}({\bf u}(t))|^2+\frac{\bar q}{2}|{\bf Z}|^{\bar q-2}[\mathcal P{\bf f}({\bf u}(t))]^2)dt
+\bar q|{\bf Z}|^{\bar q-2} {\bf Z}\cdot \mathcal P{\bf f}({\bf u}(t))dW_t.
\end{align*}
}
{While we can integrate the equation above to obtain an It\^{o} equation for $L^{\bar q}$-norm of ${\bf Z}$, we should to apply the It\^{o} formula with $\|.\|_{L^2}^{\bar q}$ for the sake of numerology:}
{\begin{align*}
d\Vert {\bf Z}\Vert^{2\bar q}_{L^{2\bar q}}
=&-2\bar q(2\bar q-1)  \Vert |{\bf Z}|^{\bar q-1}\nabla {\bf Z}\Vert_{L^2}^2 \, dt
+ 2\bar q \langle |{\bf Z}|^{2\bar q-2}{\bf Z}, \mathcal P{\bf f}({\bf u}(t)) \, dW_t \rangle\\
&+ \bar q(2\bar q-2) \int_{\mathbb R^2} |{\bf Z}|^{2\bar q-4} \Big( \sum_{k\geq 1} |  {\bf Z}\cdot (\mathcal P{\bf f}({\bf u}(t))e_k)|^2 \Big) \, dx \, dt+ \bar q\int_{\mathbb R^2} |{\bf Z}|^{2\bar q-2} \Big( \sum_{k\geq 1} | \mathcal P{\bf f}({\bf u}(t))e_k|^2 \Big)dxdt.
\end{align*}}
Invoking the H\"older inequality, we obtain, for the last term on the right hand side of the equation above, that
\begin{align}\label{followfrommar5}
\int_{\mathbb R^2} |{\bf Z}|^{2\bar q-2} \Big( \sum_{k\geq 1} | \mathcal P{\bf f}({\bf u}(t))e_k|^2 \Big) dx \le \||{\bf Z}|^{2\bar q-2}\|_{L^{\frac{2\bar q}{2\bar q-2}}} \, \|\Big( \sum_{k\geq 1} | \mathcal P{\bf f}({\bf u}(t))e_k|^2 \Big)\|_{L^{\bar q}}.
\end{align}
 {Similarly, we find
\begin{align}\label{3p8aapr}
\int_{\mathbb R^2} |{\bf Z}|^{2\bar q-4} \Big( \sum_{k\geq 1} | {\bf Z}\cdot (\mathcal P{\bf f}({\bf u}(t))e_k)|^2 \Big) \, dx \, dt\leq   \||{\bf Z}|^{2\bar q-2}\|_{L^{\frac{2\bar q}{2\bar q-2}}} \, \|\Big( \sum_{k\geq 1} | \mathcal P{\bf f}({\bf u}(t))e_k|^2 \Big)\|_{L^{\bar q}}.
\end{align}}
 Due to the growth assumption (H1) on the noise coefficient, we know that $\|\big(\sum_{k\geq 1}|\mathcal P{\bf f}({\bf u}(t))e_k|^2\big)\|_{L^{\bar q}} \lesssim 1 + \|{\bf u}\|^2_{L^{2\bar q}}$.  Then it follows from (\ref{followfrommar5}) that, for some constant $C_{\varepsilon}>0$ depending only on a parameter $\varepsilon>0$ which will be chosen appropriately, we have
\begin{align}\label{traceterm1feb}
 \int_0^t\int_{\mathbb R^2} |{\bf Z}|^{2\bar q-2} \Big( \sum_{k\geq 1} | \mathcal P{\bf f}({\bf u}(s))e_k|^2 \Big) dxds\lesssim &\int_0^t\|{\bf Z}\|_{L^{2\bar q}}^{2\bar q-2} \, \|\Big( \sum_k | \mathcal P{\bf f}({\bf u}(s))e_k|^2 \Big)\|_{L^{\bar q}}ds\nonumber\\
\le &\sup_{s\leq t} \|{\bf Z}\|_{L^{2\bar q}}^{2\bar q-2} \int_0^t\, \|\Big( \sum_{k\geq 1} | \mathcal P{\bf f}({\bf u}(s))e_k|^2 \Big)\|_{L^{\bar q}}ds\nonumber\\
\le &\sup_{s\leq t} \|{\bf Z}\|_{L^{2\bar q}}^{2\bar q-2} \int_0^t\, [1+\Vert {\bf u}\Vert_{L^{2\bar q}}^2(s)]ds\nonumber\\
\le &\varepsilon  \sup_{s\leq t} \|{\bf Z}\|_{L^{2\bar q}}^{2\bar q} + C_\varepsilon \big(\int_0^t   (1 + \|{\bf u}\|_{L^{2\bar q}}^2)ds\big)^{\bar q},
\end{align} 
where in the last step, we used Young's inequality.   {Similarly, we obtain from (\ref{3p8aapr}) that
\begin{align}\label{aprnewcorrection}
\int_0^t\int_{\mathbb R^2} |{\bf Z}|^{2\bar q-4} \Big( \sum_{k\geq 1} | {\bf Z}\cdot (\mathcal P{\bf f}({\bf u}(s))e_k)|^2 \Big) \, dx \, ds\lesssim \varepsilon  \sup_{s\leq t} \|{\bf Z}\|_{L^{2\bar q}}^{2\bar q} + C_\varepsilon \big(\int_0^t   (1 + \|{\bf u}\|_{L^{2\bar q}}^2)ds\big)^{\bar q}.
\end{align}}
Next, for the quadratic variation term, we apply the  Burkholder-Davis-Gundy (BDG) inequality, which yields
\[
\mathbb{E}\Big[\sup_{t\le T} \Big| \int_0^t   \langle |{\bf Z}|^{2\bar q-2}Z,  \mathcal P{\bf f}({\bf u}(s))\, dW_s \rangle \Big| \Big] 
\lesssim \mathbb{E} \Big[ \Big( \int_0^T   \sum_{k\geq 1} |\langle |{\bf Z}|^{2\bar q-2}{\bf Z}, \mathcal P{\bf f}({\bf u}(s)) e_k \rangle|^2 ds \Big)^{1/2} \Big].
\]  
We utilize  H\"older's inequality and Young's inequality to obtain,
\begin{align}\label{youngbefore1feb}
&\mathbb{E}\Big[\sup_{t\le T} \Big| \int_0^t  \langle |{\bf Z}|^{2\bar q-2}{\bf Z},  \mathcal P{\bf f}({\bf u}(s))\, dW_s \rangle \Big|\Big]\nonumber\\
\le &\varepsilon \, \mathbb{E} \sup_{s\le T}  \|{\bf Z}(s)\|_{L^{2\bar q}}^{2\bar q} + C_\varepsilon \, \mathbb{E} \bigg(\int_0^T  (1 + \|{\bf u}(s)\|_{L^{2\bar q}}^{\bar q}) ds\bigg)^2,
\end{align}
where, again $\varepsilon>0$ is appropriately small and $C_{\varepsilon}$ is a positive constant depending on $\varepsilon.$  Hence, upon combining (\ref{traceterm1feb}) and (\ref{aprnewcorrection}) with (\ref{youngbefore1feb}), we arrive at
\begin{align}\label{110newlabel}
\mathbb{E}\left( \sup_{t\leq T} \|{\bf Z}(t)\|^{2\bar q}_{L^{2\bar q}} \right)&+2\bar q(\bar q-1)  \mathbb E\int_0^T \Vert |{\bf Z}|^{\bar q-1} \nabla {\bf Z} \Vert^2_{L^2}ds\nonumber\\
\lesssim&  \bigg[\, \mathbb{E} \bigg(\int_0^t  (1 + \|{\bf u}(s)\|_{L^{2\bar q}}^{\bar q}) ds\bigg)^2+\mathbb E\bigg(\int_0^t  (1 + \|{\bf u}(s)\|_{L^{2\bar q}}^2) ds\bigg)^{\bar q}\bigg].
\end{align}
Finally, by using H\"{o}lder's inequality, this gives us, for some constant depending only on $T$, that
\begin{align}\label{110newlabelnewmar}
\mathbb{E}\left( \sup_{t\leq T} \|{\bf Z}(t)\|^{2\bar q}_{L^{2\bar q}} \right)&+2\bar q(\bar q-1)  \mathbb E\int_0^T \Vert |{\bf Z}|^{\bar q-1} \nabla {\bf Z} \Vert^2_{L^2}ds\lesssim C_T[1+\mathbb E(\int_0^T\sup_{s\leq t}\Vert {\bf u}(s)\Vert_{2\bar q}^{2\bar q}ds)].
\end{align}

With this estimate \eqref{110newlabelnewmar} at hand, we now turn to the term $\Phi_2$. We obtain the following bounds for any $t\in(0,T),$ using the semigroup estimates stated in Lemma \ref{lemma:semigroup-Lp-Lq}:  
\begin{align}\label{labelphi2mar}
\|\Phi_2&(n,{\bf u})(t)\|_{L^{\frac{pq}{p-q}}}\nonumber\\
\lesssim& \Vert e^{-tA}{\bf u}_0\Vert_{L^{\frac{pq}{p-q}}}+ \int_0^t (t-s)^{-\frac{1}{2}+\frac{1}{p}-\frac{1}{q}} \|  \mathcal{P} \nabla \cdot (\theta_m\big(\|{\bf n}\|_{X_{n}^s}\big) \nabla c(s) \otimes \nabla c(s) +\theta_m\big(\|{\bf u}\|_{X_{\bf u}^s}\big) {\bf u}(s) \otimes {\bf u}(s) ) \|_{L^{\frac{pq}{2(p-q)}}} \, ds \nonumber\\
&+\Vert\int_0^t e^{-(t-s)A}\mathcal  P{\bf f}({\bf u}(s))\,dW_s\Vert_{L^\frac{pq}{p-q}} \nonumber \\
\lesssim& \Vert {\bf u}_0\Vert_{L^\frac{pq}{p-q}}+ \int_0^t (t-s)^{-\frac{1}{2}+\frac{1}{p}-\frac{1}{q}} \big( \theta_m\big(\|{n}\|_{X_{n}^s}\big)\| \nabla c(s) \|_{L^{\frac{pq}{p-q}}} \| \nabla c(s) \|_{L^\frac{pq}{p-q}} +\theta_m\big(\|{\bf u}\|_{X_{\bf u}^s} \big)\| {\bf u}(s) \|_{L^{\frac{pq}{p-q}}} \|{\bf u}(s) \|_{L^{\frac{pq}{p-q}}} \big) \, ds\nonumber\\
&+\Vert\int_0^t e^{-(t-s)A}\mathcal  P{\bf f}({\bf u}(s))\,dW_s\Vert_{L^\frac{pq}{p-q}}\nonumber\\
\lesssim & \Vert {\bf u}_0\Vert_{L^\frac{pq}{p-q}}+ \int_0^t (t-s)^{-\frac{1}{2}+\frac{1}{p}-\frac{1}{q}} \big( \theta_m\big(\|{n}\|_{X_{n}^s}\big)\| n(s) \|^2_{\frac{pq}{p-q}} + \theta_m\big(\|{\bf u}\|_{X_{\bf u}^s}\big) \| {\bf u}(s) \|^2_\frac{pq}{p-q} \big) \, ds\nonumber\\
&+\Vert\int_0^t e^{-(t-s)A}\mathcal   P{\bf f}({\bf u}(s))\,dW_s\Vert_{L^\frac{pq}{p-q}} \nonumber\\
\lesssim &\Vert {\bf u}_0\Vert_{L^\frac{pq}{p-q}}+ \int_0^t (t-s)^{-\frac{1}{2}+\frac{1}{p}-\frac{1}{q}}  \, ds \,  \| (n,{\bf u}) \|^2_{X_T}\theta_m\big(\|(n,{\bf u})\|_{X^s}\big)+\Vert\int_0^t e^{-(t-s)A}\mathcal  P{\bf f}({\bf u}(s))\,dW_s\Vert_{L^\frac{pq}{p-q}} \nonumber\\
\lesssim &\Vert {\bf u}_0\Vert_{L^\frac{pq}{p-q}}+C_Tm^2+\Vert\int_0^t e^{-(t-s)A}\mathcal  P{\bf f}({\bf u}(s))\,dW_s\Vert_{L^\frac{pq}{p-q}}\nonumber\\
= &\Vert {\bf u}_0\Vert_{L^\frac{pq}{p-q}}+C_Tm^2+\Vert {\bf Z}\Vert_{L^\frac{pq}{p-q}},
\end{align}
where $\epsilon\in(0,1)$ is chosen so that $-\frac{1}{2}+\frac{1}{p}-\frac{1}{q}>-1.$  Next, we choose $\bar q=\frac{pq}{2(p-q)}$ in (\ref{110newlabelnewmar}) and combine it with (\ref{labelphi2mar}) to obtain
\[
\mathbb{E}\big( \sup_{t\leq T}\|\Phi_2(n,{\bf u})\|_{L^\frac{pq}{p-q}}^{\frac{pq}{p-q}} \big)
\lesssim \|{\bf u}_0\|_{L^\frac{pq}{p-q}}^\frac{pq}{p-q} + C_Tm^{\frac{2pq}{p-q}}
+  C_T[1+\mathbb{E}(\sup_{t\leq T}\|{\bf u}(t )\|_{L^\frac{pq}{p-q}}^{\frac{pq}{p-q}})].
\]
{  We have thus derived the desired $L^p$-estimates \eqref{bounds_phi} for $\Phi_1,\Phi_2$ defined in \eqref{eq220} and \eqref{eq221}. 

Owing to the fact that the space $S_T$ defined in \eqref{norm-S_T}, involves $L^r$-norm of the gradient of the fluid velocity, we will next prove for $r=2+\epsilon$ that,} 
\begin{equation}
    \begin{split}\label{bounds_nablaphi}
\mathbb{E}\left(\sup_{t\leq T}\Vert \nabla \Phi_2(n,{\bf u})\Vert^{\frac{pq}{p-q}}_{L^r}\right)\lesssim   \Vert \nabla { \bf u}_0\Vert^{{\frac{pq}{p-q}}}_{L^r}+C_Tm^{{\frac{2pq}{p-q}}}+C_T[1+\mathbb E(\sup_{t\leq T}\Vert \nabla {\bf u}(t)\Vert_{r}^{{\frac{pq}{p-q}}})].
    \end{split}
\end{equation}

 To that end, we apply the semigroup estimates stated in Lemma \ref{lemma:semigroup-Lp-Lq} and obtain
\begin{align}\label{labelphi2apr}
&\|\nabla \Phi_2(n,{\bf u})\|_{L^{r}}\lesssim\Vert e^{-tA}\nabla {\bf u}_0\Vert_{L^{r}}\nonumber\\
+& \int_0^t \| e^{-(t-s)A} \mathcal{P} \nabla   (n(s)  \nabla c(s) + (u\cdot \nabla) u(s) ) \|_{L^r} \, ds+\Vert\int_0^t e^{-(t-s)A}\mathcal  P\nabla  {\bf f}(u(s))\,dW_s\Vert_{L^p} \nonumber \\
&\lesssim \Vert \nabla {\bf u}_0\Vert_{L^r}+ \int_0^t (t-s)^{-{\frac{1}{2}
}+\frac{1}{r}-\frac{1}{q}} \big( \| n (s) \|_{L^p} \| \nabla c(s) \|_{L^{\frac{pq}{p-q}}})ds\nonumber\\
&+\int_0^t(t-s)^{-{\frac{1}{2}
}+\frac{1}{r}-\frac{pq+rp-rq}{rpq}} \|\nabla {\bf u}(s) \|_{L^r} \| {\bf u}(s) \|_{L^{\frac{pq}{p-q}}} \big) \, ds+\Vert\int_0^t e^{-(t-s)A}\mathcal  P\nabla {\bf f}({\bf u}(s))\,dW_s\Vert_{L^r}\nonumber\\
&\lesssim  \Vert \nabla {\bf u}_0\Vert_{L^r}+ \int_0^t (t-s)^{-{\frac{1}{2}}+\frac{1}{r}-\frac{1}{q}} \big( \| n(s) \|^2_{p})ds\nonumber\\
&+\int_0^t (t-s)^{-{\frac{1}{2}}+\frac{1}{r}-\frac{pq+rp-rq}{rpq}}  \| {\bf u}(s)\|_{\frac{pq}{p-q}} \| \nabla {\bf u}(s) \|_r \, ds+\Vert\int_0^t e^{-(t-s)A}\mathcal  P\nabla {\bf f}({\bf u}(s))\,dW_s\Vert_{L^r} \nonumber\\
&\lesssim \Vert\nabla {\bf u}_0\Vert_{L^r}+ \bigg(\int_0^t (t-s)^{-{\frac{1}{2}}+\frac{1}{r}-\frac{1}{q}}  \, ds +\int_0^t (t-s)^{-{\frac{1}{2}}+\frac{1}{r}-\frac{pq+rp-rq}{rpq}}  \, ds\bigg)\,  \| (n,{\bf u}) \|^2_{X_T}\nonumber\\
&+\Vert\int_0^t e^{-(t-s)A}\mathcal  P\nabla {\bf f}({\bf u}(s))\,dW_s\Vert_{L^r} \nonumber\\
&\lesssim \Vert \nabla  {\bf u}_0\Vert_{L^r}+C(T)m^2+\Vert\int_0^t e^{-(t-s)A}\nabla \mathcal  P{\bf f}({\bf u}(s))\,dW_s\Vert_{L^r},
\end{align}
where $\epsilon\in(0,\frac{\sqrt{17}-3}{2})$ is chosen so that $-\frac{1}{2}+\frac{1}{r}-\frac{1}{q}>-1$ and $-\frac{1}{2}+\frac{1}{r}-\frac{pq+rp-rq}{rpq}=-\frac{1}{2}+\frac{1}{p}-\frac{1}{q}>-1$.
As earlier (see estimate \eqref{110newlabelnewmar}), estimating the stochastic integral $\Vert\int_0^t e^{-(t-s)A}A^{\frac{1}{2}}\mathcal  P{\bf f}({\bf u}(s))\,dW_s\Vert_{L^r}$ poses several challenges. Therefore, we consider the following It\^o equation for any $\alpha\in(0,1],$
$$d{\bf Z}_1=-A {\bf Z}_1dt+A^{\alpha} \mathcal P{\bf f}({\bf u}(t))dW_t,\qquad ~{\bf Z}_1(0)={\bf 0}.$$
Applying the It\^{o} formula to the equation above with $\|.\|_{L^2}^{2\bar p}$ for any $\bar p\geq 1$, we obtain  
\begin{align}\label{ii1234}
d\Vert {\bf Z}_1^{\bar q}\Vert_{L^2}^{2\bar p} &=\bar p \Vert |{\bf Z}_1|^{\bar q}\Vert_{L^2}^{2(\bar p-1)}\bigg[-2\bar q(2\bar q-1)\Vert |{\bf Z}_1|^{\bar q-1}\nabla {\bf Z}_1\Vert_{L^2}^2dt+2\bar q\langle |{\bf Z}_1|^{2\bar q-2}{\bf Z}_1,\mathcal PA^{\alpha}{\bf f}({\bf u}(t))dW_t\rangle\nonumber\\
&\quad+\bar q(2\bar q-2)\int_{\mathbb R^2} |{\bf Z}_1|^{2\bar q-4}(\sum_{k\geq 1 }|{\bf Z}_1\cdot(\mathcal P A^{\alpha}{\bf f} e_k)|^2)dxdt+\bar q\int_{\mathbb R^2} |{\bf Z}_1|^{2\bar q-2} \Big( \sum_{k\geq 1} | \mathcal PA^{\alpha}{\bf f}({\bf u}(t))e_k|^2 \Big)dxdt\bigg]\nonumber\\
&\quad+\frac{4\bar q^2}{2}\bar p(\bar p-1)\Vert |{\bf Z}_1|^{\bar q}\Vert_{L^2}^{2(\bar p-2)}\langle |{\bf Z}_1|^{2\bar q-2}{\bf Z}_1,\mathcal PA^{\alpha }{\bf f}\rangle^2dt\nonumber\\
&=-2\bar p\bar q(2\bar q-1)\Vert |{\bf Z}_1|^{\bar q}\Vert_{L^2}^{2(\bar p-1)}\Vert |{\bf Z}_1|^{\bar q-1}\nabla {\bf Z}_1\Vert_{L^2}^2dt+2\bar p\bar q\Vert |{\bf Z}_1|^{\bar q}\Vert_{L^2}^{2(\bar p-1)} \langle |{\bf Z}_1|^{2\bar q-2}{\bf Z}_1,\mathcal P A^{\alpha} {\bf f}dW_t\rangle\nonumber\\
&\quad+2\bar p\bar q(\bar q-1)\Vert |{\bf Z}_1|^{\bar q}\Vert_{L^2}^{2(\bar p-1)}\int_{\mathbb R^2}|{\bf Z}_1|^{2\bar q-4}(\sum_{k\geq 1}|{\bf Z}_1\cdot\mathcal P A^{\alpha}{\bf f}e_k|^2)dx dt\nonumber\\
&\quad+\bar p\bar q\Vert |{\bf Z}_1|^{\bar q}\Vert_{L^2}^{2(\bar p-1)}\int_{\mathbb R^2} |{\bf Z}_1|^{2\bar q-2} \Big( \sum_{k\geq 1} | \mathcal PA^{\alpha}{\bf f}({\bf u}(t))e_k|^2 \Big)dxdt\nonumber\\
&\quad +\frac{4\bar p{\bar q}^2(\bar p-1)}{2}\Vert |{\bf Z}_1|^{\bar q}\Vert_{L^2}^{2(\bar p-2)}\langle |{\bf Z}_1|^{2\bar q-2}{\bf Z}_1,\mathcal P A^{\alpha} {\bf f}\rangle^2dt\nonumber\\
&:=II_1+II_2+II_3+II_4+II_5.
\end{align}
We will now move $II_1$ to the left hand side and find estimates for each term $II_{i};\, i=2,..,5$ on the right hand side. 
We first estimate $II_2$ by applying the Burkholder-Davis-Gundy (BDG) inequality. This gives us,
\begin{align}\label{estimateii2_before}
&\mathbb{E}\Big[\sup_{s\le T} \Big| \int_0^t  \Vert |{\bf Z}_1|^{\bar q}\Vert_{L^2}^{2(\bar p-1)}  \langle |{\bf Z}_1|^{2\bar q-2}{\bf Z}_1,  \mathcal PA^{\alpha}{\bf f}({\bf u}(s))\, dW_s \rangle \Big| \Big] \nonumber \\
\lesssim &\mathbb{E} \Big[ \Big( \int_0^T \Vert |{\bf Z}_1|^{\bar q}\Vert_{L^2}^{4(\bar p-1)}   \sum_{k\geq 1} |\langle |{\bf Z}_1|^{2\bar q-2}{\bf Z}_1, \mathcal PA^{\alpha}{\bf f}({\bf u}(s)) e_k \rangle|^2 ds \Big)^{1/2} \Big]\nonumber\\
\lesssim &\mathbb{E} \Big[\sup_{t\leq T} \Vert |{\bf Z}_1|^{\bar q}\Vert_{L^2}^{2(\bar p-1)}  \Big( \int_0^T     \sum_{k\geq 1} |\langle |{\bf Z}_1|^{2\bar q-2}{\bf Z}_1, \mathcal PA^{\alpha}{\bf f}({\bf u}(s)) e_k \rangle|^2 ds \Big)^{1/2} \Big].
\end{align}

We further apply H\"older's inequality to get
\begin{align}\label{3p21}
\Big( \int_0^T     \sum_{k\geq 1} |\langle |{\bf Z}_1|^{2\bar q-2}{\bf Z}_1, \mathcal PA^{\alpha}{\bf f}({\bf u}(s)) e_k \rangle|^2 ds \Big)^{1/2} \lesssim\bigg(\int_0^T \Vert {\bf Z}_1\Vert^{2(2\bar q-1)}_{L^{2\bar q}}(1+\Vert A^{\alpha} {\bf u}(s)\Vert^2_{L^{2\bar q}})ds\bigg)^{1/2},
\end{align}
where we additionally invoked the growth assumption (H1).    Moreover, with (\ref{estimateii2_before}) and (\ref{3p21}), we utilize Young's inequality to conclude

\begin{align}\label{estimateii2}
&\mathbb{E}\Big[\sup_{s\le T} \Big| \int_0^t  \Vert |{\bf Z}_1|^{\bar q}\Vert_{L^2}^{2(\bar p-1)}  \langle |{\bf Z}_1|^{2\bar q-2}{\bf Z}_1,  \mathcal PA^{\alpha}{\bf f}({\bf u}(s))\, dW_s \rangle \Big| \Big]\nonumber\\ \lesssim &
\mathbb{E}\Big[\sup_{t\leq T} \Vert |{\bf Z}_1|^{\bar q}\Vert_{L^2}^{2(\bar p-1)} \bigg(\int_0^T \Vert {\bf Z}_1\Vert^{2(2\bar q-1)}_{L^{2\bar q}}(1+\Vert A^{\alpha} {\bf u}(s)\Vert^2_{L^{2\bar q}})ds\bigg)^{1/2} \Big] \nonumber\\
\lesssim&\mathbb{E}\Big[\sup_{t\leq T} \Vert {\bf Z}_1\Vert_{L^{2\bar q}}^{2\bar p\bar q-1} \bigg(\int_0^T  (1+\Vert A^{\alpha} {\bf u}(s)\Vert^2_{L^{2\bar q}})ds\bigg)^{1/2} \Big] \nonumber\\
\le& \varepsilon \, \mathbb{E} \sup_{s\le T}  \|{\bf Z}_1(s)\|_{L^{2\bar q}}^{2\bar p\bar q} + C_\varepsilon \, \mathbb{E} \bigg(\int_0^T  (1 + \|A^{\alpha}{\bf u}(s)\|_{L^{2\bar q}}^{2}) ds\bigg)^{\bar p\bar q},
\end{align}

Next, for $II_3$, we obtain
we have, due to H\"{o}lder's inequality and assumption (H1), that
\begin{align}\label{estimateii3}
\sup_{s\leq t}\Vert |{\bf Z}_1|^{\bar q}\Vert_{L^2}^{2(\bar p-1)}&\int_0^t\int_{\mathbb R^2}|{\bf Z}_1|^{2\bar q-4}(\sum_{k}|{\bf Z}_1\cdot\mathcal P A^{\alpha}{\bf f}e_k|^2)dx ds\nonumber\\
\le &\sup_{s\leq t}\Vert |{\bf Z}_1|^{\bar q}\Vert_{L^2}^{2(\bar p-1)}\int_0^t \Vert {\bf Z}_1\Vert_{L^{2\bar q}}^{2\bar q-2}\Vert \sum_{k}|\mathcal PA^{\alpha}{\bf f}e_k|^2\Vert_{L^{2\bar q}}ds\nonumber\\
\leq&\sup_{s\leq t}\Vert {\bf Z}_1\Vert_{L^{2\bar q}}^{2\bar p\bar q-2}\int_0^t(1+\Vert A^{\alpha}{\bf u}\Vert^2_{L^{2\bar q}})ds\nonumber\\
\leq&\varepsilon (\sup_{s\leq t} \Vert {\bf Z}_1\Vert_{L^{2\bar q}}^{2\bar p\bar q})+C_{\varepsilon}\big(\int_0^t1+\Vert A^{\alpha}{\bf u}\Vert^2_{L^{2\bar q}} ds\big)^{\bar p\bar q}. 
\end{align}
Similarly, for $II_4$ we obtain
\begin{align}\label{estimateii4}
\sup_{s\leq t}\Vert |{\bf Z}_1|^{\bar q}\Vert_{L^2}^{2(\bar p-1)}\int_0^t\int_{\mathbb R^2} |{\bf Z}_1|^{2\bar q-2} \Big( \sum_{k\geq 1} | \mathcal P{\bf f}({\bf u}(s))e_k|^2 \Big)dx ds
\leq&\varepsilon (\sup_{s\leq t} \Vert {\bf Z}_1\Vert_{L^{2\bar q}}^{2\bar p\bar q})+C_{\varepsilon}\big(\int_0^t1+\Vert A^{\alpha}{\bf u}\Vert^2_{L^{2\bar q}} ds\big)^{\bar p\bar q}. 
\end{align}
Finally, for $II_5$, we observe, using the H\"older inequality, that
\begin{align}\label{estimateii4}
\int_0^t\Vert |{\bf Z}_1|^{\bar q}\Vert_{L^2}^{2(\bar p-2)}\langle |{\bf Z}_1|^{2\bar q-2}{\bf Z}_1,\mathcal P A^{\alpha} {\bf f}&\rangle^2ds\leq \sup_{s\leq t}\Vert |{\bf Z}_1|^{\bar q}\Vert_{L^2}^{2(\bar p-2)}\int_0^t\langle |{\bf Z}_1|^{2\bar q-2}{\bf Z}_1,\mathcal P A^{\alpha} {\bf f}\rangle^2ds\nonumber\\
\leq &\sup_{s\leq t}\Vert |{\bf Z}_1|^{\bar q}\Vert_{L^2}^{2(\bar p-2)}\int_0^t\big(\Vert \sum_{k\geq 1}|\mathcal PA^{\alpha} {\bf f} e_k|^2\Vert_{2\bar q} \Vert |{\bf Z}_1|^{2\bar q-1}\Vert_{L^{\frac{2\bar q}{2\bar q-1}}}^2\big)ds \nonumber\\
\leq &\sup_{s\leq t}\bigg(\Vert {\bf Z}_1\Vert_{L^{2\bar q}}^{2(2\bar q-1)}\Vert {\bf Z}_1\Vert_{L^{2\bar q}}^{2\bar q(\bar p-2)}\bigg)\int_0^t(1+\Vert A^{\alpha}{\bf u}\Vert^2_{L^{2\bar q}})ds\nonumber\\
\leq &\varepsilon (\sup_{s\leq t}\Vert {\bf Z}_1\Vert_{L^{2\bar q}}^{2\bar p\bar q})+ C_{\varepsilon}\bigg(\int_0^t(1+\Vert A^{\alpha}{\bf u}\Vert^2_{L^{2\bar q}})ds\bigg)^{\bar p\bar q}.
\end{align}
We collect (\ref{ii1234})--(\ref{estimateii4}) to conclude that
\begin{align}\label{110newlabelnewdec}
&\mathbb{E}\left( \sup_{t\leq T} \|{\bf Z}_1(t)\|^{2\bar p\bar q}_{L^{2\bar q}} \right)+2\bar p\bar q(2\bar q-1)\int_0^T\Vert |{\bf Z}_1|^{\bar q}\Vert_{L^2}^{2(\bar p-1)}\Vert |{\bf Z}_1|^{\bar q-1}\nabla {\bf Z}_1\Vert_{L^2}^2ds\nonumber\\
\lesssim&  \bigg[\, \mathbb{E} \bigg(\int_0^T  (1 + \|A^{\alpha}{\bf u}(s)\|_{L^{2\bar q}}^{2}) ds\bigg)^{\bar p\bar q}\bigg].
\end{align}

Thus, by choosing $\bar q=\frac{r}{2}$ and $2\bar p\bar q=\frac{pq}{p-q}$ in (\ref{110newlabelnewdec}), we obtain form (\ref{labelphi2apr}) the desired estimate as claimed in \eqref{bounds_nablaphi}:
\begin{align*}
\mathbb{E}\sup_{t\leq T}\Vert A^{\alpha}\Phi_2(n,{\bf u})\Vert^{\frac{pq}{p-q}}_{L^r}\lesssim  &\Vert A^{\alpha} {\bf u}_0\Vert^{{\frac{pq}{p-q}}}_{L^r}+C_Tm^{\frac{2pq}{p-q}}+\mathbb{E}\Vert\int_0^T e^{-(t-s)A}\mathcal  PA^{\alpha}{\bf f}({\bf u}(s))\,dW_s\Vert^{\frac{pq}{p-q}}_{L^r}\\
\lesssim & \Vert A^{\alpha} {\bf u}_0\Vert^{\frac{pq}{p-q}}_{L^r}+C_Tm^{{\frac{2pq}{p-q}}}+\mathbb{E}\Vert\int_0^T e^{-(t-s)A}\mathcal  PA^{\alpha}{\bf f}({\bf u}(s))\,dW_s\Vert^{^{\frac{pq}{p-q}}}_{L^r}\\
\leq & \Vert A^{\alpha}{ \bf u}_0\Vert^{{\frac{pq}{p-q}}}_{L^r}+C_Tm^{{\frac{2pq}{p-q}}}+C_T[1+\mathbb E(\sup_{t\leq T}\Vert A^{\alpha}{\bf u}(t)\Vert_{r}^{{\frac{pq}{p-q}}})].
\end{align*}

{{We next show that $\Phi_1\in C([0,T]; L^p(\mathbb R^2))$ and $\Phi_2\in  C([0,T]; L^{\frac{pq}{p-q}}(\mathbb R^2))$ $\mathbb P$-a.s., where $p$ and $q$ are given in (\ref{pqr}).  In light of (\ref{eq220}), we estimate
\begin{align}\nonumber
\Vert \Phi_1(t+h)&-\Phi_1(t)\Vert_{L^p}\leq \Vert e^{(t+h)\Delta}n_0-e^{t\Delta }n_0\Vert_{L^p}\\
&+ \Big\Vert \int_0^{t+h} e^{(t+h-s)\Delta}[\theta_m\big(\|(n,\mathbf u)\|_{X_s}\big)\,\mathbf u\cdot\nabla n+
\theta_m\big(\|n\|_{X_n^s}\big)\,\nabla\!\cdot(n\nabla c)]\,ds\nonumber\\
&-
\int_0^t e^{(t-s)\Delta}[\theta_m\big(\|(n,\mathbf u)\|_{X_s}\big)\,\mathbf u\cdot\nabla n
+
\theta_m\big(\|n\|_{X_n^s}\big)\,\nabla\!\cdot(n\nabla c)]\,ds\Big\Vert_{L^p}\nonumber\\
\leq & \Vert e^{(t+h)\Delta}n_0-e^{t\Delta }n_0\Vert_{L^p}\nonumber\\
&+\Vert \int_0^t
\big(e^{h\Delta}-I\big)e^{(t-s)\Delta}[\theta_m\big(\|(n,\mathbf u)\|_{X_s}\big)\,\mathbf u\cdot\nabla n+
\theta_m\big(\|n\|_{X_n^s}\big)\,\nabla\!\cdot(n\nabla c)]\,ds\Vert_{L^p}\nonumber\\
&+
\Vert \int_t^{t+h} e^{(t+h-s)\Delta}[\theta_m\big(\|(n,\mathbf u)\|_{X_s}\big)\,\mathbf u\cdot\nabla n+
\theta_m\big(\|n\|_{X_n^s}\big)\,\nabla\!\cdot(n\nabla c)]\,ds\Vert_{L^p}\nonumber\\
:=& III_1+III_2+III_3.
\end{align}
In terms of $III_1$, since $(e^{t\Delta})_{t\ge0}$ is a strongly continuous semigroup on $L^p(\mathbb{R}^2)$, it follows that, as $h\rightarrow 0,$
\begin{align}\label{combine_PHI1_continuity_1}
 \Vert e^{(t+h)\Delta}n_0-e^{t\Delta }n_0\Vert_{L^p}\rightarrow 0.
 \end{align}
For $III_2$, by using Lemma \ref{lemma:semigroup-Lp-Lq}, one finds
\begin{align*}
 &\int_0^t\Vert
e^{(t-s)\Delta}[\theta_m\big(\|(n,\mathbf u)\|_{X_s}\big)\,\mathbf u\cdot\nabla n+
\theta_m\big(\|n\|_{X_n^s}\big)\,\nabla\!\cdot(n\nabla c)]\Vert_{L^p}\,ds\\
\leq &\int_0^t (t-s)^{-\frac{1}{2}+\frac{1}{p}-\frac{1}{q}} \,\theta_m(\|(n,{\bf u})\|_{X_s})
   (\|n(s)\|_{L^{p}} \|\nabla c(s) + {\bf u}(s)\|_{L^{\frac{pq}{p-q}}})ds\leq m^2C_T.
   \end{align*}
Then we apply Lebesgue dominated theorem to obtain that as $h\rightarrow 0,$
\begin{align}\label{combine_PHI1_continuity_2}
  \Big \Vert \int_0^t
\big(e^{h\Delta}-I\big)e^{(t-s)\Delta}[\theta_m\big(\|(n,\mathbf u)\|_{X_s}\big)\,\mathbf u\cdot\nabla n+
\theta_m\big(\|n\|_{X_n^s}\big)\,\nabla\!\cdot(n\nabla c)]\,ds\Big\Vert_{L^p}\rightarrow 0. 
\end{align} 
In terms of $III_3,$ we use Lemma \ref{lemma:semigroup-Lp-Lq} to get as $h\rightarrow0,$
\begin{align}\label{combine_PHI1_continuity_3}
&\Vert \int_t^{t+h} e^{(t+h-s)\Delta}[\theta_m\big(\|(n,\mathbf u)\|_{X_s}\big)\,\mathbf u\cdot\nabla n+
\theta_m\big(\|n\|_{X_n^s}\big)\,\nabla\!\cdot(n\nabla c)]\,ds\Vert_{L^p}\nonumber\\
 \leq &\int_t^{t+h} (t+h-s)^{-\frac{1}{2}+\frac{1}{p}-\frac{1}{q}} \,\theta_m(\|(n,{\bf u})\|_{X_s})
   (\|n(s)\|_{L^{p}} \|\nabla c(s) + {\bf u}(s)\|_{L^{\frac{pq}{p-q}}})ds\nonumber\\
   \leq&m^2\int_t^{t+h}(t+h-s)^{-\frac{1}{2}+\frac{1}{p}-\frac{1}{q}}ds\rightarrow 0.
\end{align}
Combining (\ref{combine_PHI1_continuity_1}) and (\ref{combine_PHI1_continuity_2}) with (\ref{combine_PHI1_continuity_3}), we obtain $\Phi_1\in C([0,T];L^p(\mathbb R^2))$ $\mathbb P$-a.s.  We next discuss the continuity of $\Phi_2$ in time. 

Similarly, we consider for $p$ and $q$ given in (\ref{pqr}),
\begin{align}\label{Phi2_estimate}
\Vert \Phi_2(t+h)&-\Phi_2(t)\Vert_{L^{\frac{pq}{p-q}}} \leq \Vert e^{(t+h)\Delta}{\bf u}_0-e^{t\Delta }{\bf u}_0\Vert_{L^{\frac{pq}{p-q}}}\nonumber\\
 &+\Vert \int_0^t
\big(e^{-(t+h-s)A}-e^{-(t-s)A}\big)[\theta_m\big(\|{\bf u}\|_{X_{\bf u}^s}\big)\,\mathcal P\big(({\bf u}\cdot\nabla){\bf u}\big)+\theta_m\big(\|n\|_{X_n^s}\big)\,\mathcal P\big(n\nabla c\big)(s) ]\,ds\Vert_{L^{\frac{pq}{p-q}}}\nonumber\\
&+
\Vert \int_t^{t+h} e^{-(t+h-s)A}[\theta_m\big(\|{\bf u}\|_{X_{\bf u}^s}\big)\,\mathcal P\big(({\bf u}\cdot\nabla){\bf u}\big)+\theta_m\big(\|n\|_{X_n^s}\big)\,\mathcal P\big(n\nabla c\big)(s) ]\,ds\Vert_{L^{\frac{pq}{p-q}}}\nonumber\\
&+\Vert\int_0^{t+h}e^{-(t+h-s)A}\mathcal P{\bf f}({\bf u}(s))dW_s-\int_0^{t}e^{-(t-s)A}\mathcal P{\bf f}({\bf u}(s))dW_s\Vert_{L^\frac{pq}{p-q}}\nonumber\\
&:=IV_1+IV_2+IV_3.
\end{align}
Proceeding with the same argument as in the derivation of  (\ref{combine_PHI1_continuity_1})-(\ref{combine_PHI1_continuity_3}), we obtain $IV_1$, $IV_2\rightarrow 0$ as $h\rightarrow0.$  In addition, for $IV_3$, noting assumption (H1), we have ${\bf f}:  (L^{\frac{pq}{p-q}}\cap \dot{W}^{1,r})_{\sigma}\rightarrow \gamma (U;(L^{\frac{pq}{p-q}}\cap \dot{W}^{1,r})_{\sigma})$ satisfies
\begin{align}\label{cty}
\mathbb E\bigg(\int_0^{T} \bigg\Vert\bigg(\sum_{k\geq 1}|\mathcal P{\bf f}({{\bf u}(s)})e_k|^2\bigg)^{\frac{1}{2}}\bigg\Vert^{\frac{pq}{p-q}}_{L^{\frac{pq}{p-q}}} ds\bigg)\leq K \mathbb E\bigg(\int_0^{T} (1+\Vert {\bf u}\Vert^{{\frac{pq}{p-q}}}_{L^{\frac{pq}{p-q}}} )(s)ds\bigg)<+\infty, 
\end{align}
where we used $(n,{\bf u})\in S_T.$  It follows from $\mathcal P{\bf f}({\bf u})\in \gamma(U;L^{\frac{pq}{p-q}}_{\sigma})$  that $\mathcal P{\bf f}({\bf u})\in\mathcal M^{\frac{pq}{p-q}}(0,T;\gamma(U;L^{\frac{pq}{p-q}}_{\sigma})).$  Note that $L^{\frac{pq}{p-q}_\sigma}$ satisfies the hypothesis (H1)-(H4) in Theorem 3.2 and Corollary 3.5 (i) in \cite{brzezniak1997stochastic}. Hence,  we deduce that $\int_0^t e^{-(t-s)A}\mathcal P{\bf f}({{\bf u}(s))}dW_s$ has a continuous modification.  Thus, as $h\rightarrow 0$, $IV_3\rightarrow 0$ almost surely. 

Note that $\mathcal P\nabla {\bf f}({\bf u})\in \gamma(U;L^{{r}}_{\sigma})$   with $r=2+\epsilon$ for $\epsilon>0$ small enough and
\begin{align}
\mathbb E\bigg(\int_0^T \bigg\Vert\bigg( \sum_{k\geq 1}|\mathcal P{\nabla\bf f}({\bf u})e_k|^2\bigg)^{\frac{1}{2}}\bigg\Vert_{L^r}^{\frac{pq}{p-q}} ds\bigg)\leq K\mathbb E\bigg(\int_0^T \big(1+\Vert\nabla { \bf u}\Vert^{\frac{pq}{p-q}}_{L^r}\big)(s)ds \bigg)<+\infty,\label{ctynabla}
\end{align}
where we used $(n,{\bf u})\in S_T.$ It follows that $\mathcal P\nabla {\bf f}({\bf u})\in \mathcal M^{\frac{pq}{p-q}}(0,T;L^r_{\sigma})$.  Moreover, we apply Theorem 3.2 and Corollary 3.5 in \cite{brzezniak1997stochastic}, then proceed with the same argument as above  to find $\Vert\nabla  \Phi_2(t+h)-\nabla \Phi_2(t)\Vert_{L^r} \rightarrow 0$ as $h\rightarrow 0$.  

Hence, in summary, we proved that $n$ and ${\bf u}$ are continuous in time almost surely.  
}
 
}
Combining the continuity of $n$ and ${\bf u}$ in time shown above,  (\ref{bounds_phi}) with (\ref{bounds_nablaphi}) together show that $\Phi=(\Phi_1,\Phi_2)$ defined in (\ref{eq220}) and (\ref{eq221}) maps $S_T$ into itself.

Next, we will show that the mapping $\Phi$ is a contraction  in $S_T$. That is, for any $(n_1, {\bf u}_1), (n_2, {\bf u}_2) \in S_T$,
\begin{align}\label{phiestimate}
\mathbb{E} &\big(\sup_{t\leq T}\|\Phi_1(n_1, {\bf u}_1)(t) - \Phi_1(n_2, {\bf u}_2)(t)\|^{\frac{pq}{p-q}}_{L^p}+\sup_{t\leq T}\|\Phi_2(n_1, {\bf u}_1)(t) - \Phi_2(n_2, {\bf u}_2)(t)\|^{\frac{pq}{p-q}}_{L^{\frac{pq}{p-q}}}\nonumber\\
&+\sup_{t\leq T}\|A^{\alpha}\Phi_2(n_1, {\bf u}_1)(t) - A^{\alpha}\Phi_2(n_2, {\bf u}_2)(t)\|^\frac{pq}{p-q}_{L^r}\big)\nonumber\\
\leq& C_{ m,T}(\,\mathbb E \|{\bf u}_1 - {\bf u}_2\|^\frac{pq}{p-q}_{X_{\bf u}^T}+  \mathbb E\|n_1 - n_2\|^\frac{pq}{p-q}_{X_n^T}) .
\end{align}

We begin with finding contraction estimates for $\Phi_1$.
First note that we almost surely have
\begin{align*}
   & \|\Phi_1(n_1, {\bf u}_1)(t) - \Phi_1(n_2, {\bf u}_2)(t)\|_{L^p}
\lesssim \left\|
\int_0^t e^{(t-s)\Delta} \nabla \cdot 
\big[\theta_m(\|n_1\|_{X_n^s})  n_1 \nabla c_1
 - \theta_m(\|n_2\|_{X_n^s})  n_2 \nabla c_2 \big]\, ds
\right\|_{L^p}\\
&\qquad+  \left\|
\int_0^t e^{(t-s)\Delta} \nabla \cdot 
\big[\theta_m(\|{\bf u}_1\|_{X_{\bf u}^s}) \theta_m(\|n_1\|_{X_n^s}){\bf u}_1 n_1
 - \theta_m(\|{\bf u}_2\|_{X_{\bf u}^s}) \theta_m(\|n_2\|_{X_n^s}) {\bf u}_2 n_2 \big]\, ds
\right\|_{L^{p}}\\
&\qquad:= I_1(t) + I_2(t).
\end{align*}

Now, we estimate $I_1$.  Similarly to the proof of (\ref{phi1estimatemar}), we find that
\begin{align}\label{I1nov5}
I_1(t) \lesssim& 
 \int_0^t (t-s)^{- \frac{1}{2}+\frac{1}{p}-\frac{1}{q}} 
\|\theta_m(\|n_1\|_{X_n^s}) n_1 \nabla c_1 
- \theta_m(\|n_2\|_{X_n^s}) n_2 \nabla c_2\|_{L^q}\, ds.
\end{align}
To bound this term, we denote the nonlinearity by:
\[
J(s) = 
\|\theta_m(\|n_1\|_{X_n^s}) n_1(s)\nabla c_1(s)
- \theta_m(\|n_2\|_{X_n^s}) n_2(s)\nabla c_2(s)\|_{L^q}.
\]

To bound $J$, we will distinguish the following six cases.

 \noindent {\bf Case (J1).} Suppose 
\[
\|n_1\|_{X_n^s} 
\vee \|n_2\|_{X_n^s} \le 2m,
\]
then we have
\begin{align*}
J(s) 
&= \|\theta_m(\|n_1\|_{X_n^s}) n_1\nabla c_1
- \theta_m(\|n_2\|_{X_n^s})n_2\nabla c_2\|_{L^q} \\
&\le 
\|\theta_m(\|n_1\|_{X_n^s})n_1\nabla c_1
- \theta_m(\|n_2\|_{X_n^s})n_1\nabla c_1\|_{L^q} \\
&\quad + 
\|\theta_m(\|n_2\|_{X_n^s})n_1\nabla c_1
- \theta_m(\|n_2\|_{X_n^s})n_2\nabla c_2\|_{L^q} \\
&\le 
|\theta_m(\|n_1\|_{X_n^s})
 - \theta_m(\|n_2\|_{X_n^s})
 \|n_1(s)\nabla c_1(s)\|_{L^q} \\
&\quad + \|n_1\nabla c_1(s) - n_2\nabla c_2(s)\|_{L^q} \\
&\le 
|\theta_m(\|n_1\|_{X_n^s})
 - \theta_m(\|n_2\|_{X_n^s})
 \|n_1(s)\|_{L^{p}}\|\nabla c_1(s)\|_{L^{\frac{pq}{p-q}}} \\
&\quad + \|n_1\|_{p}\|\nabla(c_1-c_2)(s)\|_{L^{\frac{pq}{p-q}}}
 + \|n_1(s)-n_2(s)\|_{L^p}\|\nabla c_2(s)\|_{L^{\frac{pq}{p-q}}}.
\end{align*}
Under the assumption (J1), we apply the fact that $\|\theta_m'\|_{\infty} \le \frac{C}{m}$, to find for $J(s)$ that
\begin{align*}
J(s)
&\lesssim 4m^2 \Big(|
\theta_m(\|n_1\|_{X_n^s})
- \theta_m(\|n_1\|_{X_n^s}|\Big) \\
&\quad + 2m(\|\nabla(c_1-c_2)(s)\|_{L^{\frac{pq}{p-q}}} + \|n_1(s)-n_2(s)\|_{L^p}) \\
&\lesssim m^2 \frac{1}{m} 
\left( \|n_1 - n_2\|_{X_n^s}\right)
+ 4m( \|n_1(s)-n_2(s)\|_{X_n^s}) \\
&\lesssim  m \left(  \|n_1 - n_2\|_{X_n^s}\right),
\end{align*}
where we additionally used 
\[
\|n_1 - n_2\|_{X_n^s} = \sup_{r \in [0,s]} \|n_1(r) - n_2(r)\|_{L^{p}}.
\]

 \noindent {\bf Case (J2).} Suppose  $\|n_1\|_{X_n^s} > 2m \text{ and } \|n_2\|_{X_n^s} > 2m$,\text{ then we have }
$$J(s) = 0.$$

 \noindent {\bf Case (J3).} Suppose  $\|n_1\|_{X_n^s} > 2m \text{ and } \|n_2\|_{X_n^s}  \le 2m$, then we have
\begin{align*}
 J(s) &= \|\theta_m(\|n_2\|_{X_n^s})   n_2 \nabla c_2\|_{L^q} \notag \\
&  = |\theta_m(\|n_2\|_{X_n^s}) - \theta_m(\|n_1\|_{X_n^s})| \,  \|n_2 \nabla c_2\|_{L^q} \notag \\
& \lesssim \frac{1}{m} \|n_1 - n_2\|_{X_n^s}\Vert n_2\Vert_{L^{p}}\Vert \nabla c_2\Vert_{L^\frac{pq}{p-q}}\\
&\lesssim m\Vert n_1-n_2\Vert_{X_n^s}.
\end{align*}

 \noindent {\bf Case (J4).}  Suppose  
$\|n_1\|_{X_n^s} < 2m \text{ and } \|n_2\|_{X_n^s} > 2m,$  then we have
\begin{align*}
J(s) &= \|\theta_m(\|n_1\|_{X_n^s})   n_1 \nabla c_1\|_{L^q} \notag \\
& = |\theta_m(\|n_1\|_{X_n^s}) - \theta_m(\|n_2\|_{X_n^s})| \,  \|n_1 \nabla c_1\|_{L^q} \notag \\
& \lesssim \frac{1}{m} \|n_1 - n_2\|_{X_n^s}\Vert n_1\Vert_{L^{p}}\Vert \nabla c_1\Vert_{L^\frac{pq}{p-q}}\\
&\lesssim m\Vert n_1-n_2\Vert_{X^s_n}.
\end{align*}
\noindent Putting (J1)--(J4) together, we get
\begin{equation}
J(s) \le C_m \big( \|n_1 - n_2\|_{X_n^s} \big), \label{eq:3.14}
\end{equation}
where constant $C_m>0$ depends on $m.$ 

 Substituting \eqref{eq:3.14} into (\ref{I1nov5}), we obtain
\begin{align}
 \, I_1(t) &\le C_{m} \big( \|n_1 - n_2\|_{X_n^T} \big) 
\int_0^t (t-s)^{-\frac{1}{2}+\frac{1}{p}-\frac{1}{q}} \, ds.    \label{eq:3.15}
\end{align}

\noindent Now, we turn to the term $I_2$. Our aim is to prove that,
\begin{equation}
\big\| \theta_m(\|{\bf u}_1\|_{X^s_{\bf u}}) \theta_m(\|n_1\|_{X^s_n}) {\bf u}_1 n_1 
- \theta_m(\|{\bf u}_2\|_{X^s_{\bf u}}) \theta_m(\|n_2\|_{X^s_n}) {\bf u}_2 n_2 \big\|_{L^q} 
\le C_m \big( \|{\bf u}_1 - {\bf u}_2\|_{X_s^u} + \|n_1 - n_2\|_{X^s_n} \big).
\end{equation}
As earlier, we define
\begin{align*}
\bar J(s):=\theta_m(\|{\bf  u}_1\|_{X_{\bf u}^s}) \theta_m(\|n_1\|_{X_n^s}) {\bf u}_1 n_1
 - \theta_m(\|{\bf u}_2\|_{X_{\bf u}^s}) \theta_m(\|n_2\|_{X_n^s}) {\bf u}_2 n_2,
\end{align*}
and bound $\bar J$ by splitting into the following six cases.

\noindent {\bf Case ($\mathbf{\bar{J}}$1).} Suppose 
\[
\|n_1\|_{X_n^s} \vee \|{\bf u}_1\|_{X_{\bf u}^s} \vee \|n_2\|_{X_n^s} \vee \|c_2\|_{X_c^s} \le 2m,
\] 
then we have
\begin{align*}
\bar J(s) &= \big\| \theta_m(\|n_1\|_{X_n^s}) \theta_m(\|{\bf u}_1\|_{X_{\bf u}^s}) n_1  {\bf u}_1
- \theta_m(\|n_2\|_{X_n^s}) \theta_m(\|{\bf u}_2\|_{X_{\bf u}^s}) n_2 {\bf u}_2 \big\|_{L^q} \\
&\le \big\| \theta_m(\|n_1\|_{X_n^s}) \theta_m(\|{\bf u}_1\|_{X_s^u}) n_1{\bf u}_1
- \theta_m(\|n_2\|_{X_s^n}) \theta_m(\|{\bf u}_2\|_{X_s^u}) n_1 {\bf u}_1 \big\|_{L^q} \\
&\quad + \big\| \theta_m(\|n_2\|_{X_s^n}) \theta_m(\|{\bf u}_2\|_{X_s^u}) n_1 {\bf u}_1
- \theta_m(\|n_2\|_{X_s^n}) \theta_m(\|{\bf u}_2\|_{X_s^u}) n_2  {\bf u}_2 \big\|_{L^q} \\
&\le \big|\theta_m(\|n_1\|_{X_n^s}) \theta_m(\|{\bf u}_1\|_{X_{\bf u}^s})
- \theta_m(\|n_2\|_{X_n^s}) \theta_m(\|{\bf  u}_2\|_{X_{\bf u}^s}) \big| \, \|n_1(s) {\bf u}_1(s)\|_{L^q} \\
&\quad + \| n_1 {\bf u}_1(s) - n_2 {\bf u}_2(s) \|_{L^q} \\
&\le \big|\theta_m(\|n_1\|_{X_n^s}) \theta_m(\|{\bf u}_1\|_{X^s_{\bf u}})
- \theta_m(\|n_2\|_{X^s_n}) \theta_m(\|{\bf u}_2\|_{X^s_{\bf u}}) \big| \, \|n_1(s)\|_{L^p} \|\nabla {\bf u}_1(s)\|_{L^\frac{pq}{p-q}} \\
&\quad + \|n_1\|_{L^p}\|{\bf u}_1(s) - {\bf u}_2(s)\|_{L^{\frac{pq}{p-q}}} + \|n_1(s) - n_2(s)\|_{L^p}\|{\bf u}_2(s)\|_{L^{\frac{pq}{p-q}}}.
\end{align*}

\noindent For our assumption in (J1) and the fact that $\|\theta_m'\|_\infty \le C m$, we find
\begin{align*}
\bar J(s)
&\lesssim 4m^2 \Big(|
\theta_m(\|n_1\|_{X_n^s}) \theta_m(\|{\bf u}_1\|_{X^s_{\bf u}})
- \theta_m(\|n_2\|_{X^s_n}) \theta_m(\|{\bf u}_2\|_{X^s_{\bf u}})|\Big) \\
&\quad + 2m(\|\nabla(c_1-c_2)(s)\|_{L^{\frac{pq}{p-q}}} + \|n_1(s)-n_2(s)\|_{L^p}) \\
&\lesssim 4m^2 \Big(|
\theta_m(\|n_1\|_{X_n^s}) \theta_m(\|u_1\|_{X^s_{\bf u}})-\theta_m(\|n_1\|_{X_n^s}) \theta_m(\|u_2\|_{X^s_{\bf u}})\\
&+\theta_m(\|n_1\|_{X_n^s}) \theta_m(\|u_2\|_{X^s_{\bf u}})
- \theta_m(\|n_2\|_{X^s_n}) \theta_m(\|u_2\|_{X^s_{\bf u}})|\Big) \\
&\quad + 2m(\|(u_1-u_2)(s)\|_{L^{\frac{pq}{p-q}}} + \|n_1(s)-n_2(s)\|_{L^p})\\
&\lesssim m^2 \frac{1}{m} 
\left( \|n_1 - n_2\|_{X_n^s}+ \|u_1 - u_2\|_{X_{\bf u}^s}\right)
+ 2m( \|n_1(s)-n_2(s)\|_{X_n^s}+ \|u_1(s)-u_2(s)\|_{X_{\bf u}^s}) \\
&\lesssim  m \left(  \|n_1 - n_2\|_{X_n^s}+\|u_1 - u_2\|_{X_{\bf u}^s}\right),
\end{align*}
where $\|n_1 - n_2\|_{X_n^s} = \sup_{r \in [0,s]} \|n_1(r) - n_2(r)\|_{L^p}$.

\medskip
\noindent{\bf Case ($\mathbf{\bar{J}}$2).} Suppose 
\[
\|n_1\|_{X_n^s} \vee \|{\bf u}_1\|_{X_{\bf u}^s} > 2m \quad \text{and} \quad \|n_2\|_{X_n^s} \vee \|{\bf u}_2\|_{X_{\bf u}^s} > 2m.
\]
Then $\bar J(s) = 0$.

\medskip
\noindent {\bf Case ($\mathbf{\bar{J}}$3).} Suppose $\|n_1\|_{X_n^s} > 2m$ and $\|n_2\|_{X_n^s} \vee \|{\bf u}_2\|_{X_{\bf u}^s} \le 2m$, then one has
\begin{align}
\bar J(s) =& \|\theta_m(\|n_2\|_{X_n^s}) \theta_m(\|{\bf u}_2\|_{X_{\bf u}^s}) n_2  {\bf u}_2\|_{L^q}\nonumber\\
\le & \|[\theta_m(\|n_2\|_{X_n^s}) \theta_m(\|{\bf u}_2\|_{X_{\bf u}^s})-\theta_m(\|n_1\|_{X_n^s}) \theta_m(\|{\bf u}_1\|_{X_{\bf u}^s})]n_2  {\bf u}_2\|_{L^q}\nonumber\\
\le &|[\theta_m(\|n_2\|_{X_n^s}) \theta_m(\|{\bf u}_2\|_{X_{\bf u}^s})-\theta_m(\|n_2\|_{X_n^s}) \theta_m(\|{\bf u}_1\|_{X_{\bf u}^s})\nonumber\\
&+\theta_m(\|n_2\|_{X_n^s}) \theta_m(\|{\bf u}_1\|_{X_{\bf u}^s})-\theta_m(\|n_1\|_{X_n^s}) \theta_m(\|{\bf u}_1\|_{X_{\bf u}^s})|\Vert n_2\Vert_{L^{p}}\Vert {\bf u}_2\Vert_{L^{\frac{pq}{p-q}}}\nonumber\\
\le& C_m (\|n_1 - n_2\|_{X_n^s}+\|{\bf u}_1 - {\bf u}_2\|_{X_{\bf u}^s}),
\end{align}
where $C_m>0$ is a constant depending on $m.$

\medskip
\noindent {\bf Case ($\mathbf{\bar{J}}$4).} Suppose $\|{\bf u}_1\|_{X_{\bf u}^s} > 2m$ and $\|n_2\|_{X_n^s} \vee \|{\bf u}_2\|_{X_{\bf u}^s} \le 2m$:
\[
\bar J(s) \le C_m (\|n_1 - n_2\|_{X_n^s}+\|{\bf u}_1 - {\bf u}_2\|_{X_{\bf u}^s}),
\]
where $C_m>0$ is a constant depending on $m.$

\medskip
\noindent The proofs of Cases ($\bar{\text{J}}$5) and ($\bar{\text{J}}$6) are similar to ($\bar{\text{J}}$3) and ($\bar{\text{J}}$4), which are

\noindent {\bf Case ($\mathbf{\bar{J}}$5).} Suppose $\|n_2\|_{X_n^s} > 2m, \ \|n_1\|_{X_n^s} \vee \|{\bf u}_1\|_{X_{\bf u}^s} \le 2m$, then 
$$\bar J(s) \le C_m (\|n_1 - n_2\|_{X_n^s}+\Vert {\bf u}_1-{\bf u}_2\Vert_{X_{\bf u}^s}),$$
for some positive constant $C_m$ depending on $m.$

\medskip
\noindent {\bf Case ($\mathbf{\bar{J}}$6).} Suppose $\|{\bf u}_2\|_{X_{\bf u}^s} > 2m, \ \|n_1\|_{X_n^s} \vee \|{\bf u}_1\|_{X_{\bf u}^s} \le 2m$, then  
$$\bar J(s) \le C_m (\|{\bf u}_1 - {\bf u}_2\|_{X_{\bf u}^s}+\|n_1-n_2\|_{X_n^s}),$$ for some positive constant $C_m$ depending on $m.$

\noindent Putting ($\bar{\text{J}}$1)--($\bar{\text{J}}$6) together, we obtain
\begin{equation}
\bar J(s) \le C_m \big( \|n_1 - n_2\|_{X_n^s} + \|{\bf u}_1 - {\bf u}_2\|_{X_{\bf u}^s} \big),\label{eq:3.141}
\end{equation}
for some positive constant $C_m$ depending on $m$.

\noindent Substituting \eqref{eq:3.141} into (\ref{I1nov5}), we obtain for $t\in(0,T),$ that
\begin{align}
I_2(t) &\le C_{ m} \big(  \|n_1 - n_2\|_{X_n^T} + \|{\bf u}_1 - {\bf u}_2\|_{X_{\bf u}^T} \big) 
\int_0^t (t-s)^{-\frac{1}{2}+\frac{1}{p}-\frac{1}{q}} \, ds \notag \\
&\le C_{ m,T} \big( \|n_1 - n_2\|_{X_n^T} + \|{\bf u}_1 - {\bf u}_2\|_{X_{\bf u}^T} \big),
\label{eq:3.151}
\end{align}
where $C_{m,T}$ is some positive constant depending on $m$ and $T.$ 

Combining (\ref{eq:3.15}) with (\ref{eq:3.151}), we have the following desired estimate for $\Phi_1$: For any $t\in(0,T),$ we have
\begin{align}\label{phi1estimate}
\|\Phi_1(n_1, {\bf u}_1)(t) - \Phi_1(n_2, {\bf u}_2)(t)\|_{L^p}
\leq C_{ m,T} \big( \|n_1 - n_2\|_{X_n^T} + \|{\bf u}_1 - {\bf u}_2\|_{X_{\bf u}^T} \big),
\end{align}
where constant $C_{m,T}>0$ depending on $m$ and $T.$

To ultimately prove \eqref{phiestimate}, we will next estimate $\|\Phi_2(n_1, {\bf u}_1) - \Phi_2(n_2, {\bf u}_2)\|_{L^\frac{pq}{p-q}}$. Note that we have for $t\in(0,T),$
\begin{align}\label{gamma123mar}
\|\Phi_2(n_1, {\bf u}_1)(t)& - \Phi_2(n_2, {\bf u}_2)(t)\|_{L^\frac{pq}{p-q}}\\
&\lesssim  \int_0^t (t-s)^{-\frac{1}{2}+\frac{1}{p}-\frac{1}{q}}
\|\theta_m(\|{\bf u}_1\|_{X_{\bf u}^s})({\bf u}_1  \otimes {\bf u}_1)
 - \theta_m(\|{\bf u}_2\|_{X_{\bf u}^s})({\bf u}_2\otimes{\bf  u}_2)\|_{L^\frac{pq}{2(p-q)}}\, ds\nonumber \\
&\quad +  \int_0^t (t-s)^{-\frac{1}{2}+\frac{1}{p}-\frac{1}{q}}
\|\theta_m(\|n_1\|_{X_n^s})\nabla c_1\otimes \nabla c_1
 - \theta_m(\|n_2\|_{X_n^s})\nabla c_2\otimes \nabla c_2\|_{L^\frac{pq}{2(p-q)}}\, ds\nonumber \\
&\quad +\Big\|
\int_0^t e^{-(t-s)A} \mathcal P[{\bf f}({\bf u}_1(s)) - {\bf f}({\bf u}_2(s))]\, dW_s
\Big\|_{L^\frac{pq}{p-q}} \nonumber\\
&=: \Gamma_1(t) + \Gamma_2(t) + \Gamma_3(t).
\end{align}
To estimate the term $\Gamma_1$, we set
\[
J_1(s) = \|\theta_m(\|{\bf u}_1\|_{X_{\bf u}^s})({\bf u}_1\otimes {\bf u}_1) 
         - \theta_m(\|{\bf u}_2\|_{X_{\bf u}^s})({\bf u}_2\otimes {\bf u}_2)\|_{L^\frac{pq}{2(p-q)}}.
\]

As earlier, we will bound $J_1$ by considering four different cases. 

\medskip
\noindent
{\bf Case ($\mathbf{{J}_1}$1).} Suppose $\|{\bf u}_1\|_{X_{\bf u}^s} \vee \|{\bf u}_2\|_{X_{\bf u}^s} \le 2m$.  
By using the definition of $\theta_m$, we get
\begin{align*}
J_1(s) 
&\le \|{\bf u}_1\otimes {\bf u}_1 -{\bf u}_2\otimes {\bf u}_2\|_{L^\frac{pq}{2(p-q)}} 
   + |\theta_m(\|{\bf u}_1\|_{X_{\bf u}^s}) - \theta_m(\|{\bf u}_2\|_{X_{\bf u}^s})|
     \,\|{\bf{u}}_2\otimes {\bf u}_2\|_{L^\frac{pq}{2(p-q)}} \\
&\le  \, \|{\bf u}_1\|_{L^\frac{pq}{p-q}}\|({\bf u}_1 - {\bf u}_2)\|_{L^\frac{pq}{p-q}} 
   + \|{\bf u}_1 - {\bf u}_2\|_{L^\frac{pq}{p-q}}\|{ \bf u}_2\|_{L^\frac{pq}{p-q}} 
   +  \frac{C}{m}\|{\bf u}_1 - {\bf u}_2\|_{X_{\bf u}^s}\|{\bf u}_2\|_{L^\frac{pq}{p-q}}^2 \\
&\lesssim m\|{\bf u}_1 - {\bf u}_2\|_{X_{\bf u}^s}.
\end{align*}

\medskip
\noindent
{\bf Case ($\mathbf{{J}_1}$2).} Suppose $\|{\bf u}_1\|_{X_{\bf u}^s} \le 2m$ and $\|{\bf u}_2\|_{X_{\bf u}^s} > 2m$.  
We have
\begin{align*}
J_1(s) &= \|\theta_m(\|{\bf u}_1\|_{X_{\bf u}^s})({\bf u}_1\otimes {\bf u}_1)\|_{L^\frac{pq}{2(p-q)}} \\
&= |\theta_m(\|{\bf u}_1\|_{X_{\bf u}^s}) - \theta_m(\|{\bf u}_2\|_{X_{\bf u}^s})|
   \,\|{\bf u}_1\otimes {\bf u}_1\|_{L^\frac{pq}{2(p-q)}} \\
&\lesssim m \|{\bf u}_1 - {\bf u}_2\|_{X_{\bf u}^s}.
\end{align*}

\medskip
\noindent
\noindent{\bf Case ($\mathbf{{J}_1}$3).} Suppose $\|{\bf u}_1\|_{X_{\bf u}^s} > 2m$ and $\|{\bf u}_2\|_{X_{\bf u}^s} \le 2m$.  
As in the case (J$_1$2), we have
\[
J_1(s) \lesssim m \|{\bf u}_1 - {\bf u}_2\|_{X_{\bf u}^s}.
\]

\medskip
\noindent
{\bf Case ($\mathbf{{J}_1}$4).} Suppose $\|{\bf u}_1\|_{X_{\bf u}^s} \wedge \|{\bf u}_2\|_{X_{\bf u}^s} > 2m$, then
\[
J_1(s) = 0.
\]
Hence, it follows by combining these four cases that,
\begin{equation}\label{eq:Gamma1mar}
\Gamma_1(t) \leq \, C_{m,T}\, \|{\bf u}_1 - {\bf u}_2\|_{X_{\bf u}^t}
\end{equation}
for some positive constant $C_{m,T}.$

Similarly, to estimate $\Gamma_{2}$, we define
$$\tilde J(s):=\|\theta_m(\|n_1\|_{X_n^s})\nabla c_1\otimes \nabla c_1
 - \theta_m(\|n_2\|_{X_n^s})\nabla c_2\otimes \nabla c_2\|_{L^\frac{pq}{2(p-q)}},$$
and consider the following cases.

 \noindent {\bf Case ($\mathbf{\tilde{J}}$1).}  Suppose $\|n_1\|_{X_n^s} \vee \|n_2\|_{X_n^s} \le 2m$.  
From the definition of $\theta_m$, we get
\begin{align}
\tilde J(s):=&\|\theta_m(\|n_1\|_{X_n^s})\nabla c_1\otimes \nabla c_1-\theta_m(\|n_2\|_{X_n^s})\nabla c_1\otimes \nabla c_1\nonumber\\
&\qquad+\theta_m(\|n_2\|_{X_n^s})\nabla c_1\otimes \nabla c_1
 - \theta_m(\|n_2\|_{X_n^s})\nabla c_2\otimes \nabla c_2\|_{L^\frac{pq}{2(p-q)}}\nonumber\\
 \leq &\|\theta_m(\|n_1\|_{X_n^s})-\theta_m(\|n_2\|_{X_n^s})\|\|\nabla c_1\otimes \nabla c_1\|_{L^{\frac{pq}{2(p-q)}}}+\|\nabla c_1\otimes \nabla c_1-\nabla c_2\otimes \nabla c_2\|_{L^{\frac{pq}{2(p-q)}}}\nonumber\\
 \lesssim &m\Vert n_1-n_2\Vert_{X_n^s}+\Vert \nabla c_1\otimes \nabla(c_1-c_2)\Vert_{L^{\frac{pq}{2(p-q)}}}+\|\nabla c_2\otimes\nabla(c_1-c_2)\Vert_{L^{\frac{pq}{2(p-q)}}}\nonumber\\
 \lesssim &m\Vert n_1-n_2\Vert_{X_n^s}.
 \end{align}
 \medskip
\noindent
{\bf Case ($\mathbf{\tilde{J}}$2).} Suppose $\|n_1\|_{X_{\bf u}^s} \le 2m$ and $\|n_2\|_{X_{\bf u}^s} > 2m$.  
We have
\begin{align*}
\tilde J(s) &= \|\theta_m(\|n_1\|_{X_{\bf u}^s})(\nabla c_1\otimes \nabla c_1)\|_{L^\frac{pq}{2(p-q)}} \\
&= |\theta_m(\|n_1\|_{X_n^s}) - \theta_m(\|n_2\|_{X_n^s})|
   \,\|\nabla c_1\otimes \nabla c_1\|_{L^\frac{pq}{2(p-q)}} \\
&\lesssim m \|n_1 - n_2\|_{X_n^s}.
\end{align*}

\medskip
\noindent
\noindent{\bf Case ($\mathbf{\tilde{J}}$3).} Suppose $\|n_1\|_{X_n^s} > 2m$ and $\|n_2\|_{X_n^s} \le 2m$.  
Similar to case (J2), we have
\[
\tilde J(s) \lesssim m \|n_1 - n_2\|_{X_{\bf u}^s}.
\]

\medskip
\noindent
{\bf Case ($\mathbf{\tilde{J}}$4).} Suppose $\|n_1\|_{X_{\bf u}^s} \wedge \|n_2\|_{X_n^s} > 2m$.  
Then
\[
\tilde J(s) = 0.
\]
Hence, we have,
\begin{equation}\label{eq:Gamma2mar}
\Gamma_2(t) \leq \, C_{m,T} \|n_1 - n_2\|_{X_n^T}, \quad t \le T,
\end{equation}
for some positive constant $C_{m,T}.$

To find bounds for $\Gamma_3$ we consider $\Gamma_3(t) = \|{\bf Z}_2(t)\|_{L^\frac{pq}{p-q}}$, where ${\bf Z}_2$ satisfies the following SPDE:
\begin{equation}
\begin{cases}
d{\bf Z}_2(s) =- A  {\bf Z}_2(s)\, ds +\mathcal P[{\bf f}({\bf u}_1(s)) - {\bf f}({\bf u}_2(s))]\, dW_s,\\[4pt]
{\bf Z}_2(0) = {\bf 0}.
\end{cases}
\end{equation}
Using It\^{o}'s formula, the BDG inequality and Young's inequality, we have
\begin{align}\label{110newlabel1}
\frac{1}{2}\mathbb{E}\left( \sup_{t\leq T} \|{\bf Z}_2(t)\|^\frac{pq}{p-q}_{L^\frac{pq}{p-q}} \right)
\lesssim  \, \mathbb{E} \bigg(\int_0^T  (\|{\bf u}_1(s)-{\bf u}_2(s)\|_{L^\frac{pq}{p-q}}^\frac{pq}{p-q}) ds\bigg)
\leq & C_T \mathbb E(\Vert {\bf u}_1-{\bf u}_2\Vert_{X_{\bf u}^T}^\frac{pq}{p-q}).
\end{align}
for some positive constant $C_T>0.$

In conclusion, we can deduce from (\ref{eq:Gamma1mar}), (\ref{eq:Gamma2mar}) and (\ref{gamma123mar}) that we have,
\begin{align}\label{eq318}
\mathbb E\big[\sup_{t\leq T}\|\Phi_2(n_1, {\bf u}_1)(t) &- \Phi_2(n_2, {\bf u}_2)(t)\|^\frac{pq}{p-q}_{L^{\frac{pq}{p-q}}}\big]
\lesssim  \Gamma_1^\frac{pq}{p-q}(t) + \Gamma^{\frac{pq}{p-q}}_2(t) + \Gamma^{\frac{pq}{p-q}}_3(t)\nonumber\\
\leq  C_{m,T} & [\,\mathbb E \|{\bf u}_1 - {\bf u}_2\|^\frac{pq}{p-q}_{X_{\bf u}^T}+  \mathbb E\|n_1 - n_2\|^\frac{pq}{p-q}_{X_n^T}+\mathbb E(\Vert {\bf u}_1-{\bf u}_2\Vert_{X_{\bf u}^T}^\frac{pq}{p-q})]
\end{align}
for some positive constant $C_{m,T}$ depending on $m$ and $T.$

Now we estimate $\|A^{\alpha}\Phi_2(n_1, {\bf u}_1) - A^{\alpha}\Phi_2(n_2, {\bf u}_2)\|_{L^r}$.  We find
\begin{align}\label{combine1mar6}
&\|A^{\alpha}\Phi_2(n_1, {\bf u}_1)(t) - A^{\alpha}\Phi_2(n_2, {\bf u}_2)(t)\|_{L^r}
\nonumber\\
\lesssim&  \int_0^t (t-s)^{-\frac{1}{2}+\frac{1}{r}-\frac{pq+rp-rq}{rpq}}
\|\theta_m(\|{\bf u}_1\|_{X_{\bf u}^s})({\bf u}_1  \cdot \nabla {\bf u}_1)
 - \theta_m(\|{\bf u}_2\|_{X_{\bf u}^s})({\bf u}_2  \cdot \nabla {\bf u}_2)\|_{L^\frac{rpq}{pq+rp-rq}}\, ds\nonumber \\
&\quad +  \int_0^t (t-s)^{-\frac{1}{2}+\frac{1}{r}-\frac{1}{q}}
\|\theta_m(\|n_1\|_{X_n^s})n_1\nabla c_1
 - \theta_m(\|n_2\|_{X_n^s})n_2\nabla c_2\|_{L^q}\, ds \nonumber\\
&\quad +\Big\|
\int_0^t e^{(t-s)\Delta} \mathcal PA^{\alpha}[{\bf f}({\bf u}_1(s)) - {\bf f}({\bf u}_2(s))]\, dW_s
\Big\|_{L^r} \nonumber\\
&=: \bar \Gamma_1(t) + \bar \Gamma_2(t) + \bar\Gamma_3(t).
\end{align}
To estimate $\bar\Gamma_1$, set
\[
\bar J_1(s) = \|\theta_m(\|{\bf u}_1\|_{X_{\bf u}^s})({\bf u}_1\cdot \nabla {\bf u}_1) 
         - \theta_m(\|{\bf u}_2\|_{X_{\bf u}^s})({\bf u}_2\cdot \nabla {\bf u}_2)\|_{L^\frac{rpq}{pq+rp-rq}}.
\]
We will bound $\bar J_1$ in four different cases. 

\medskip
\noindent
{\bf Case ($\mathbf{\bar{J}_1}$1).} Suppose $\|{\bf u}_1\|_{X_{\bf u}^s} \vee \|{\bf u}_2\|_{X_{\bf u}^s} \le 2m$.  
In light of the definition of $\theta_m$, we get
\begin{align*}
\bar J_1(s) 
&\le \|{\bf u}_1\cdot\nabla {\bf u}_1 -{\bf u}_2\cdot\nabla {\bf u}_2\|_{L^\frac{rpq}{pq+rp-rq}} 
   + |\theta_m(\|{\bf u}_1\|_{X_{\bf u}^s}) - \theta_m(\|{\bf u}_2\|_{X_{\bf u}^s})|
     \,\|{\bf u}_2\cdot\nabla {\bf u}_2\|_{L^\frac{rpq}{pq+rp-rq}} \\
&\le  \, \|{\bf u}_1\|_{L^\frac{pq}{p-q}}\|\nabla ({\bf u}_1 - {\bf u}_2)\|_{L^r} 
   + \|{\bf u}_1 - {\bf u}_2\|_{L^\frac{pq}{p-q}}\| \nabla {\bf u}_2\|_{L^r} 
   +  \frac{C}{m}\|{\bf u}_1 - {\bf u}_2\|_{X_{\bf u}^s}\|{\bf u}_2\|_{L^\frac{pq}{p-q}}\Vert\nabla {\bf u}_2\Vert_r \\
&\lesssim m\|{\bf u}_1 -{\bf u}_2\|_{X_{\bf u}^s}.
\end{align*}

\medskip
\noindent
{\bf Case ($\mathbf{\bar{J}_1}$2).} Suppose $\|{\bf u}_1\|_{X_{\bf u}^s} \le 2m$ and $\|{\bf u}_2\|_{X_{\bf u}^s} > 2m$.  
We have
\begin{align*}
\bar J_1(s) &= \|\theta_m(\|{\bf u}_1\|_{X_{\bf u}^s})({\bf u}_1\cdot \nabla {\bf u}_1)\|_{L^\frac{rpq}{pq+rp-rq}} \\
&= |\theta_m(\|{\bf u}_1\|_{X_{\bf u}^s}) - \theta_m(\|{\bf u}_2\|_{X_{\bf u}^s})|
   \,\|{\bf u}_1\cdot\nabla {\bf u}_1\|_{L^\frac{rpq}{pq+rp-rq}} \\
&\lesssim m \|{\bf u}_1 - {\bf u}_2\|_{X_{\bf u}^s}.
\end{align*}

\medskip
\noindent
\noindent{\bf Case ($\mathbf{\bar{J}_1}$3).} Suppose $\|{\bf u}_1\|_{X_{\bf u}^s} > 2m$ and $\|{\bf u}_2\|_{X_{\bf u}^s} \le 2m$.  
Similar to case (J2), we have
\[
\bar J_1(s) \lesssim m \|{\bf u}_1 - {\bf u}_2\|_{X_{\bf u}^s}.
\]

\medskip
\noindent
{\bf Case ($\mathbf{\bar{J}_1}$4).} Suppose $\|{\bf u}_1\|_{X_{\bf u}^s} \wedge \|{\bf u}_2\|_{X_{\bf u}^s} > 2m$.  
Then
\[
\bar J_1(s) = 0.
\]

\medskip
Hence, it follows that,
\begin{equation}\label{eq:Gamma1mar6}
\bar \Gamma_1(t) \leq \, C_{m,T}\, \|{\bf u}_1 - {\bf u}_2\|_{X_{\bf u}^t}.
\end{equation}

Similarly, to treat $\bar\Gamma_2$, we define
$$\tilde J_1(s):=\|\theta_m(\|n_1\|_{X_n^s})n_1 \nabla c_1
 - \theta_m(\|n_2\|_{X_n^s}) n_2 \nabla c_2\|_{L^q}.$$
Consider the following four cases:\\
 \noindent {\bf Case ($\mathbf{\tilde{J}_1}$1).}  Suppose $\|n_1\|_{X_n^s} \vee \|n_2\|_{X_n^s} \le 2m$.  
From the definition of $\theta_m$, we get
\begin{align}
\tilde J_1(s):=&\|\theta_m(\|n_1\|_{X_n^s})n_1 \nabla c_1-\theta_m(\|n_2\|_{X_n^s})n_1 \nabla c_1+\theta_m(\|n_2\|_{X_n^s})n_1 \nabla c_1
 - \theta_m(\|n_2\|_{X_n^s})n_2 \nabla c_2\|_{L^q}\nonumber\\
 \leq &\|\theta_m(\|n_1\|_{X_n^s})-\theta_m(\|n_2\|_{X_n^s})\|\|n_1 \nabla c_1\|_{L^q}+\|n_1\nabla c_1-n_2 \nabla c_2\|_{L^q}\nonumber\\
 \lesssim &m\Vert n_1-n_2\Vert_{X_n^s}+\Vert  (n_1-n_2)\nabla c_1\Vert_{L^{\frac{pq}{2(p-q)}}}+\|n_2\nabla(c_1-c_2)\Vert_{L^q}\nonumber\\
 \lesssim &m\Vert n_1-n_2\Vert_{X_n^s}.
 \end{align}
 \medskip
\noindent
{\bf Case ($\mathbf{\tilde{J}_1}$2).} Suppose $\|n_1\|_{X_{\bf u}^s} \le 2m$ and $\|n_2\|_{X_{\bf u}^s} > 2m$.  
We have
\begin{align*}
\tilde J_1(s) &= \|\theta_m(\|n_1\|_{X_{\bf u}^s})n_1\nabla c_1\|_{L^q} \\
&= |\theta_m(\|n_1\|_{X_n^s}) - \theta_m(\|n_2\|_{X_n^s})|
   \,\|n_1 \nabla c_1\|_{L^q} \\
&\lesssim m \|n_1 - n_2\|_{X_n^s}.
\end{align*}

\medskip
\noindent
\noindent{\bf Case ($\mathbf{\tilde{J}_1}$3).} Suppose $\|n_1\|_{X_n^s} > 2m$ and $\|n_2\|_{X_n^s} \le 2m$.  
Similar to case (J2), we find
\[
\tilde J_1(s) \lesssim m \|n_1 - n_2\|_{X_{\bf u}^s}.
\]

\medskip
\noindent
{\bf Case ($\mathbf{\tilde{J}_1}$4).} Suppose $\|n_1\|_{X_{\bf u}^s} \wedge \|n_2\|_{X_n^s} > 2m$.  
Then
\[
\tilde J_1(s) = 0.
\]
Combining these cases, we have thus shown that
\begin{equation}\label{eq:Gamma2mar6}
\bar\Gamma_2(t) \lesssim \,C_T \|n_1 - n_2\|_{X_n^T}, \quad t \le T.
\end{equation}
Next, note that $\bar\Gamma_3(t) = \|{\bf Z}_3(t)\|_{L^r}$, where ${\bf Z}_3$ satisfies the following SPDE:
\begin{equation}
\begin{cases}
d{\bf Z}_3(s) = -A {\bf Z}_3(s)\, ds +\mathcal PA^{\alpha}[{\bf f}({\bf u}_1(s)) - {\bf f}({\bf u}_2(s))]\, dW_s,\\[4pt]
{\bf Z}_3(0) = {\bf 0}.
\end{cases}
\end{equation}
Using It\^{o}'s formula, BDG inequality, Young's inequality and assumption (H2), we obtain
\begin{align}\label{110newlabel1decmar}
\frac{1}{2}\mathbb{E}\left( \sup_{t\leq T} \|{\bf Z}_3(t)\|^{\frac{pq}{p-q}}_{L^r} \right)
\lesssim&  \, \mathbb{E} \bigg(\int_0^T  (\|A^{\alpha}({\bf u}_1-{\bf u}_2)(s)\|_{L^r}^\frac{pq}{p-q}) ds\bigg)
\lesssim  C_T \mathbb E(\Vert {\bf u}_1-{\bf u}_2\Vert_{X_{\bf u}^T}^\frac{pq}{p-q})
\end{align}
for some constant $C_T>0$ depending on $T.$  We collect (\ref{combine1mar6}), (\ref{eq:Gamma1mar6}), (\ref{eq:Gamma2mar6}) and (\ref{110newlabel1decmar}) to obtain
\begin{align}\label{eq318mar6}
\mathbb E\big[\sup_{t\leq T}\|A^{\alpha}\Phi_2(n_1, {\bf u}_1)(t) - &A^{\alpha}\Phi_2(n_2, {\bf u}_2)(t)\|^\frac{pq}{p-q}_{L^r}\big]
\lesssim  \bar\Gamma_1^\frac{pq}{p-q}(t) + \bar \Gamma^{\frac{pq}{p-q}}_2(t) + \bar \Gamma^{\frac{pq}{p-q}}_3(t)\nonumber\\
\leq  C_{m,T} & [\,\mathbb E \|{\bf u}_1 - {\bf u}_2\|^\frac{pq}{p-q}_{X_{\bf u}^T}+  \mathbb E\|n_1 - n_2\|^\frac{pq}{p-q}_{X_n^T}+\mathbb E(\Vert {\bf u}_1-{\bf u}_2\Vert_{X_{\bf u}^T}^\frac{pq}{p-q})].
\end{align}

By virtue of (\ref{phi1estimate}), (\ref{eq318}) and (\ref{eq318mar6}), we conclude the proof of the desired estimate \eqref{phiestimate} restated below:
\begin{equation}\label{eq:phi_estimate}
\|{\bf\Phi}(n_1, {\bf u}_1) - {\bf \Phi}(n_2,  {\bf u}_2)\|^{\frac{pq}{p-q}}_{S_T}
\le C_{m,T} \|(n_1, {\bf u}_1) - (n_2,  {\bf u}_2)\|^{\frac{pq}{p-q}}_{S_T},
\end{equation}
where $C_{m,T}$ is some positive constant depending on $m$ and $T.$

We choose $T = T_m$ for which $C_{m,T} = \frac{1}{2}$. For this choice, \eqref{eq:phi_estimate} tells us that $\Phi$ is a contraction on the space $S_{T_m}$.  Observe that the deterministic constant $T_m>0$ does not depend on $n_0,{\bf u}_0$.
Therefore, thanks to the Banach fixed point theorem, we conclude that there exists a unique element
$(n^1_m, {\bf u}^1_m) \in S_{T_m}$ that solves (\ref{3.1}) for $t \in [0, T_m]$.

If $T_m>0$ can be chosen arbitrarily large, then we are done. If not, then next we will extend this solution in time and therefore prove the existence of global-in-time solution for the approximation system \eqref{3.1}.
For that purpose, we define for any $T>0$ and $m>0$, the following closed subspace of $S_T$ (see \eqref{norm-S_T}):
\begin{align*}
&S_T^1 =\{ \{\mathcal{F}_t\}_{t \in [0,T]}-\text{adapted, } L^{p}(\mathbb R^2)  \times (L^\frac{pq}{p-q}(\mathbb R^2)\cap \dot{W}^{1,r}(\mathbb R^2
)) \text{-valued processes }
(n(t), {\bf u}(t))_{t\in[0,T]}\\
&\text{such that }
(n,  {\bf u}) = (n^1_m,  {\bf u}^1_m) \quad \text{on } [0, T_m], \ \mathbb{P}\text{-a.s.}\}.
\end{align*}
Then $(S_T^1, \|\cdot\|_{S_T})$ is a Banach space, where $S_T$-norm is defined in \eqref{norm-S_T}.

We introduce a mapping ${\bf\Phi}^1 = (\Phi_1^1,  \Phi_2^1)$ on $S_T^1$ by defining
\begin{align*}
\Phi_1^1(n, {\bf u})(t + T_m)
&:= e^{t\Delta} n_m(T_m)
 - \int_{T_m}^{T_m + t} e^{(T_m + t - s)\Delta}
[ \theta_m(\|n\|_{X^s_n}) 
 \nabla \cdot (n \nabla c) + \theta_m\big(\|(n,{\bf u})\|_{X_s}\big){\bf u}\cdot\nabla n ]\, ds \\
\Phi_2^1(n,{\bf u})(t + T_m)
&:= e^{-tA} {\bf u}_m(T_m)
 - \int_{T_m}^{T_m + t} e^{-(T_m + t - s)A}
 \theta_m(\|{\bf u}\|_{X^s_{\bf u}}) \mathcal P\{({\bf u}(s) \cdot \nabla){\bf u}(s)\}\, ds \\
&\quad - \int_{T_m}^{T_m + t} e^{-(T_m + t - s)A}
 \theta_m(\|n\|_{X^s_n}) \mathcal P\{\nabla\cdot(\nabla c(s)\otimes \nabla c(s))\}\, ds \\
&\quad + \int_{T_m}^{T_m + t} e^{-(T_m + t - s)A}
 {\bf f}(u(s))\, dW_s.
\end{align*}

Using the same arguments as above, we can prove that ${\bf \Phi}^1$ is a contraction mapping on the space $S_{2T_m}^1$ endowed with the norm defined in \eqref{norm-S_T}. Moreover, since $T_m$ does not depend on the initial data we also have that ${\bf \Phi}^1$ satisfies the estimate \eqref{eq:phi_estimate}. This again, gives us a solution $(n^2_m,{\bf u}^2_m)\in S^1_{2T_m}$ for the approximation system (\ref{3.1}) on the time interval $[0,2T_m]$ such that it agrees with our constructed solution $(n^1_m,{\bf u}^1_m)$ on $[0,T_m]$.

Hence, for any $k\geq 0$ and $T=kT_m$, we repeat this argument $k$-many times and thus obtain, for any $T>0$,  
$$\text{a solution
$(n_m, {\bf u}_m) \in S_T$ 
for the approximation system (\ref{3.1}) on $[0,T]$. }$$
In other words, we have constructed a global-in-time solution to the approximation system \eqref{3.1}. This finishes Step 1.\\

\noindent Next, we proceed to {\bf Step 2}:\\
{Observe that, for as long as $\theta_m$ appearing in \eqref{3.1} is equal to 1, the solution $(n_m, {\bf u}_m)$ to \eqref{3.1} constructed above, defines is in fact a solution to the original system \eqref{PKSNS-time-dependent}. This motivates the following definition of a stopping time}.
\begin{equation}\label{eq:tau_m}
\tau_m = \inf\big\{t > 0 \; ; \;
\|n_m\|_{X^t_n}  \vee
\|{\bf u}_m\|_{X^t_{\bf u}} \ge m \big\}.
\end{equation}
Using standard arguments, we can show that $\tau_m$ is indeed an $\{\mathcal{F}_t\}_{t\geq 0}-$stopping time.

Then, using 
$$(n_m,{\bf u}_m)\in L^{\frac{pq}{p-q}}(\Omega; C([0,T];L^p(\mathbb R^2)))\times\big( L^{\frac{pq}{p-q}}(\Omega; C([0,T];L^p(\mathbb R^2)\cap {\dot{W}}^{1,r}(\mathbb R^2))) \big)$$
with $p$, $q$ and $r$ given in (\ref{pqr}), we can further show that for $m \gg \|n_0\|_p \vee \|{\bf u}_0\|_{\frac{pq}{p-q}}\vee \Vert \nabla {\bf u}_0\Vert_r $,
this stopping time is almost surely strictly positive, i.e.
\[
\mathbb{P}(\tau_m > 0) = 1.
\]
Therefore, as described above, due to  the definition of $\theta_m$, it follows that for any $m \geq 0$,
\begin{align}\label{localsol}
    (n_m,  {\bf u}_m ,\tau_m) 
\text{ is a local-in-time mild solution to the system } ~(\ref{PKSNS-time-dependent}).
\end{align}

This finishes Step 2 and thus concludes the proof of Theorem \ref{thm3mar8}.

\end{proof}
\begin{theorem}\label{thm:unique}
    Assume that $(n^1,{\bf u}^1,\tau^1)$ and $(n^2,{\bf u}^2,\tau^2)$ are two local solutions to \eqref{PKSNS-time-dependent} in the sense of Definition \ref{def:local_solution} (see Theorem \ref{thm3mar8}). Then,
    $$(n_1,{\bf u}_1)=(n_2,{\bf u}_2) \quad\text{ in } L^p(\mathbb R^2)\times L^{\frac{pq}{p-q}}(\mathbb R^2)\quad \text{ for every } t\leq \tau^1\wedge \tau^2, \quad \mathbb P-a.s.$$
\end{theorem}
\begin{proof}
 Suppose that $(n_1, {\bf u}_1, \tau_1)$ and $(n_2,  {\bf u}_2, \tau_2)$ are two local mild  solutions of system (\ref{PKSNS-time-dependent}).  
Set $\tau = \tau_1 \wedge \tau_2$.   Define
\begin{align}\label{settauiRmar8}
  \tau_i^R = \inf\{t \ge 0 : \|n_i\|_{X_n^t}  + \|{\bf u}_i\|_{X_{\bf u}^t} \ge R\} \wedge \tau_i, 
  \quad i = 1, 2,
\end{align}
and set $\tau_R = \tau_1^R \wedge \tau_2^R$.  Noting that  $\nabla \cdot {{\bf u}_i} = 0$, we have $\nabla \cdot ({\bf u}_i(t) n_i(t)) = {\bf u}_i(t) \cdot \nabla n_i(t)$ for $i = 1, 2$.  By using $n_i$-equation and ${\bf u}_i$-equation, we then obtain 
\begin{align*}
n_1(t)-n_2(t)=\int_0^te^{(t-s)\Delta}\nabla\cdot(n_1\nabla c_1-n_2\nabla c_2+n_1{\bf u}_1-n_2{\bf u}_2)(s)ds,
\end{align*}
and
\begin{align*}
{\bf u}_1(t)-{\bf u}_2(t)=&\int_0^te^{-(t-s)A}\nabla\cdot({\bf u}_1\otimes {\bf u}_1-{\bf u}_2\otimes {\bf u}_2+\nabla c_1\otimes \nabla c_1-\nabla c_2\otimes \nabla c_2)(s)ds\\
&+\int_0^t e^{-(t-s)A}[{\bf f}({\bf u}_1)-{\bf f}({\bf u}_2)](s)dW_s.
\end{align*}
We have from (\ref{eq:3.4a}), (\ref{eq:3.4b}) and (\ref{eq:3.4a1}) that 
\begin{align}\label{combine1mar8sun}
&\Vert n_1-n_2\Vert_{L^p}\nonumber\\
\lesssim&\int_0^t(t-s)^{-\frac{1}{2}+\frac{1}{p}-\frac{1}{q}}(\Vert n_1\nabla c_1-n_2\nabla c_1\Vert_{L^q}+\Vert n_2\nabla c_1-n_2\nabla c_2\Vert_{L^q}+ \Vert n_1{\bf u}_1-n_2{\bf u}_1\Vert_{L^q}+\Vert n_2{\bf u}_1-n_2{\bf u}_2\Vert_{L^q})ds\nonumber\\
\lesssim &\int_0^t(t-s)^{-\frac{1}{2}+\frac{1}{p}-\frac{1}{q}}(\Vert \nabla c_1\Vert_{L^\frac{pq}{p-q}}\Vert n_1-n_2\Vert_{L^p}+\Vert n_2\Vert_{L^p}\Vert \nabla c_1-\nabla c_2\Vert_{L^{\frac{pq}{p-q}}}\nonumber\\
&+\Vert u_1\Vert_\frac{pq}{p-q}\Vert n_1-n_2\Vert_{L^p}+\Vert n_2\Vert_{L^p}\Vert {\bf u}_1-{\bf u}_2\Vert_{L^{\frac{pq}{p-q}}})ds\nonumber\\
\lesssim &\int_0^t(t-s)^{-\frac{1}{2}+\frac{1}{p}-\frac{1}{q}}(C_R\Vert n_1-n_2\Vert_{L^p}+C_R\Vert {\bf u}_1-{\bf u}_2\Vert_{L^{\frac{pq}{p-q}}})ds,
\end{align}
where $C_R>0$ is a constant depending on $R$ given in (\ref{settauiRmar8}).  
In addition,
\begin{align}\label{uniquenumber1mar8}
\Vert {\bf u}_1-{\bf u}_2\Vert_{L^\frac{pq}{p-q}}\lesssim&\int_0^t(t-s)^{-\frac{1}{2}+\frac{1}{p}-\frac{1}{q}}(\Vert {\bf u}_1\otimes {\bf u}_1-{\bf u}_1\otimes {\bf u}_2\Vert_{L^{\frac{pq}{2(p-q)}}}+\Vert {\bf u}_1\otimes {\bf u}_2-{\bf u}_2\otimes {\bf u}_2\Vert_{L^{\frac{pq}{2(p-q)}}}\nonumber\\
&+\Vert \nabla c_1\otimes  \nabla c_1-\nabla c_1\otimes \nabla c_2\Vert_{L^{\frac{pq}{2(p-q)}}}+\Vert \nabla c_1\otimes \nabla c_2-\nabla c_2\otimes \nabla c_2\Vert_{L^{\frac{pq}{2(p-q)}}})ds\nonumber\\
&+\Vert \int_0^t e^{-(t-s)A}[{\bf f}({\bf u}_1)-{\bf f}({\bf u}_2)](s)dW_s\Vert_{L^{\frac{pq}{p-q}}}\nonumber\\
\lesssim &\int_0^t(t-s)^{-\frac{1}{2}+\frac{1}{p}-\frac{1}{q}}(\Vert {\bf u}_1\Vert_{L^{\frac{pq}{p-q}}}\Vert {\bf u}_1-{\bf u}_2\Vert_{L^{\frac{pq}{p-q}}}+\Vert {\bf u}_2\Vert_{L^{\frac{pq}{p-q}}}\Vert {\bf u}_1- {\bf u}_2\Vert_{L^{\frac{pq}{p-q}}}\nonumber\\
&+\Vert \nabla c_1\Vert_{L^{\frac{pq}{p-q}}}\Vert \nabla c_1-\nabla c_2\Vert_{L^{\frac{pq}{p-q}}}+\Vert \nabla c_2\Vert_{L^{\frac{pq}{p-q}}}\Vert\nabla c_1-\nabla c_2\Vert_{L^{\frac{pq}{p-q}}})ds\nonumber\\
\lesssim &\int_0^t(t-s)^{-\frac{1}{2}+\frac{1}{p}-\frac{1}{q}}(C_R\Vert n_1-n_2\Vert_{L^p}+C_R\Vert {\bf u}_1-{\bf u}_2\Vert_{L^{\frac{pq}{p-q}}})ds\nonumber\\
&+\Vert \int_0^t e^{-(t-s)A}[{\bf f}({\bf u}_1)-{\bf f}({\bf u}_2)](s)dW_s\Vert_{\frac{pq}{p-q}}\nonumber\\
:=&\int_0^t(t-s)^{-\frac{1}{2}+\frac{1}{p}-\frac{1}{q}}(C_R\Vert n_1-n_2\Vert_{L^p}+C_R\Vert {\bf u}_1-{\bf u}_2\Vert_{L^{\frac{pq}{p-q}}})ds
+\Vert \bar Z\Vert_{L^{\frac{pq}{p-q}}},
\end{align}
where
\begin{align*}
\bar Z(t):=\int_0^t e^{-(t-s)A}[{\bf f}({\bf u}_1)-{\bf f}({\bf u}_2)](s)dW_s.
\end{align*}
In light of assumption (H2), we use BDG's inequality and Young's inequality to get
\begin{align}\label{110newlabelnew11}
&\mathbb{E}\left( \sup_{t\leq T\wedge \tau_R} \|\bar Z(t)\|^{\frac{pq}{p-q}}_{L^{\frac{pq}{p-q}}} \right)
\leq C_{T\wedge \tau_R}[ \mathbb E(\int_0^{T\wedge \tau_R}\sup_{s\leq T\wedge \tau_R}\Vert {\bf u}_1-{\bf u}_2\Vert_{L^{\frac{pq}{p-q}}}^{\frac{pq}{p-q}}(s)ds)],
\end{align}
where $C_{T\wedge \tau_R}>0$ is a constant depending on $T\wedge\tau_R.$
Thus, we combine (\ref{uniquenumber1mar8}), (\ref{110newlabelnew11}) with (\ref{combine1mar8sun}) to get
\begin{align}\label{thusasmar8sun}
&\mathbb E(\sup_{t\leq T\wedge \tau_R}\Vert n_1-n_2\Vert_{L^p}^\frac{pq}{p-q}+\sup_{t\leq T\wedge \tau_R}\Vert {\bf u}_1-{\bf u}_2\Vert^{\frac{pq}{p-q}}_{L^{\frac{pq}{p-q}}})\nonumber\\
\leq &C_{R,T\wedge \tau_R}\int_0^{T\wedge\tau_R}[1+(t-s)^{-\frac{1}{2}+\frac{1}{p}-\frac{1}{q}}]\mathbb E(\sup_{r\le s} \|n_1-n_2\|_{L^p}^{\frac{pq}{p-q}}+\sup_{r\le s}  \|{\bf u}_1-{\bf u}_2\|^\frac{pq}{p-q}_{L^{\frac{pq}{p-q}}})(r)ds,
\end{align}
where  $C_{T\wedge \tau_R}>0$ is a constant depending on $T\wedge\tau_R.$
Define 
\begin{align}
\Theta(T):=\mathbb E\left(\sup_{t\leq T\wedge \tau_R}\Vert n_1-n_2\Vert_{L^p}^\frac{pq}{p-q}+\sup_{t\leq T\wedge \tau_R}\Vert {\bf u}_1-{\bf u}_2\Vert^{\frac{pq}{p-q}}_{L^{\frac{pq}{p-q}}}\right),
\end{align}
then  we rewrite (\ref{thusasmar8sun}) as
\begin{align}
\Theta(T)\lesssim C_{R,T\wedge\tau_R}\int_0^{T\wedge\tau_R}(t-s)^{-\frac{1}{2}+\frac{1}{p}-\frac{1}{q}}\Theta(s)ds+\int_0^{T\wedge \tau_R}\Theta(s)ds.
\end{align}
Then, by Gr\"{o}nwall inequality, one finds
$$\Theta(T)\equiv 0.$$
It follows that $(n_1,{\bf u}_1)=(n_2,{\bf u}_2)$ in $L^p(\mathbb R^2)\times L^{\frac{pq}{p-q}}(\mathbb R^2)$  { for every $t\leq \tau_R$} $\mathbb P$-a.s. {Now let,
$$\tau:= \limsup \tau_R,$$
}
Then, by letting $R\rightarrow\infty$ we obtain that
$$(n_1,{\bf u}_1)=(n_2,{\bf u}_2) \quad\text{ in } L^p(\mathbb R^2)\times L^{\frac{pq}{p-q}}(\mathbb R^2)\quad \text{ for every } t\leq \tau, \quad \mathbb P-a.s.$$

\end{proof}

\begin{theorem}\label{thm:regularity}
The unique local solution $(n_m,{\bf u}_m,\tau_m)$ in the sense of Definition \ref{def:local_solution} additionally satisfies the following regularity result:

For any $m\geq 1$,
\begin{itemize}
\item[(i)]  $n_m \in L^{2}(\Omega;C^1((0,\tau_m); C^2(\mathbb R^2 )))$ and ${\bf u}_m \in L^{2}(\Omega; L(0,\tau_m; W^{2,r}(\mathbb R^2)))$ for $r=2+\epsilon$ with $\epsilon>0$ sufficiently small. 
\item[(ii)]  $n_m(x,t) > 0$ for all $x\in \mathbb{R}^2$ and  $t<\tau_m$ $\mathbb P$-a.s.
\end{itemize}
\end{theorem}
\begin{proof}
Recall that we constructed $(n_m,\bu_m)|_{[0,\tau_m)}\in S_T$ and thus 
\begin{align}\label{base_estimate:nu}
    &\mathbb E\left(\sup_{t\leq T}\Vert n_m(t\wedge \tau_m)\Vert_{L^p}^{\frac{pq}{p-q}}\right)\leq C_m,\nonumber\\
    &\mathbb E\left(\sup_{t\leq T}\Vert {\bf u}_m(t\wedge \tau_m)\Vert_{L^{\frac{pq}{p-q}}}^{\frac{pq}{p-q}}+\sup_{t\leq T}\Vert \nabla {\bf u}_m(t\wedge \tau_m)\Vert_{L^r}^{\frac{pq}{p-q}}\right)\leq C_m ,
\end{align}
for some constant $C_m>0$ depending on $m.$   

Our next aim is to upgrade the spatial regularity of $(n_m,{\bf u}_m)$.   First, note that  $(n,{\bf u})$ satisfies \eqref{3p1mar} for any $t<\tau$ we have almost surely,  that
\begin{align}\label{massconserve}
    \| n(t) \|_{L^1} = \int_{\mathbb{R}^2} n(x,t) \, dx = \int_{\mathbb{R}^2} e^{t\Delta} n_0(x) \, dx = \int_{\mathbb{R}^2} n_0(x) \, dx.
\end{align}

In other words, we have that the cell mass is conserved for any time almost surely.

{Now, we continue to establish higher regularity, as claimed in Theorem \ref{thm:regularity}, of our maximal solution $(n_m,{\bf u}_m)$. More precisely, we will show the following bounds in the following order. Recall that $p,q,r$ are as defined in \eqref{pqr}. 

{\bf Claims:} First, by using (\ref{base_estimate:nu}), we shall prove that for $p$ and $q$ given in (\ref{pqr}),
\begin{align}\label{estimaten_La}
\mathbb E\left(\sup_{t\leq \tau_m}\|n_m(t)\|_{L^{a}}^\frac{pq}{2(p-q)}\right)\leq C_m,\qquad ~\mathbb E\left(\sup_{t\leq \tau_m}\|{\bf u}_m(t )\|_{L^\infty}^2\right)\leq C_m
\end{align}
for $\frac{2-\epsilon}{1-\epsilon}< a < \frac{2-\epsilon}{\epsilon}$ with $\epsilon>0$ sufficiently small and some positive constant $C$ depending on $m.$

Next, by combining \eqref{estimaten_La} with \eqref{base_estimate:nu} and applying the semigroup estimate given in Lemma \ref{lemma:semigroup-Lp-Lq}, we will prove that
\begin{align}\label{linfestimateof_gradient_n}
\mathbb E\left(\sup_{t\leq \tau_m}\Vert n_m(t )\Vert_{L^{\infty}}^\frac{pq}{4(p-q)}\right)\leq C_m.
\end{align}
Then, using \eqref{linfestimateof_gradient_n} and interpolation inequalities, we will prove that for any $\alpha\in(0,\frac{1}{2})$ and $\bar\alpha\in(\frac{1}{2},1),$
\begin{align}\label{A_alphabaralpha_u_r_estimate}
\mathbb E\left(\sup_{t\leq \tau_m} \Vert A^{\alpha} {\bf u}_m(t)\Vert_{L^{\frac{pq}{p-q}}}^{\frac{pq}{8(p-q)}}\right)\leq C_m,\qquad 
\mathbb E\left(\sup_{t\leq \tau_m}\Vert A^{\bar\alpha}{\bf u}_m(t) \Vert^{\frac{pq}{8(p-q)}}_{L^r}\right)\leq C_m. 
\end{align}
In addition, we will deduce from Lemma \ref{lemma:semigroup-Lp-Lq}, \eqref{linfestimateof_gradient_n} and interpolation, for $\alpha\in (0,\frac12)$ that
\begin{align}\label{Alaphanr_interpolation}
\mathbb E\left(\sup_{t\leq \tau_m}\|A^{\alpha}n_m\|^{\frac{pq}{8(p-q)}}_{L^{l}}\right)\le C_m,\qquad \text{ for } l=p,\infty
\end{align}
The bound given in \eqref{Alaphanr_interpolation}, will then be used to upgrade the bounds \eqref{A_alphabaralpha_u_r_estimate} with the aid of Lemma \ref{lemma:semigroup-Lp-Lq} to obtain that 
\begin{align}\label{laplace_estimate_u}
\mathbb E\left(\sup_{t\leq \tau_m}\Vert A {\bf u}_m(t)\Vert_{L^r}^{\frac{pq}{32(p-q)}}\right)\leq C_m.
\end{align}

Then, by the Sobolev embedding theorem, together with \eqref{laplace_estimate_u} and \eqref{A_alphabaralpha_u_r_estimate}, we will deduce that for $\alpha_1\in(0,1)$,  
\begin{align}\label{Abaralpha_alpha1_u_n_estimate}
\mathbb E\left(\sup_{t\leq \tau_m}\|A^{\alpha_1}{\bf u}_m(t)\|^{\frac{pq}{64(p-q)}}_{L^{\frac{pq}{p-q}}}\right)\leq C_m,\qquad \mathbb E\left(\sup_{t\leq \tau_m}\Vert A^{\alpha_1}n_m(t)\Vert_{L^\infty}^{\frac{pq}{16(p-q)}}\right)\leq C_m.
\end{align}
Finally, by combining (\ref{A_alphabaralpha_u_r_estimate}) and (\ref{Abaralpha_alpha1_u_n_estimate}) we will arrive at the following bounds for any $\tilde k\in (0,\frac{3}2)$: 
$$ \mathbb E\left(\sup_{t\leq \tau_m}\|A^{\tilde k}n_m(t)\|^2_{L^\infty}\right)\leq C_m.$$
} 

We begin with \eqref{estimaten_La}.  In light of (\ref{eq220}) and (\ref{eq221}), using (\ref{eq:3.3}), we have
\begin{align}\label{npestimatenov}
\|n_m(t)\|_{L^a}
&\lesssim\, \|e^{t\Delta}n_m(0)\|_{L^a}
   +  \int_{0}^{t} (t-s)^{-\frac{1}{2} +\frac{1}{a}- \frac{1}{q}} 
   \| n_m(s)\nabla c_m(s) + n_m(s){\bf u}_m(s) \|_{L^q}\, ds \notag \\
& \lesssim\,\|n_m(0)\|_{L^a}
   +   \int_{0}^{t} (t-s)^{-\frac{1}{2}+\frac{1}{a} - \frac{1}{q}} 
   \big( \|n_m(s)\|_{L^p}^2 + \|n_m(s)\|_{L^p}\|{\bf u}_m(s)\|_{L^{\frac{pq}{p-q}}} \big)\, ds,
\end{align}
for $\frac{2-\epsilon}{1-\epsilon}< a < \frac{2-\epsilon}{\epsilon}$ with $\epsilon>0$ sufficiently small.  Thus, 
\begin{align}\label{nmar8bound}
\mathbb E\sup_{t\leq \tau_m}\|n_m(t)\|_{L^{a}}^\frac{pq}{2(p-q)}
&\lesssim \, \sup_{t\leq \tau_m}\|e^{t\Delta}n_m(0)\|^\frac{pq}{2(p-q)}_{L^a}+C\mathbb E
   \big(\sup_{t\leq \tau_m}\|n_m(t)\|^\frac{pq}{p-q}_{L^p}+\sup_{t\leq \tau_m}\|{\bf u}_m(t)\|^\frac{pq}{p-q}_{L^{\frac{pq}{p-q}}} \big)\lesssim C_m,
\end{align}
where we used \eqref{base_estimate:nu}.

By the semigroup estimate (\ref{eq:3.4a})-(\ref{eq:3.4b}), one finds 
\begin{align}\label{npestimatenov}
\|n_m(t)\|_{L^\infty}
&\lesssim \, \|e^{t\Delta}n_m(0)\|_{L^\infty}
   +  \int_{0}^{t} (t-s)^{-\frac{1}{2} - \frac{pq+(p-q)a}{apq}} 
   \| n_m(s)\nabla c_m(s) + n_m(s){\bf u}_m(s) \|_{L^{\frac{apq}{pq+(p-q)a}}}\, ds \notag \\
& \lesssim \,\|n_m(0)\|_{L^\infty}
   + \int_{0}^{t} (t-s)^{-\frac{1}{2}- \frac{pq+(p-q)a}{apq}} 
   \bigg( \|n_m(s)\|_{L^a}\Vert n_m\Vert_{L^p} + \|n_m(s)\|_{L^a}\|{\bf u}_m(s)\|_{L^{\frac{pq}{p-q}}} \bigg)\, ds
   \end{align},
   where $\epsilon\in(0,\frac{1}{2})$ is chosen so that $-\frac{1}{2}-\frac{pq+(p-q)a}{apq}>-1.$
   Thus,
   \begin{align*}
\Vert n_m\Vert_{L^\infty}^\frac{pq}{4(p-q)}   & \lesssim \,\|n_m(0)\|^{\frac{pq}{4(p-q)}}_{L^\infty}
   + C_m\big( \sup_{s\leq \tau_m}\|n_m(s)\|^\frac{pq}{2(p-q)}_{L^a}+\sup_{s\leq \tau_m}\Vert n_m\Vert^\frac{pq}{2(p-q)}_{L^p} +\sup_{s\leq \tau_m}\|{\bf u}_m(s)\|^\frac{pq}{2(p-q)}_{L^{\frac{pq}{p-q}}} \big)  .
\end{align*}
It follows from (\ref{nmar8bound}) that
\begin{align}\label{nlinfboundmar}
\mathbb E\left(\sup_{t\leq T}\Vert n_m(t\wedge \tau_m)\Vert_{L^{\infty}}^\frac{pq}{4(p-q)}\right)\leq C_m,
\end{align}
where constant $C_m>0$ depends on $m$ and by definition we have $\frac{pq}{4(p-q)}=\frac{1}{\epsilon}-\frac12$ for $\epsilon>0$ chosen to be sufficiently small.

Next, we apply interpolation inequality,
\begin{align*}
\Vert {\bf u}_m\Vert^{2}_{L^\infty}\lesssim \Vert\nabla {\bf u}_m\Vert^{2\theta}_{L^r}\Vert {\bf u}_m\Vert^{2(1-\theta)}_{L^{\frac{pq}{p-q}}}\leq\Vert \nabla {\bf u}_m\Vert_{L^r}^r+\Vert {\bf u}_m\Vert^{\frac{2(1-\theta)r}{r-2\theta}}_{L^{\frac{pq}{p-q}}}\leq \Vert \nabla {\bf u}_m\Vert_{L^r}^r+\Vert {\bf u}_m\Vert^{2+\epsilon}_{L^{\frac{pq}{p-q}}},
\end{align*}
for
$$0<\theta=\frac{\frac{p-q}{pq}}{\frac{1}{2}+\frac{p-q}{pq}-\frac{1}{r}}=\frac{2+\epsilon}{4}=\frac{1}{2}+\frac{\epsilon}{4}<1,$$
for $\epsilon>0$ sufficiently small.  
Then it follows from $\mathbb E(\sup_{t\leq \tau_m}\Vert \nabla {\bf u}\Vert^{\frac{pq}{p-q}}_{L^{r}})\leq C_m$  and   \eqref{base_estimate:nu} that
$$\mathbb E\left(\sup_{t\leq T}\|{\bf u}_m(t\wedge\tau_m)\|_{L^\infty}^2\right)\leq C_m,$$
where $C_m>0$ depending on $m$ is a positive constant.

Next we invoke (\ref{eq:3.4afractional}) and (\ref{eq:3.4bfractional}) in Lemma \ref{lemma:semigroup-Lp-Lq} to get, for any $\alpha\in(0,\frac{1}{2}),$ 
\begin{align}\label{130nov12mar}
&\|A^{\alpha} {\bf u}_m(t)\|_{L^{\frac{pq}{p-q}}}
\lesssim\, \|e^{-tA}A^{\alpha} {\bf u}_m(0)\|_{L^{\frac{pq}{p-q}}}
   +  \int_{0}^{t} (t-s)^{-{\alpha}+\frac{1}{q}-\frac{1}{p}-\frac{1}{b} }
   \|n_m(s) \nabla c_m(s) + {\bf u}_m(s)\cdot\nabla  {\bf u}_m(s)\|_{L^b}\, ds \notag\\ 
   &+\Vert \int_0^te^{-(t-s)A}\mathcal PA^{\alpha}{\bf f}({\bf u}_m(s))dW_s\Vert_{L^{\frac{pq}{p-q}}}\nonumber\\
&\lesssim\, \|\nabla {\bf u}_m(0)\|_{L^{\frac{pq}{p-q}}}
   +  \int_{0}^{t} (t-s)^{-{\alpha}+\frac{1}{q}-\frac{1}{p}-\frac{1}{b}}
   \big( \|n_m(s)\|_{L^r}\Vert \nabla c_m(s)\Vert_{L^{\frac{pq}{p-q}}} +\Vert \nabla {\bf u}_m(s)\Vert_{L^r} \|{\bf u}_m(s)\|_{L^{\frac{pq}{p-q}}} \big)\, ds\nonumber\\
   &+\Vert \int_0^te^{-(t-s)A}\mathcal PA^{\alpha}{\bf f}({\bf u}_m(s))dW_s\Vert_{L^{ \frac{pq}{p-q}}},
\end{align}
where 
\begin{align}\label{exponent_b}
\frac{1}{b}=\frac{1}{2+\epsilon}+\frac{\epsilon}{4-2\epsilon}>\frac{1}{2},\text{ for an appropriately small }\epsilon>0.
\end{align}
Note that for $r=2+\epsilon,$ we have the interpolation inequality $\Vert n_m\Vert_{L^r}\lesssim \Vert n_m\Vert_{L^\infty}^{\frac{r-1}{r}}\Vert n_m\Vert^r_{L^1}$. Recall also that $\Vert n_m\Vert_{L^1}=\Vert n_m(0)\Vert_{L^1}$ is shown in \eqref{massconserve}. We thus obtain,
\begin{align}\label{nrboundestimate}
\Vert n_m\Vert^{\frac{pq}{4(p-q)}}_{L^r}\lesssim \Vert n_m\Vert_{L^\infty}^{\frac{(r-1)pq}{4r(p-q)}}.
\end{align}
Hence, for (\ref{130nov12mar}) we obtain for some constant $C_m$ that
\begin{align}\label{regularity_Aalphaum_local}
\Vert A^{\alpha} {\bf u}_m\Vert_{L^{\frac{pq}{p-q}}}^{\frac{pq}{8(p-q)}}\lesssim & C_m(\Vert n_m\Vert^{{\frac{pq}{4(p-q)}}}_{L^r}+\Vert n_m\Vert^{\frac{pq}{4(p-q)}}_{L^p})+C_m[||\nabla {\bf u}_m\Vert^{{\frac{pq}{4(p-q)}}}_{L^r}+\Vert {\bf u}_m\Vert^{\frac{pq}{4(p-q)}}_{L^{\frac{pq}{p-q}}}]\nonumber\\
&+\Vert \int_0^te^{-(t-s)A}\mathcal PA^{\alpha}{\bf f}({\bf u}_m(s))dW_s\Vert_{L^{\frac{pq}{p-q}}}^{\frac{pq}{8(p-q)}}.
\end{align}
We apply $\sup_{t\leq\tau_m}$ and then expectation on both sides of the inequality above.  By choosing $\bar q=\frac{pq}{2(p-q)}$ and $2\bar p\bar q=\frac{pq}{8(p-q)}$ in \eqref{110newlabelnewdec}, we have the estimate of the last term in (\ref{regularity_Aalphaum_local}).  Combining this with (\ref{nlinfboundmar}) and (\ref{nrboundestimate}), we further obtain, {for any $\alpha\in(0,\frac{1}{2}),$ and $p,q$ satisfying \eqref{pqr}-\eqref{quotient}} that,
\begin{align}\label{alpha1emar8}
\mathbb E\left(\sup_{t\leq \tau_m} \Vert A^{\alpha} {\bf u}_m(t)\Vert_{L^{\frac{pq}{p-q}}}^{\frac{pq}{8(p-q)}}\right)\leq C_m,
\end{align}
for some constant $C_m>0$ depending on $m.$

Next, for $\bar\alpha\in(1/2,1),$  we consider
\begin{align}\label{130nov12marnow}
\|A^{\bar\alpha}{\bf u}_m(t)\|_{L^r}
&\lesssim \, \|e^{-tA}A^{\bar\alpha}{\bf u}_m(0)\|_{L^r}
   +  \int_{0}^{t} (t-s)^{-{\bar\alpha} +\frac{1}{r} - \frac{1}{b}}
   \|n_m(s) \nabla c_m(s) + {\bf u}_m(s)\cdot\nabla  {\bf u}_m(s)\|_{L^b}\, ds \notag\\ 
   &\quad +\Vert \int_0^te^{-(t-s)A}\mathcal PA^{\bar\alpha}{\bf f}({\bf u}_m(s))dW_s\Vert_{r}\nonumber\\
&\lesssim \, \|A^{\bar\alpha}{\bf u}_m(0)\|_{L^r}
   +  \int_{0}^{t} (t-s)^{-{\bar\alpha} +\frac{1}{r} - \frac{1}{b}}
   \big( \|n_m(s)\|_{L^r}\|n_m(s)\|_{p} +\Vert \nabla {\bf u}_m(s)\Vert_{r} \|{\bf u}_m(s)\|_{\frac{pq}{p-q}} \big)\, ds\nonumber\\
   &\quad +\Vert \int_0^te^{-(t-s)A}\mathcal PA^{\bar\alpha}{\bf f}({\bf u}_m(s))dW_s\Vert_{L^r}.
\end{align}
By choosing $\bar q=\frac{r}{2}$ and $2\bar p\bar q=\frac{pq}{8(p-q)}$ in \eqref{110newlabelnewdec}, we obtain from (\ref{nlinfboundmar}), (\ref{nrboundestimate}) and (\ref{130nov12marnow}), {for any $\bar\alpha\in(\frac{1}{2},1),$ and $r$ given in \eqref{pqr}} that 
\begin{align}\label{Aulrmar8}
\mathbb E\left(\sup_{t\leq \tau_m}\Vert A^{\bar\alpha}{\bf u}_m(t) \Vert^{\frac{pq}{8(p-q)}}_{L^r}\right)\lesssim C_m,
\end{align}
where $C_m>0$ depends on $m$.

 To upgrade the bounds above and find estimates for $A{\bf u}$, we first estimate $n$. We consider for { $k =p, \infty$} and $\alpha\in(0,\frac{1}{2}),$ 
\begin{align}
&\sup_{t\leq \tau_m}\|A^{\alpha} n_m(t)\|^\frac{pq}{8(p-q)}_{L^k}
\lesssim  \Vert A^{\alpha}n_m(0)\Vert_{L^{k}}^\frac{pq}{8(p-q)}\nonumber\\
+ &\sup_{t\leq \tau_m}\bigg(\int_{0}^{t} (t-s)^{-{\alpha}-\frac{1}{2} - \frac{1}{q}+\frac{1}{p}} 
   \big( \|n_m(s)\|_{L^k}\Vert\nabla c_m\Vert_{L^{\frac{pq}{p-q}}} + \|n_m(s)\|_{L^k}\|{\bf u}_m(s)\|_{L^{\frac{pq}{p-q}}} \big)\, ds\bigg)^\frac{pq}{8(p-q)}\nonumber\\
\lesssim&\Vert A^{\alpha}n_m(0)\Vert^\frac{pq}{8(p-q)}_{L^k}\nonumber\\
& +
 \sup_{t\leq \tau_m}\left(\|n_m(s)\|_{L^k}\|n_m(s)\|_{L^p} + \|n_m(s)\|_{L^p}\|{\bf u}_m(s)\|_{L^{\frac{pq}{p-q}}}\right)^\frac{pq}{8(p-q)}\bigg(\int_0^t (t-s)^{-{\alpha}-\frac{1}{2} - \frac{1}{q}+\frac{1}{p}} \, ds\bigg)^\frac{pq}{8(p-q)}\nonumber\\
 \lesssim &\Vert A^{\alpha}n_m(0)\Vert^\frac{pq}{8(p-q)}_{L^k}+C_m\left(\sup_{t\leq \tau_m}\Vert n_m(t)\Vert_{L^{k}}^{\frac{pq}{4(p-q)}}+\sup_{t\leq \tau_m}\Vert n_m(t)\Vert_{L^p}^{\frac{pq}{4(p-q)}}\right)\nonumber\\
 &+C_m\left(\sup_{t\leq \tau_m}\Vert n_m(t)\Vert^{\frac{pq}{4(p-q)}}_{L^p}
 +\sup_{t\leq \tau_m}\Vert {\bf u}_m(t)\Vert^{\frac{pq}{4(p-q)}}_{L^{\frac{pq}{p-q}}}\right).
\end{align}
Then it follows from \eqref{base_estimate:nu} that, for {any $\alpha\in(0,\frac12)$, we have}
\begin{align}\label{Aalnlpmar8}
   \mathbb E\left( \sup_{t\leq \tau_m}\|A^{\alpha} n_m(t)\|^\frac{pq}{8(p-q)}_{L^k} \right)
\lesssim C_m,~k=p,\infty.
\end{align}
We observe an immediate consequence of the bound above by using the following interpolation inequality:
\begin{align*}
\|A^{\alpha}n_m\|^{\frac{pq}{16(p-q)}}_{L^{r}}
\le
\|A^{\alpha}n_m\|_{L^\infty}^{\big(1-\frac{p}{r}\big)\frac{pq}{16(p-q)}}
\|A^{\alpha}n_m\|_{L^{p}}^{\frac{p^2q}{16r(p-q)}},
\end{align*}
In other words, for some constant $C_m>0$ depending on $m$, we have for any $\alpha\in(0,\frac12)$ that
\begin{align*}
\mathbb E\left(\sup_{t\leq \tau_m}\|A^{\alpha}n_m\|^{\frac{pq}{16(p-q)}}_{L^{r}}\right)
\le
C_m.
\end{align*}

Finally, for $A\bu$ we find the following estimate. 
We apply the semigroup estimates (\ref{eq:3.4afractional}) and (\ref{eq:3.4bfractional}) to write
\begin{align}\label{130nov12mar8}
\|A{\bf u}_m(t)\|_{L^r}
\lesssim& \, \|A{\bf u}_m(0)\|_{L^r}
   +  \int_{0}^{t} (t-s)^{-{\frac34} +\frac{1}{r} - \frac{1}{b}}
   \big( \|A^{\frac14}n_m(s)\|_{L^r}\|n_m(s)\|_{L^p}+\|n_m(s)\|_{L^r}\|A^{\frac14}n_m(s)\|_{L^p} \nonumber\\
  & +\Vert \nabla A^{\frac14}{\bf u}_m(s)\Vert_{L^r} \|{\bf u}_m(s)\|_{L^{\frac{pq}{p-q}}}+\Vert \nabla {\bf u}_m(s)\Vert_{L^r} \|A^{\frac14}{\bf u}_m(s)\|_{L^{\frac{pq}{p-q}}} \big)\, ds\nonumber\\
   &+\Vert \int_0^te^{-(t-s) A}\mathcal PA{\bf f}({\bf u}_m(s))dW_s\Vert_{L^r},
\end{align}
where $b$ is given in (\ref{exponent_b}).   Choosing $\bar q=\frac{r}{2}$ and $2\bar p\bar q=\frac{pq}{32(p-q)}$ in \eqref{110newlabelnewdec}, we take $\alpha=\frac14$ in (\ref{alpha1emar8}) and (\ref{Aalnlpmar8})   and $ \bar\alpha=\frac34$ in (\ref{Aulrmar8}), then combine them with (\ref{130nov12mar8}) to obtain, for $r$ defined in \eqref{pqr}, that
\begin{align}\label{uritself}
\mathbb E\left(\sup_{t\leq \tau_m}\Vert A {\bf u}_m(t)\Vert_{L^r}^{\frac{pq}{32(p-q)}}\right)\leq C_m,
\end{align}
where $C_m>0$ is a constant depending on $m.$  Moreover, we use Sobolev embedding theorem and the estimate (\ref{uritself}) to obtain
\begin{align}\label{gradentlinfbound}
\mathbb E(\sup_{t\leq \tau_m}\Vert \nabla {\bf u}_m\Vert_{L^\infty}^{\frac{pq}{32(p-q)}})\leq C_m,
\end{align}
for some constant $C_m>0.$  Then, we consider for $\bar\alpha\in(1/2,1),$
\begin{align}\label{130nov12marnow}
\|A^{\bar\alpha}{\bf u}_m(t)\|_{L^{\frac{pq}{p-q}}}
&\lesssim \, \|e^{-tA}A^{\bar\alpha}{\bf u}_m(0)\|_{L^{\frac{pq}{p-q}}}
   +  \int_{0}^{t} (t-s)^{-{\bar\alpha}}
   \|n_m(s) \nabla c_m(s) + {\bf u}_m(s)\cdot\nabla  {\bf u}_m(s)\|_{L^{\frac{pq}{p-q}}}\, ds \notag\\ 
   &\quad +\Vert \int_0^te^{-(t-s)A}\mathcal PA^{\bar\alpha}{\bf f}({\bf u}_m(s))dW_s\Vert_{L^{\frac{pq}{p-q}}}\nonumber\\
&\lesssim \, \|A^{\bar\alpha}{\bf u}_m(0)\|_{L^{\frac{pq}{p-q}}}
   +  \int_{0}^{t} (t-s)^{-{\bar\alpha} }
   \big( \|n_m(s)\|_{L^\infty}\|n_m(s)\|_{L^p} +\Vert \nabla {\bf u}_m(s)\Vert_{L^{\infty}} \|{\bf u}_m(s)\|_{L^\frac{pq}{p-q}} \big)\, ds\nonumber\\
   &\quad +\Vert \int_0^te^{-(t-s)A}\mathcal PA^{\bar\alpha}{\bf f}({\bf u}_m(s))dW_s\Vert_{L^\frac{pq}{p-q}}.
\end{align}
{By choosing $\bar q=\frac{pq}{2(p-q)}$ and $2\bar p\bar q=\frac{pq}{64(p-q)}$ in \eqref{110newlabelnewdec} and using \eqref{nlinfboundmar}, \eqref{base_estimate:nu} and \eqref{gradentlinfbound}}, we have for any $\bar\alpha\in(\frac{1}{2},1),$
\begin{align}\label{Abaralphaestimatelower}
\mathbb E(\sup_{t\leq T}\|A^{\bar\alpha}{\bf u}_m(t\wedge\tau_m )\|^{\frac{pq}{64(p-q)}}_{L^{\frac{pq}{p-q}}})\leq C_m,
\end{align}
where $C_m$ is a positive constant depending on $m.$

Next, using this higher regularity of $\bu$, we find higher order estimates for $n$. Firstly, for any $\alpha_1\in(0,1)$ we write $\alpha_1=k+\alpha$ for some $k,\alpha \in(0,1/2),$ and observe 
\begin{align}\label{365mar8}
&\sup_{t\leq \tau_m}\|A^{ \alpha_1} n_m(t)\|^\frac{pq}{16(p-q)}_{L^\infty }\lesssim  \Vert A^{k+\alpha}n_m(0)\Vert_{L^\infty}^\frac{pq}{16(p-q)}+ \sup_{t\leq \tau_m}
\bigg[\int_{0}^{t} (t-s)^{-{k} - \frac{1}{q}+\frac{1}{p}} 
   \bigg( \|A^{\alpha} n_m(s)\|_{L^\infty}\Vert\nabla c_m\Vert_{L^{\frac{pq}{p-q}}}\nonumber\\
   &+\Vert  n_m(s)\|_{L^\infty}\Vert A^{\alpha}\nabla c_m\Vert_{L^{\frac{pq}{p-q}}}
   + \|A^{\alpha}n_m(s)\|_{L^\infty}\|{\bf u}_m(s)\|_{L^{\frac{pq}{p-q}}}+\|n_m(s)\|_{L^\infty}\|A^{\alpha} {\bf u}_m(s)\|_{L^{\frac{pq}{p-q}}} \bigg)\, ds\bigg]^\frac{pq}{16(p-q)}\nonumber\\
\lesssim&\Vert A^{k+\alpha}n_m(0)\Vert^\frac{pq}{16(p-q)}_{L^\infty}
 +\sup_{t\leq \tau_m}\bigg[ \|A^{\alpha} n_m(s)\|_{L^\infty}\|n_m(s)\|_{L^p}+ \|n_m(s)\|_{L^\infty}\|A^{\alpha} n_m(s)\|_{L^p}\nonumber\\
 &+ \|A^{\alpha}n_m(s)\|_{L^\infty}\|{\bf u}_m(s)\|_{L^{\frac{pq}{p-q}}}+ \|n_m(s)\|_{L^\infty}\|A^{\alpha}{\bf u}_m(s)\|_{L^{\frac{qp}{p-q}}}\bigg]^\frac{pq}{16(p-q)}\bigg(\int_0^t (t-s)^{-{k} - \frac{1}{q}+\frac{1}{p}} \, ds\bigg)^\frac{pq}{16(p-q)}\nonumber\\
 \lesssim &\Vert A^{k+\alpha}n_m(0)\Vert^\frac{pq}{16(p-q)}_{L^\infty}+C_m\left(\sup_{t\leq \tau_m}\Vert A^{\alpha}n_m(t)\Vert_{L^\infty}^{\frac{pq}{8(p-q)}}+\sup_{t\leq \tau_m}\Vert n_m(t)\Vert_{L^p}^{\frac{pq}{8(p-q)}}\right)\nonumber\\
 &+C_m\left(\sup_{t\leq \tau_m}\Vert n_m(t)\Vert_{L^\infty}^{\frac{pq}{8(p-q)}}+\sup_{t\leq \tau_m}\Vert A^{\alpha}n_m(t)\Vert_{L^p}^{\frac{pq}{8(p-q)}}\right)\nonumber\\
 &+C_m\left(\sup_{t\leq \tau_m}\Vert A^{\alpha}n_m(t)\Vert^{\frac{pq}{8(p-q)}}_{L^\infty}
 +\sup_{t\leq \tau_m}\Vert {\bf u}_m(t)\Vert^{\frac{pq}{8(p-q)}}_{L^{\frac{pq}{p-q}}}+\sup_{t\leq \tau_m}\Vert n_m(t)\Vert^{\frac{pq}{8(p-q)}}_{L^\infty}
 +\sup_{t\leq \tau_m}\Vert A^{\alpha}{\bf u}_m(t)\Vert^{\frac{pq}{8(p-q)}}_{L^{\frac{pq}{p-q}}}\right).
\end{align}
Since ${\bf{u}}_m\in L^{\frac{pq}{p-q}}(\Omega;L^\infty((0,\tau_m);L^{\frac{pq}{p-q}}(\mathbb R^2)))$, by using  (\ref{130nov12mar}), (\ref{alpha1emar8}) and (\ref{365mar8}), we find for any $\alpha_1\in (0,1)$ that
\begin{align}\label{Ahalfnestimatmar8}
    \mathbb E\left(\sup_{t\leq \tau_m}\Vert A^{\alpha_1}n_m(t)\Vert_{L^\infty}^{\frac{pq}{16(p-q)}}\right)\lesssim C_m
\end{align}
for some positive constant $C_m$ depending on $m.$

Next, we will bootstrap the estimate above to obtain the following higher order bounds.   

For any $\tilde k\in (0,\frac32)$ we write $\tilde k = k+\alpha_1$ for some $k\in (0,\frac12)$ and $\alpha_1\in (0,1),$ then apply  Lemma \ref{lemma:semigroup-Lp-Lq} to get
	\begin{align}\label{3p80highorder_regularity_An}
		\sup_{t\leq \tau_m}&\|A^{\tilde k} n_m(t)\|^\frac{pq}{128(p-q)}_{L^\infty} \lesssim \Vert A^{k+\alpha_1}n_m(0)\Vert_{L^\infty}^\frac{pq}{128(p-q)}\nonumber\\
		&+ \sup_{t\leq \tau_m}\bigg(\int_{0}^{t} (t-s)^{-{k} - \frac{1}{q}+\frac{1}{p}} 
		\big( \|A^{\alpha_1} n_m(s)\|_{L^\infty}\Vert\nabla c_m\Vert_{L^{\frac{pq}{p-q}}}+\Vert  n_m(s)\|_{L^\infty}\Vert A^{\alpha_1}\nabla c_m\Vert_{L^{\frac{pq}{p-q}}}\nonumber\\
		& + \|A^{\alpha_1}n_m(s)\|_{L^\infty}\|{\bf u}_m(s)\|_{L^{\frac{pq}{p-q}}}+\|n_m(s)\|_{L^\infty}\|A^{\alpha_1} {\bf u}_m(s)\|_{L^{\frac{pq}{p-q}}} \big)\, ds\bigg)^\frac{pq}{128(p-q)}\nonumber\\
		\lesssim&\Vert A^{\alpha_1}n_m(0)\Vert^\frac{pq}{128(p-q)}_{L^\infty}
		+[ \sup_{t\leq \tau_m}(\|A^{\alpha_1} n_m(s)\|_{L^\infty}\|n_m(s)\|_{L^p})+ \sup_{t\leq \tau_m}(\|n_m(s)\|_{L^\infty}\|A^{\alpha_1} n_m(s)\|_{L^p})\nonumber\\
		&+ \sup_{t\leq \tau_m}(\|A^{\alpha_1}n_m(s)\|_{L^\infty}\|{\bf u}_m(s)\|_{L^{\frac{pq}{p-q}}})\nonumber\\
		&+ \sup_{t\leq \tau_m}(\|n_m(s)\|_{L^\infty}\|A^{\alpha_1}{\bf u}_m(s)\|_{L^{\frac{qp}{p-q}}})]^\frac{pq}{128(p-q)}\bigg(\int_0^t (t-s)^{-{k} - \frac{1}{q}+\frac{1}{p}} \, ds\bigg)^\frac{pq}{128(p-q)}\nonumber\\
		\lesssim &\Vert A^{\alpha_1}n_m(0)\Vert^\frac{pq}{128(p-q)}_{L^\infty}+C_m\sup_{t\leq \tau_m}(\Vert A^{\alpha_1}n_m(t)\Vert_{L^\infty}^{\frac{pq}{64(p-q)}}+\Vert n_m(t)\Vert_{L^p}^{\frac{pq}{64(p-q)}})\nonumber\\
		&+C_m\sup_{t\leq \tau_m}(\Vert n_m(t)\Vert_{L^\infty}^{\frac{pq}{64(p-q)}}+\Vert A^{\alpha_1}n_m(t)\Vert_{L^p}^{\frac{pq}{64(p-q)}})\nonumber\\
		&+C_m\sup_{t\leq \tau_m}\left(\Vert A^{\alpha_1}n_m(t)\Vert^{\frac{pq}{64(p-q)}}_{L^\infty}
		+\Vert {\bf u}_m(t)\Vert^{\frac{pq}{64(p-q)}}_{L^{\frac{pq}{p-q}}}+\Vert n_m(t)\Vert^{\frac{pq}{64(p-q)}}_{L^\infty}
		+\Vert A^{\alpha_1}{\bf u}_m(t)\Vert^{\frac{pq}{64(p-q)}}_{L^{\frac{pq}{p-q}}}\right),
	\end{align}
    where we picked $\alpha_1=\bar \alpha$ in (\ref{Abaralphaestimatelower}).  Here, we choose $\epsilon\in(0,\frac{2}{129})$ so that $\frac{pq}{128(p-q)}>2$.  
Hence, we deduce from \eqref{Ahalfnestimatmar8}, \eqref{3p80highorder_regularity_An}, (\ref{Aalnlpmar8}) and (\ref{Abaralphaestimatelower}) that for any $\tilde k\in(0,\frac{3}{2})$,
\begin{align}
    \label{anMAR8end}
    \mathbb E\left(\sup_{t\leq T}\|A^{\tilde k}n_m(t\wedge\tau_m)\|^2_{L^\infty}\right)\lesssim C_m,
\end{align}
where constant $C_m>0$ depending on $m.$

In summary, we have from (\ref{nlinfboundmar}), (\ref{Ahalfnestimatmar8}) and (\ref{anMAR8end}) that $n_m\in L^2(\Omega; L^\infty((0,\tau_m);W^{2,\infty}(\mathbb R^2))). $  Then, by the parabolic estimate (See Theorems 11 and 16 in Chapter 1 of \cite{friedman2008partial}), we obtain $n_m$ satisfies
$$n_m\in L^2(\Omega;C^1((0,\tau_m);C^2(\mathbb R^2))).$$
Additionally, we have from (\ref{uritself}) and ${\bf u}_m\in L^{\frac{pq}{p-q}}(\Omega;L^\infty(0,\tau_m),\dot{W}^{1,r}(\mathbb R^2)))$ that 
\begin{align}
    {\bf u}_m\in L^2(\Omega; L^\infty((0,\tau_m); W^{2,r}(\mathbb R^2))),
\end{align}
where by definition \eqref{pqr}, we have $r=2+\epsilon$ for $\epsilon>0$ chosen to be sufficiently small.  This completes the proof of statement (i) of Theorem \ref{thm:regularity}.

Next, statement (ii) of Theorem \ref{thm:regularity} is obtained by a standard application of the maximum principle. In other words, we obtain strict positivity of the cell density $n>0$ for any $x\in\mathbb R^2$ and $t<\tau$ $\mathbb P$-a.s.    This finishes the proof of Theorem \ref{thm:regularity}.

\end{proof}

\section{Global-in-time Existence}\label{sect4}
In this section, we are concerned with the global-in-time well-posedness of (\ref{PKSNS-time-dependent}) under the subcritical cellular mass.  For that purpose, we first discuss the existence of the maximal local mild solution of (\ref{PKSNS-time-dependent}) in the sense of Definition \ref{def:local_solution}.
Recall that for stopping times $\{\tau_m,\, m \in \mathbb{N}\}$ defined in (\ref{eq:tau_m}), we constructed local mild solution $(n_m,  {\bf u}_m,\tau_m)$  in \eqref{localsol} such that
$$(n_m,\bu_m)|_{(0,\tau_m)}\in S_T.$$

Note that $\tau_m \le \tau_{m+1}$ almost surely. Hence, thanks to the uniqueness result proved in Theorem \ref{thm:unique}, we infer that
\[
(n_{m+1}, {\bf u}_{m+1}) = (n_m,{\bf u}_m) \quad \text{on } [0, \tau_m).
\]
Now we let,
\begin{align}\label{def:tau}
    \tau := \lim_{m \to \infty} \tau_m,
\end{align}
and define $(n, {\bf u})$ on $[0, \tau)$ by
\[
(n,  {\bf u}) = (n_m, {\bf u}_m) \quad \text{on } t \in [0, \tau_m).
\]
Since
\[
\|n_m\|_{X_n^{\tau_m}}  \vee \|u_m\|_{X_{\bf u}^{\tau_m}} \ge m 
\quad \text{on } \{\omega: \tau < \infty\},
\]
we have
\begin{align}\label{blowup_criterion}
\limsup_{t\to\tau} \sup_{s\in(0,t)} \Big(\|n(s)\|_{L^p} + \|\bu(s)\|_{L^{\frac{pq}{p-q}}}+\|\nabla\bu(s)\|_{L^r}\Big)=\infty,\quad\bP-a.s.
\end{align}
Therefore, $(n, {{\bf  u}}, \tau)$ is a maximal local solution of system (\ref{PKSNS-time-dependent}). 
Hence, our goal of extending this maximal solution for all times means that we must show that  $\tau=\infty$ almost surely. To show this, we will prove boundedness of the norms above.

First we want to derive an entropy estimate for the positive part of the entropy.
{However, due to the unboundedness of the domain, traditional estimates of free energy applied in the current set up do not provide us the boundedness of the desired entity.}  
For that purpose, we introduce a modification of the free energy $\mathcal{E}_{\Gamma}$ of \eqref{freeenergy}, in the spirit of \cite{SimingHe}, for estimating cellular density $n$ as follows:
\begin{equation}\label{modified}
\mathcal{E}_{\Gamma}[n,{\bf u}]
=\int_{\mathbb R^2} \left( n\,\Gamma(n)-\frac{nc}{2}+\frac{|{\bf u}|^{2}}{2} \right)\,dx ,
\end{equation}
where, for appropriately small constants $\eta,\delta>0$, the function $\Gamma$ is defined as
\begin{equation}\label{gamman}
\Gamma(n)=
\begin{cases}
\ln n, & n\ge \eta,\\[0.4em]
\ln  \eta+\eta^{-1}(n-\eta)-\dfrac{\eta^{-2}}{2}(n-\eta)^{2}, & n<\eta,
\end{cases}
\qquad
\eta:=\eta(\delta,M)=\min\left\{1,\dfrac{\delta}{M}\right\}.
\end{equation}

The function $\Gamma$ is chosen so that it coincides with $\ln n$ for large values of $n$,
while remaining bounded from below when $n$ is small.
More precisely, for $n<\eta$ we replace $\ln(\eta+(n-\eta))$ by its second-order Taylor expansion
around $\eta$, and for $n\ge \eta$, we retain the original logarithmic function.  We obtain the following estimates for $\mathcal E_{\Gamma}$ and the positive part of entropy.
\begin{lemma}\label{lem:modified_energy_growth}
Assume that $(n,{\bf u},\tau)$ is the maximal local mild solution of (\ref{PKSNS-time-dependent}) as obtained above,  where initial data $(n_0,{\bf u}_0)$ satisfies (\ref{regularic}).  If the initial cell mass satisfies the subcritical mass condition:
\begin{align}\label{subcriticalmass_condition}
\int_{\mathbb R^2} n_0 dx<8\pi. 
\end{align}
Then, the  modified free energy $\mathcal E_{\Gamma}[n,\bf {u}]$ defined in  (\ref{modified}) satisfies the following estimate for any $k\geq 1$:
\begin{align}\label{estimateemodifytu}
 &\mathbb E\Big[\sup_{t\leq T{  \wedge \tau}}|\mathcal E_{\Gamma}(t)|^k\Big]+ \mathbb E\bigg(\int_0^{T{  \wedge \tau}} \bar{\mathcal G}(r) \, dr\bigg)^k \nonumber\\
 \lesssim &e^{C_kT^k} 
\left(
|\mathcal E_{\Gamma}(0)|^k+\bigg[(\frac{3}{2}-\ln\eta)^kM^k+\frac{C^p(M)M^k}{(8\pi)^k}+C_k+\delta^k\bigg](T^k+T)
\right), \quad \forall\, 0<t<T,
\end{align}
where $C_k>0$ is a constant depending on $k$, $\mathcal E_{\Gamma}(0)=\mathcal E_{\Gamma}(n_0,{\bf u}_0)$ and
$$\bar{\mathcal G}_{\Gamma}:=\int_{\{n\ge\eta\}} n|\nabla\ln n-\nabla c|^{2}\,dxdt+\Vert \nabla {\bf u}\Vert_{L^2}^2dt
     +\frac{2}{3}\int_{\{n<\eta\}} (4\eta^{-1}-3\eta^{-2}n)|\nabla n|^{2}\,dxdt.$$
Furthermore,  for any $\delta>0$ there exists a constant
$C(\mathcal{E}_\Gamma[n_0,{\bf u}_0],M,\delta)$ such that
\begin{align}\label{eq218jan}
\mathbb E\Big[\sup_{t\leq T{  \wedge \tau}}\big( &\Vert n(t) \ln^+ n(t)\Vert^k_{L^1}+\|{\bf u}(t)\|_{L^{2}}^{2k}\big)\Big]\nonumber\\
&\lesssim  e^{C_kT^k}
\left(
|\mathcal E_{\Gamma}(0)|^k+\bigg[(\frac{3}{2}-\ln\eta)^kM^k+\frac{C^k(M)M^k}{(8\pi)^k}+C_k+\delta^k\bigg](T^k+T)
\right)\nonumber\\
&+C_k\bigg[M^k\ln\eta(\delta,M)^{-k}
     +\bigg(\frac{3}{2}M\bigg)^k
     +C^k(M)\bigg(\frac{M}{8\pi}\bigg)^k\bigg],
\end{align}
where $C_k>0$ is a constant and constant $C(M)>0$ depends on $M.$

\end{lemma}

\begin{proof}

{First recall that $\Delta c=-n$. We also recall Theorem \ref{thm:regularity} in which we proved that $n\in C^1(K_\tau;L^{2}(\mathbb R^2))$ $\mathbb P$-a.s. for any compact set $K_\tau$ in $(0,\tau)$. Now, we use the even property of $-(\Delta)^{-1}$ in \eqref{csolution_convolution}   to obtain, }
\begin{align}\label{cibp_estimate_1}
\int_{\mathbb R^2} ndc\,dx=-\int_{\mathbb R^2} n d(\Delta^{-1}n)\,dx=-\int_{\mathbb R^2}(\Delta^{-1}n)dn\,dx=\int_{\mathbb R^2}dnc\,dx.
\end{align}
To show \eqref{estimateemodifytu}, using  the fact \eqref{cibp_estimate_1}, we apply the divergence-free
condition on the vector field ${\bf u}$ together with integration by parts in \eqref{PKSNS-time-dependent} to obtain
\begin{align}\label{combine1jan1}
{d}\bigg[\int_{\mathbb R^2} \left( n\Gamma(n)-\frac{nc}{2} \right)dx\bigg]
&=\int_{\mathbb R^2}dn(\Gamma(n))+nd(\Gamma(n))dxdt-\int_{\mathbb R^2}dncdxdt\nonumber\\
&-\frac{1}{2}\int_{\mathbb R^2}ndcdxdt+\frac{1}{2}\int_{\mathbb R^2} dncdxdt\nonumber\\
&= \int_{\mathbb R^2} dn(\Gamma(n)-c)\,dxdt
  +\int_{\mathbb R^2} nd(\Gamma(n))\,dxdt\nonumber\\
&= -\int_{\mathbb R^2} (n\nabla\ln n-\nabla c\, n)\cdot(\Gamma'(n)\nabla n-\nabla c)\,dxdt
    -\int_{\mathbb R^2} {\bf u}\cdot\nabla n\,\Gamma(n)\,dxdt \nonumber\\
&\quad +\int_{\mathbb R^2} \nabla\cdot({\bf u}n)\,c\,dxdt
    -\int_{\mathbb R^2} \nabla(n\Gamma'(n))\cdot(n\nabla\ln n-\nabla c\,n)\,dxdt \nonumber\\
&\quad -\int_{\mathbb R^2} {\bf u}\cdot\nabla n\,n\Gamma'(n)\,dxdt.
\end{align}
Then by It\^{o}'s formula, we find
\begin{align}\label{combinejan2}
d\big(\frac{1}{2}\|{\bf u}\|_{L^2}^2\big)
= \Big( - \|\nabla {\bf u}\|_{L^2}^2 
    + \langle {\bf u}, n\nabla c \rangle 
    + \frac{1}{2}\|{\bf f}({\bf u})\|_{L_2(U;L_{\sigma}^2)}^2 \Big)\,dt
    + \langle {\bf u}, {\bf f}({\bf u})\,dW_t \rangle.
\end{align}
Combining (\ref{combine1jan1}) with (\ref{combinejan2}), one has
\begin{align}\label{modified_energy_E_evolution}
&{d}\bigg(\int_{\mathbb R^2
} \left( n\Gamma(n)-\frac{nc}{2} \right)dx\bigg)+d(\frac{1}{2}\Vert {\bf u}\Vert_{L^2}^2)\nonumber\\
=&-\int_{\mathbb R^2
}  (n\nabla\ln n-\nabla c\, n)\cdot(\Gamma'(n)\nabla n-\nabla c)\,dxdt
    -\int_{\mathbb R^2
}  {\bf u}\cdot\nabla n\,\Gamma(n)\,dxdt \nonumber\\
&\quad
    -\int_{\mathbb R^2
}  \nabla(n\Gamma'(n))\cdot(n\nabla\ln n-\nabla c\,n)\,dxdt \nonumber\\
&\quad -\int_{\mathbb R^2
}  {\bf u}\cdot\nabla n\,n\Gamma'(n)\,dxdt - \|\nabla {\bf u}\|_{L^2}^2 dt
    + \frac{1}{2}\|{\bf f}({\bf u})\|_{L_2(U;L_{\sigma}^2)}^2 \,dt
    + \langle {\bf u}, {\bf f}({\bf u})\,dW_t \rangle\nonumber\\
    &:=\sum_{i=1}^7T_i.
\end{align}

Next, we consider the second term $T_2$ and the fourth term $T_4$.
Since the function $\Gamma$ is finite near the origin, we define
\[
E(r):=\int_0^{r}\Gamma(s)\,ds,
\qquad
G(r):=\int_0^{r}s\Gamma'(s)\,ds .
\]
Using the divergence-free condition $\nabla\cdot {\bf u}=0$ and integration by parts, we compute
\begin{align}\label{t2mar9mon1}
T_2=-\int_{\mathbb R^2} {\bf u}\cdot\nabla(E(n))\,dxdt
   =\int_{\mathbb R^2} (\nabla\cdot {\bf u})\,E(n)\,dxdt
   =0,
\end{align}
and similarly,
\begin{align}\label{t4mar9mon2}
T_4=-\int_{\mathbb R^2} {\bf u}\cdot\nabla(G(n))\,dxdt
   =\int_{\mathbb R^2} (\nabla\cdot {\bf u})\,G(n)\,dxdt
   =0.
\end{align}

Next, we estimate the terms $T_1+T_3$.
Recalling the definition of $\Gamma$ in \eqref{gamman}, the cut-off threshold
$\eta=\min\{\delta/M,1\}$, and the fact that
\[
\Gamma'(n)=2\eta^{-1}-\eta^{-2}n \qquad \text{for } n\le\eta,
\]
a direct calculation yields
\begin{align}\label{t1t3mar9mon}
T_1+T_3
&= -\int_{\{n\ge\eta\}} (n\nabla\ln n-\nabla c\,n)\cdot\left(\frac{1}{n}\nabla n-\nabla c\right)\,dxdt \nonumber\\
&\quad -\int_{\{n<\eta\}} (n\nabla\ln n-\nabla c\,n)\cdot\left((2\eta^{-1}-\eta^{-2}n)\nabla n-\nabla c\right)\,dxdt\nonumber \\
&\quad -\int_{\{n<\eta\}} (2\eta^{-1}-\eta^{-2}n)\nabla n\cdot(n\nabla\ln n-\nabla c\,n)\,dxdt\nonumber \\
&\quad +\int_{\{n<\eta\}} n\eta^{-2}\nabla n\cdot(n\nabla\ln n-\nabla c\,n)\,dxdt .
\end{align}
Noting the inequality
\[
\sup_{n<\eta}\sqrt{\left(-3\eta^{-2}n^{2}+4\eta^{-1}n\right)}
\leq \frac{2}{\sqrt{3}} < 2 ,\]
one has from (\ref{t1t3mar9mon}) that 
\begin{align*}
T_1+T_3
&= -\int_{\{n\ge\eta\}} n|\nabla\ln n-\nabla c|^{2}\,dxdt
    -\int_{\{n<\eta\}} (4\eta^{-1}-3\eta^{-2}n)|\nabla n|^{2}\,dxdt \\
&\quad +\int_{\{n<\eta\}}(-3\eta^{-2}n+4\eta^{-1})n\,\nabla c\cdot\nabla n\,dxdt
    -\int_{\{n<\eta\}} n|\nabla c|^{2}\,dxdt \\
&\le -\int_{\{n\ge\eta\}} n|\nabla\ln n-\nabla c|^{2}\,dxdt
     -\int_{\{n<\eta\}} (4\eta^{-1}-3\eta^{-2}n)|\nabla n|^{2}\,dxdt \\
&\quad +\frac{2}{\sqrt{3}}\int_{\{n<\eta\}}  \sqrt{\left(-3\eta^{-2} n + 4\eta^{-1}\right)n}
|\nabla c||\nabla n|\,dxdt
     -\int_{\{n<\eta\}} n|\nabla c|^{2}\,dxdt
     +\int_{\{n<\eta\}} \nabla n\cdot\nabla c\,dxdt .
\end{align*}
With the aid of Young's inequality, we obtain 
\begin{align}\label{t1t3uppermar9mon}
T_1+T_3
&\le -\int_{\{n\ge\eta\}} n|\nabla\ln n-\nabla c|^{2}\,dxdt
      -\int_{\{n<\eta\}} \frac{2(4\eta^{-1}-3\eta^{-2}n)}{3}\,|\nabla n|^{2}\,dxdt \nonumber\\
&\quad +\int_{\{n<\eta\}} \nabla n\cdot\nabla c\,dxdt .
\end{align}
Now, we show that 
\begin{equation}
\int_{\{n<\eta\}} \nabla n\cdot\nabla c\,dx \le \delta .
\end{equation}
Indeed,  we use the choice of $\eta$  and integration by parts to  get
\begin{align}\label{crosstermncmar9mon2}
\int_{\{n<\eta\}} \nabla n\cdot\nabla c\,dx
&= \int_{\mathbb R^2} \nabla(\min\{n,\eta\})\cdot\nabla c\,dx \nonumber\\
&= -\int_{\mathbb R^2}  \min\{n,\eta\}\,\Delta c\,dx
 \le \eta\int_{\mathbb R^2}  n\,dx
 \le \eta M
 \le \delta ,
\end{align}
where we used the identity
 $\nabla n\,\mathbf{1}_{\{n<\eta\}}=\nabla(\min\{n,\eta\})$ \text{a.e.}  for $n\in W^{1,p}(\mathbb{R}^{2})$, $1<p<\infty$.
 Collecting (\ref{t1t3uppermar9mon}) and (\ref{crosstermncmar9mon2}), we conclude that
\begin{equation}\label{428smalldeltamar9}
T_1+T_3
\le -\int_{\{n\ge\eta\}} n|\nabla\ln n-\nabla c|^{2}\,dxdt
     -\frac{2}{3}\int_{\{n<\eta\}} (4\eta^{-1}-3\eta^{-2}n)|\nabla n|^{2}\,dxdt
     +\delta dt
\le \delta dt.
\end{equation}
Thus, using  (\ref{t2mar9mon1}), (\ref{t4mar9mon2}) and (\ref{428smalldeltamar9}) in (\ref{modified_energy_E_evolution}) yields
\begin{align}\label{thusjan}
{d}\bigg(\int_{\mathbb R^2} \left( n\Gamma(n)-\frac{nc}{2} \right)dx\bigg)+d(\frac{1}{2}\Vert {\bf u}\Vert_{L^2}^2)+\bar{\mathcal G}_{\Gamma} \leq \delta dt+\frac{1}{2}\Vert {{\bf f}(\bf u})\Vert_{L^2{(U;L_{\sigma}^2})}^2dt+\sum_{j\geq 1} ({\bf u},{\bf f}_j({\bf u}))_{L^2}dW^j, 
\end{align}
where $W^j = \left<W,e_j\right>$ and 
\begin{align}\label{barGgamma}
\bar{\mathcal G}_{\Gamma}:=\int_{\{n\ge\eta\}} n|\nabla\ln n-\nabla c|^{2}\,dxdt+\Vert \nabla {\bf u}\Vert_{L^2}^2dt
     +\frac{2}{3}\int_{\{n<\eta\}} (4\eta^{-1}-3\eta^{-2}n)|\nabla n|^{2}\,dxdt.
\end{align}
Now, observe that since $\Gamma(n)$, given in (\ref{gamman}), is bounded from below by the finite value
$\ln \eta(\delta,M)-3/2$, we have
that 
\begin{equation}\label{eq231jan}
\bigg(-\int_{\{n<1\}} n\,\Gamma(n)\,dx\bigg)
\le \left(-\ln \eta(\delta,M) + \frac{3}{2}\right)M,
 \end{equation}
 holds $\mathbb{P}$-almost surely.
We apply Lemma \ref{logHLSlemma} and \eqref{eq231jan},
to obtain, for some $\delta_2>0$,
\begin{align}
\int_{\mathbb R^2} n\Gamma(n)\,dx-\frac12\int_{\mathbb R^2} nc\,dx&\geq\left(1-\frac{M}{8\pi}\right)\int_{\mathbb R^2} n\ln^{+}n\,dx
   +\int_{\{n<1\}} n\Gamma(n)\,dx -C(M)\frac{M}{8\pi}\nonumber\\
   &\geq \delta_2\int_{\mathbb R^2} n\ln^+n\, dx+(\ln\eta-\frac{3}{2})M-\frac{C(M)M}{8\pi}\nonumber\\
   &\geq (\ln\eta-\frac{3}{2})M-\frac{C(M)M}{8\pi}.\label{energylower}
\end{align}
In light of assumption (H1) and (\ref{energylower}), we obtain that for $k\geq 1,$ 
\begin{align}\label{itoformula_quadratic_modified_energy}
& \mathbb E\Big[\frac{1}{2}\int_0^{T\wedge\tau}\|{\bf f}({\bf u})\|_{L_2(U;L^2)}^{2k}(s)\,ds\Big]\lesssim (1+\mathbb E\int_0^{T\wedge\tau}\Vert {\bf u}\Vert^{2k}_{L^2}\,ds )\nonumber\\
 \lesssim&
 \mathbb E\big[\int_0^{T\wedge\tau}|\mathcal J(t)|^k\,ds\big] +(\frac{3}{2}-\ln\eta)^kM^kT+\frac{C^k(M)M^k}{(8\pi)^k}T.
\end{align}
Moreover, we have for $k\geq 1,$ 
by using the BDG inequality together with Young's inequality, that
\begin{align}\label{413nowmar9mon1}
\mathbb{E}\bigg[\sup_{s \le t} \int_0^{s\wedge  \tau}\bigg|\sum_{j\geq 1}\left( {\bf u}(r), {\bf f}_j({\bf u}(r)) \right)_{L^2} 
\, dW^j(r)\bigg|^k\bigg]
\le&
\varepsilon \, \mathbb{E}\sup_{s \le t\wedge\tau}\|{\bf u}(s)\|_{L^2}^{2k}
+
C_{\varepsilon,k}\,
\mathbb{E}\left(
\int_0^{t\wedge\tau} \big(\|{\bf u}(r)\|_{L^2}^2 + 1\big)\,dr
\right)^k\nonumber\\
\leq& \varepsilon \mathbb E[\sup_{t\leq T\wedge\tau}|\mathcal E_{\Gamma}(t)|^k]+C_{\varepsilon,k}{ T^{k-1}}\mathbb E\int_0^{T\wedge \tau}|\mathcal E_{\Gamma}|^k(r)dr\nonumber\\
&+[(\frac{3}{2}-\ln\eta)^kM^k+\frac{C^k(M)M^k}{(8\pi)^k}]T^k+C_{\varepsilon,k}T^k.
\end{align}

We next integrate (\ref{thusjan}) over $(0,T\wedge\tau)$ and raise both sides to a power $k\geq 1$, take $\sup_{t\leq \tau\wedge T}$ and then $\mathbb{E}$ on both sides, and finally apply (\ref{itoformula_quadratic_modified_energy}) and (\ref{413nowmar9mon1}) to obtain
\begin{align}\label{eqgronwallbeforemar10}
 &\mathbb E[\sup_{t\leq T\wedge \tau}|\mathcal E_{\Gamma}(t)|^k]+\mathbb E\big(\int_0^{T\wedge\tau}\bar{\mathcal G}(r)dr\big)^k\nonumber\\
 \lesssim& |\mathcal E_{\Gamma}(0)|^k+C_kT^{k-1}\mathbb E\int_0^{T\wedge\tau}|\mathcal E_{\Gamma}|^k(r)dr+\left[(\frac{3}{2}-\ln\eta)^kM^k+\frac{C^k(M)M^k}{(8\pi)^k}+C_k\right](T^k+T).
 \end{align}
By applying the Gr\"{o}nwall lemma to (\ref{eqgronwallbeforemar10}), we then arrive, for any $k\geq 1$, at
\begin{align}\label{lastinequalitymar8sunmar10}
 &\mathbb E[\sup_{t\leq T\wedge\tau}|\mathcal E_{\Gamma}(t)|^k]+\mathbb E\bigg(\int_0^{T\wedge\tau} \bar{\mathcal G}(r) \, dr\bigg)^k \nonumber\\
 \lesssim &e^{C_kT^k} 
\left(
|\mathcal E_{\Gamma}(0)|^k+\bigg[(\frac{3}{2}-\ln\eta)^kM^k+\frac{C^k(M)M^k}{(8\pi)^k}+C_k+\delta^k\bigg](T^k+T)
\right), \quad \forall\, 0<t<T,
\end{align}
where $C_k>0$ is a constant depending on $k$.
This completes the proof of \eqref{estimateemodifytu}.

Then, for any $k\geq 1,$ we have from (\ref{estimateemodifytu}) that the following estimate of
 the positive part of the entropy and the fluid energy holds:
\begin{align}\label{entropyestimate_lemm6}
\mathbb E\bigg[\left(1-\frac{M}{8\pi}\right)&\sup_{t\leq T\wedge\tau}\bigg(\int_{\mathbb R^2} n\ln^{+}n\,dx\bigg)^k
+\frac12\sup_{t\leq T\wedge \tau}(\|{\bf u}\|_{L^2}^2)^k\bigg]\nonumber\\
\lesssim& e^{C_kT^k} 
\left(
|\mathcal E_{\Gamma}(0)|^k+\bigg[(\frac{3}{2}-\ln\eta)^kM^k+\frac{C^k(M)M^k}{(8\pi)^k}+C_k+\delta^k\bigg](T^k+T)
\right)\nonumber\\
    & +\Big(M^k\ln\eta(\delta,M)^{-k}
     +\big(\frac{3}{2}M\big)^k
     +C^k(M)\frac{M^k}{(8\pi)^k}\Big)C_k,
\end{align}
for some constant $C_k>0.$  Moreover, in light of \eqref{subcriticalmass_condition}, one finds some $\hat\delta>0$ such that $1-\frac{M}{8\pi}>\hat\delta>0$, which completes  the proof of this lemma by using (\ref{entropyestimate_lemm6}).

\end{proof}

Next, we define the following $(\mathcal{F}_t)_{t\geq 0}-$stopping time.
For any \(R>0\), we let
\begin{align}\label{stopingtime1}
T_R &:= \inf \Bigg\{ t > 0 : 
 \sup_{s\in[0,t{ \wedge \tau}]} \left(\int_{\mathbb R^2} n(s)\ln^+ n(s) \, dx +
\| {\bf u}(s) \|^2_{L^2} \right) > R
\Bigg\}.
\end{align}
For this stopping time, we prove the following result.
\begin{proposition}\label{proposition1}
Assume that $(n,{\bf u},\tau)$ is the maximal local mild solution of (\ref{PKSNS-time-dependent}) with initial data $(n_0,{\bf u}_0)$ satisfying (\ref{regularic}) and the subcritical mass condition:
$$\int_{\mathbb R^2} n_0 dx<8\pi.$$  Then we have 
\begin{align}
\mathbb E\sup_{t\leq T{ \wedge\tau \wedge T_R
}}\left(\Vert\nabla{\bf u}(t)\Vert_{L^r}^{\frac{pq}{p-q}}+\Vert {\bf u}(t)\Vert_{L^{\frac{pq}{p-q}}}^{\frac{pq}{p-q}}+\Vert n(t)\Vert_{L^{p}}^{\frac{pq}{p-q}}\right)\leq C_T(R),
\end{align}
where $p$, $q$ and $r$ are given in (\ref{pqr}).
\end{proposition}
\begin{proof}
Recall the definition of the modified energy \eqref{modified}.
Applying It\^o's  product rule to \eqref{modified_energy_E_evolution}  yields 
\[
\begin{aligned}
d|\mathcal E_{\Gamma}|^k
={}& k |\mathcal E_{\Gamma}|^{k-2} \mathcal E_{\Gamma}\frac{1}{2}\Vert {\bf f}({\bf u})\Vert_{L_2}^2 \, dt+
k |\mathcal E_{\Gamma}|^{k-2} \mathcal E_{\Gamma}(T_1+T_3) \,+ k |\mathcal E_{\Gamma}|^{k-2} \mathcal E_{\Gamma} \sum_{j\geq 1}( {\bf u},{\bf f}_j({\bf u}))_{L^2} \, dW^j_t\\
&+ \frac{1}{2} k (k-1) |\mathcal E_{\Gamma}|^{k-2} \sum_{j\geq 1}( {\bf u},\mathcal P{\bf f}_j({\bf u}))^2_{L^2}\, dt,
\end{aligned}
\]
where $T_1$ and $T_3$ are defined in (\ref{t1t3mar9mon}), and $W^j_t = \left<W(t), e_j\right>$. Then, it follows that 
\begin{align}\label{afterint}
d|\mathcal E_{\Gamma}|^k&+k|\mathcal E_{\Gamma}|^{k-2} \mathcal E_{\Gamma}\bar{\mathcal G}_{\Gamma}
\leq{} k |\mathcal E_{\Gamma}|^{k-2} \mathcal E_{\Gamma}(\frac{1}{2}\Vert {\bf f}({\bf u})\Vert_{L_2}^2) \, dt
+ k |\mathcal E_{\Gamma}|^{k-2} \mathcal E_{\Gamma} \sum_{j\geq 1}( {\bf u},{\bf f}_j({\bf u}))_{L^2} \, dW^j_t\nonumber\\
&+ \frac{1}{2} k (k-1) |\mathcal E_{\Gamma}|^{k-2} \sum_{j\geq 1}( {\bf u},\mathcal P{\bf f}_j({\bf u}))^2_{L^2}\, dt+k |\mathcal E_{\Gamma}|^{k-2} \mathcal E_{\Gamma}\delta dt,
\end{align}
where $\bar {\mathcal G}$ is defined in (\ref{barGgamma}) and $\delta>0$ is given in \eqref{428smalldeltamar9}.
After integration on (\ref{afterint}) over {{$(0,t\wedge T_R\wedge\tau)$}}, we have
\begin{align}\label{combining1mar10}
|\mathcal E_{\Gamma}|^k({t\wedge T_R\wedge\tau})&+\int_0^{{t\wedge T_R\wedge\tau}}k |\mathcal E_{\Gamma}|^{k-2} \mathcal E_{\Gamma}\bar{\mathcal G}_{\Gamma}ds
\leq{} |\mathcal E^k_{\Gamma}|(0)+\int_0^{{t\wedge T_R\wedge \tau}}k|\mathcal E_{\Gamma}|^{k-2} \mathcal E_{\Gamma}(\frac{1}{2}\Vert {\bf f}({\bf u})\Vert_{L_2}^2) \, ds\nonumber\\
+& \int_0^{{t\wedge T_R\wedge\tau}}k |\mathcal E_{\Gamma}|^{k-2} \mathcal E_{\Gamma} \sum_{j\geq 1}( {\bf u},{\bf f}_j({\bf u}))_{L^2} \, dW^j_s\nonumber\\
&+\int_0^{{t\wedge T_R\wedge\tau}} \frac{1}{2} k (k-1) |\mathcal E_{\Gamma}|^{k-2} \sum_{j\geq 1}( u,\mathcal P{\bf f}_j(t,u))^2_{L^2}\, ds+\int_0^{{t\wedge T_R\wedge\tau}}k|\mathcal E_{\Gamma}|^{k-2} \mathcal E_{\Gamma}\delta ds .
\end{align}
In light of (\ref{stopingtime1}), we find on $(0, {T_R\wedge\tau} )$:
\begin{align}\label{combining1mar102}
\frac{1}{2}\|{\bf f}({\bf u})\|_{L_2(U;L_{\sigma}^2)}^2\lesssim (1+\Vert {\bf u}\Vert^2_{L^2} )\lesssim 1+R,
\end{align}
and thus $(0,T_R\wedge \tau)$,
\begin{align}\label{combining1mar103}
\sum_{j\geq 1}( {\bf u},\mathcal P {\bf f}_j({\bf u}))^2_{L^2}\leq \Vert {\bf u}\Vert_{L^2}\Vert {\bf f}({\bf u})\Vert_{L_2(U;L_{\sigma}^2)}\lesssim 1+R.
\end{align}
In addition, we utilize BDG inequality to get for constant $C>0$,
\begin{align}\label{441mar19tue}
k\mathbb E\big(\sup_{t\in(0,T)}\int_0^{t\wedge T_R\wedge\tau}& |\mathcal E_{\Gamma}|^{k-2} \mathcal E_{\Gamma} \sum_{j\geq 1}( {\bf u},{\bf f}_j({\bf u}))_{L^2} dW_s^j\big)
\leq C k\mathbb E\big(\int_0^{T\wedge T_R\wedge\tau}  |\mathcal E_{\Gamma}|^{2k-2}\big| \sum_{j\geq 1}({\bf  u},{\bf f}_j({\bf u}))_{L^2} \big|^2ds\big)^{\frac{1}{2}}\nonumber\\
\leq &Ck \mathbb E\big[\sup_{t\in(0,T\wedge T_R\wedge\tau)}(|\mathcal E_{\Gamma}|^{k-1}\Vert {\bf u}\Vert_{2}^2)\int_0^{T\wedge T_R\wedge\tau}|\mathcal E_{\Gamma}|^{k-1}\Vert {\bf f}_j({\bf u})\Vert_{2}^2ds\big]^{\frac{1}{2}}\nonumber\\
\leq &Ck \mathbb E\big[\sup_{ {t\in(0,T\wedge T_R
\wedge\tau)}}(|\mathcal E_{\Gamma}|^{k-1} )R^2\int_0^{ {T\wedge T_R\wedge\tau}}|\mathcal E_{\Gamma}|^{k-1}| ds\big]^{\frac{1}{2}}.
\end{align}
We collect (\ref{combining1mar10})--(\ref{441mar19tue}) to obtain 
\[
\begin{aligned}
 \mathbb E[\sup_{t\in(0,T\wedge T_R
 \wedge\tau)}|\mathcal E_{\Gamma}|^k(t)]
\leq{}& |\mathcal E^k_{\Gamma}|(0)+Ck\mathbb E\bigg(\int_0^{T\wedge T_R
\wedge\tau}|\mathcal E_{\Gamma}|^{k-2} \mathcal E_{\Gamma}(1+R) \, ds\bigg)\\
&+ \mathbb E\bigg(\sup_{t\in(0,T)}\int_0^{t\wedge T_R\wedge\tau}k |\mathcal E_{\Gamma}|^{k-2} \mathcal E_{\Gamma} \sum_{j\geq 1}( {\bf u},{\bf f}_j({\bf u}))_{L^2} \, dW^j_s\bigg)\\
&+\mathbb E\bigg(\int_0^{T\wedge T_R
\wedge\tau} \frac{1}{2} k (k-1) |\mathcal E_{\Gamma}|^{k-2} \sum_{j\geq 1}( {\bf u},\mathcal P{\bf f}_j({\bf u}))^2_{L^2}\, ds\bigg)+\mathbb E\bigg(\int_0^{T\wedge T_R\wedge\tau}k |\mathcal E_{\Gamma}|^{k-2} \mathcal E_{\Gamma}\delta ds\bigg)\\
\leq{}& |\mathcal E^k_{\Gamma}|(0)+Ck\mathbb E(\int_0^{ {T\wedge T_R\wedge\tau}}|\mathcal E_{\Gamma}|^{k-1}ds)(1+R)
 +Ck \mathbb E\big[\int_0^{ {T\wedge T_R\wedge\tau}}|\mathcal E_{\Gamma}|^{k-1} ds \big]R \\
&+Ck^2\mathbb E\bigg(\int_0^{T\wedge T_R\wedge\tau}|\mathcal E_{\Gamma}|^{k-2}ds\bigg)(1+R)+k\delta \mathbb E(\int_0^{T\wedge T_R
\wedge\tau}|\mathcal E_{\Gamma}|^{k-1}ds)\\
\leq & |\mathcal E^k_{\Gamma}|(0)+Ck (R^{k-1}+R^{k})T
 +Ck TR^{k} \\
&+Ck^2T(R^{k-2}+R^{k-1})+k\delta T R^{k-1},
\end{aligned}
\]
which implies for any deterministic constant $C_0>0$ that, 
$$\mathbb E[\sup_{t\in(0,T\wedge T_R\wedge\tau)}C_0^k|\mathcal E_{\Gamma}|^k](t)\leq C C_0^kk^2 T R^k.$$

Hence, we have for any $k\geq 1$ that, 
\begin{align}\label{nlognboundfeb26}
    \mathbb E[\sup_{{(0,T\wedge T_R\wedge\tau)}}e^{C_0|\mathcal E_{\Gamma}| }]\leq \mathbb E\left[\sup_{(0,T\wedge T_R
\wedge\tau)}\sum_{k=0}^\infty \frac{C_0^k|\mathcal E_{\Gamma}|^k}{k!}\right] \leq CT((CR)^2+CR)e^{C_0R},
\end{align}
Then by using (\ref{energylower}) and (\ref{nlognboundfeb26}), we obtain for any constant $C_0>0$ that,
\begin{align}\label{estimate_exp_nlogn}
\mathbb E\left(\sup_{t\leq T\wedge\tau{\wedge T_R}}e^{C_0\int n\ln^+ n dx}\right)\leq C({ R,T}).
\end{align}

We shall estimate density $n$ in $L^2(\mathbb R^2)$.  To begin with, we decompose it as follows for any random variable $K_R(T)$ (which will be chosen, depending on $R$ and $T$, carefully below). On the set $\{(\omega,t): t \leq T\wedge T_R\wedge \tau\}$, we write
\[
n = (n-K_R(T))_+ + \min\{n,K_R(T)\}, \quad K_R(T)>1.
\]
Moreover, on $\{(\omega,t): t \leq T\wedge T_R\wedge \tau\}$, we define the following modified mass:
\[
\eta_{K_R(T)} := \int_{\mathbb R^2} (n-K_R(T))_+ \, dx.
\]
In light of the definition of $\eta_{K_R}$, we obtain
\begin{equation}\label{log-estimate}
\eta_{K_R(T)}\le \int_{\mathbb R^2} \frac{\ln^+ n}{\ln K_R(T)} (n-K_R(T))_+ \, dx,
\end{equation}
which means on $\{(\omega,t): t \leq T\wedge T_R\wedge \tau\}$ that,
\begin{equation}\label{etaK-estimate}
\eta_{K_R(T)}{ (t)}  \le \frac{ \Vert n{ (t)} \ln^+n{ (t)} \Vert_{L^1}}{ \ln K_R(T)}.
\end{equation}
Note that if by choosing the vertical cut-off level $K_R(T)$ large enough almost surely, we can made $\eta_{K_R(T)}$ arbitrarily small. 
This requires, in our setting, $K_R(T)$ to be a random variable.

Next, we use the divergence-free condition of the fluid vector field ${\bf u}$, the Gagliardo-Nirenberg-Sobolev inequality and the Nash inequality given in Lemma \ref{Nashinequality} to estimate the time evolution of the $L^2$ norm of the truncated density $(n-K_R(T))_+$ as follows:
\begin{align}\label{obtainfrommar10tue1}
&\frac{1}{2}{d} \|(n-K_R(T))_+\|_{L^2}^2 
\le - \int_{\mathbb R^2} |\nabla (n-K_R(T))_+|^2 \, dxdt
+ \frac{1}{2} \int_{\mathbb R^2} (n-K_R(T))_+^3 \, dxdt\nonumber\\
&\qquad\qquad\qquad+ \frac{3}{2}K_R(T) \int_{\mathbb R^2} (n-K_R(T))_+^2 \, dxdt + K_R^2(T) Mdt \nonumber\\
&\le - \big(1 - \frac{C_{\text{GNS}} \eta_{K_R}}{2} \big) \|\nabla (n-K_R(T))_+\|_{L^2}^2dt 
+ \frac{3}{2} K_R(T) \|(n-K_R(T))_+\|_{L^2}^2dt + K_R^2(T) Mdt, 
\end{align}
where $C_{\text{GNS}}$ is the constant from the Gagliardo-Nirenberg-Sobolev inequality. 
This motivates the choice of the random variable 
$$K_R{(T)}:= e^{\sup\limits_{t\leq T{\wedge \tau\wedge T_R}}\Vert n(t)\ln^+n (t)\Vert_{L^1}C_{GNS}}$$ 
so that  we have from \eqref{etaK-estimate} that, 
$$\sup_{t\leq T\wedge \tau\wedge T_R}\eta_{  K_R(T)}(t)<\frac{1}{C_{\text{GNS}}}.$$
 Moreover, (\ref{estimate_exp_nlogn}) implies  for any $p \geq 1$ that
$$\mathbb E(K_R^{p}(T))\leq C(R,T),$$
for some positive constant $C(R,T)$ depending on $R$ and $T.$  Hence, we obtain from (\ref{obtainfrommar10tue1}) and Lemma \ref{Nashinequality} that  
\begin{align}\label{L2_estimate_n_KT}
\frac{1}{2}(\|(n-K_R(T))_+\|_{L^2}^2)(t\wedge\tau\wedge T_R)&\le  \frac{1}{2}(\|(n_0-K_R(T))_+\|_{L^2}^2)- \frac{1}{2} \int_0^{T\wedge\tau\wedge T_R}(\|\nabla (n-K_R(T))_+\|_{L^2}^2)ds \nonumber\\
&+ 2 \int_0^{T\wedge\tau\wedge T_R}[K_R(T)\|(n-K_R(T))_+\|_{L^2}^2]ds +  \int_0^{T\wedge\tau\wedge T_R} [K_R^2(T)] Mds \nonumber\\
&\le \frac{1}{2}(\|(n_0-K_R(T))_+\|_{L^2}^2)- \frac{1}{4 C_N M^2}\int_0^{T\wedge\tau\wedge T_R}( \|(n-K_R(T))_+\|_{L^2}^4)ds\nonumber \\
&+  2 \int_0^{T\wedge\tau\wedge T_R}[K_R(T)\|(n-K_R(T))_+\|_{L^2}^2]ds\nonumber\\
&+  \int_0^{T\wedge\tau\wedge T_R} [K_R^2(T)] Mds ,
\end{align}
where $C_N$ are the constants from Nash inequality in Lemma \ref{Nashinequality}.  { On the set $A:=\{(\omega,t):  \Vert (n-K_R(T))_+\Vert_{L^2}^4(t\wedge T_{R}\wedge \tau)\geq 8{M^2C_N K_R(T)} \Vert (n-K_R(T))_+\Vert^{2}_{L^2}(t\wedge T_{R}\wedge \tau) \}$,  we deduce from (\ref{L2_estimate_n_KT}) that $\Vert( n-K_R(T))_+\Vert^2_{L^2}({ t\wedge T_{R}\wedge \tau})\leq\|(n_0-K_R(T))_+\|_{L^2}^2+2 K_R^2(T)MT.$   Whereas, on the set $A^c$, we have $\Vert (n-K_R(T))_+\Vert_{L^2}^2(t\wedge T_{R}\wedge \tau)\leq 8{M^2C_N K_R(T)}.  $  In summary, we find }
\begin{align}\label{pointwise_KR_L2_n}
\sup_{t\leq T{ \wedge \tau\wedge T_R}}\|(n-K_R(T))_+\|_{L^2}^2)\leq \Vert n_0\Vert_{L^2}^2+2 K_R^2(T)MT+8M^2C_NK_R(T).
\end{align}
 By raising to a $\frac{k}{2}$-th power for any $k\geq 2$ and taking the expectation of (\ref{pointwise_KR_L2_n}),  
\begin{align*}
\mathbb E(\sup_{t\leq T{\wedge \tau}\wedge T_R}\|n(t)\|^k_{L^2}) &\le 2^{\frac{k}{2}}\mathbb E( \,\big\| \sup_{t\leq T{{\wedge \tau\wedge T_R}}}\min\{n(t), K_R(T)\} \big\|^k_{L^2}) + 2^{\frac{k}{2}} \mathbb E(\sup_{t\leq T{\wedge \tau\wedge T_R}}\,\big\| (n(t) - K_R(T))_+ \big\|^k_{L^2}) \\
&\le 2^{\frac{k}{2}} \mathbb E[K_R^{\frac{k}{2}}(T)] M^{\frac{k}{2}} + C_{\|n_0\|_2, M,k}T^{\frac{k}{2}} \mathbb E[K_R^{{k}}(T)+K_R^{\frac{k}{2}}(T)],
\end{align*}
which implies for any $k>1$ that,
\begin{align}\label{thankstomar10tue1}
\mathbb E\left(\sup_{t\leq T{\wedge \tau\wedge T_R}}\Vert n(t)\Vert_{L^2}^k\right)\leq C(k,R,T),
\end{align}
for some constant $C(k,R,T)>0$ depending on $k$, $R$ and $T.$

 An immediate consequence of the bounds above 
is obtained by applying the following interpolation inequality for any $2<k\leq 4$ and $l\geq k$,
\begin{align*}
\Vert n\Vert^q_{L^{\frac{k}2}}\leq \Vert n\Vert_{L^1}^{\frac{(4-k)l}{k}} \Vert n\Vert_{L^2}^{\frac{2l(k-2)}{k}}\leq M^{\frac{(4-k)l}{k}}\Vert n\Vert_{L^2}^{\frac{2l(k-2)}{k}}.
\end{align*}
Thanks to (\ref{thankstomar10tue1}), we thus obtain for $p,q$ defined in \eqref{pqr}, by setting $k=2p=4-2\epsilon$ 
in the inequality above that, we obtain for any $l\geq k$ that
\begin{align}\label{Lpestimate_1and2_n}
\mathbb E(\sup_{t\leq T\wedge\tau\wedge T_R}\Vert n\Vert_{L^{p}}^{l})\leq M^{\frac{(4-k)l}{k}}E\bigg(\sup_{t\leq\tau\wedge T\wedge T_R}\Vert n\Vert_{L^2}^{\frac{2l(k-2)}{k}}\bigg)\leq C(k,l,R,T).
\end{align}
{Observe that, for an appropriately small $\epsilon$ (say $\epsilon<\frac2{129}$), setting $l=\frac{pq}{p-q}=\frac4\epsilon-2$ gives us the desired estimates in Proposition \ref{proposition1}.}

Next, we will upgrade these bounds spatially in $L^4$ { which will lead us to the required $L^\infty$ estimates for $\nabla c$}.

We compute from \eqref{PKSNS-time-dependent} the
 time evolution of the $L^4$ norm of the cell density $n$ and obtain
\begin{align}\label{n4estimatemar10tue}
&\frac{1}{4}{d}\|n\|_{L^4}^4\le -\frac{3}{4}\|\nabla (n^2)\|_{L^2}^2dt
   + \frac{3}{4}\|n^2\|_{L^{5/2}}^{5/2}dt\nonumber\\
   \le &-\frac{3}{4}\|\nabla (n^2)\|_{L^2}^2dt
   + C_{\mathrm{GNS}}\|\nabla (n^2)\|_{L^2}^{1/2}\|n^2\|_{L^2}^2dt.
   \end{align}

By using Lemma \ref{Nashinequality}, we find {from (\ref{n4estimatemar10tue}) that }
   \begin{align*}
\frac{1}{4}{d}\|n\|_{L^4}^4&\leq - \frac{\| n^2 \|_{L^2}^4}{C_N\|n^2\|_{L^1}^2}dt + C_{GNS} \| n^2 \|_{L^2}^{8/3}dt\,
\end{align*}
where $C_N$ is the Nash inequality constant given in (\ref{nash_formula}).  Then we have that the following holds $\bP$-a.s.:
\begin{align}\label{nL4_inequality_global}
\frac{1}{4}\Vert n\Vert_{L^4}^4(t\wedge T_{R}\wedge \tau)\leq& \frac{1}{4}\Vert n\Vert_{L^4}^4(0)-\int_0^{t\wedge T_{R}\wedge \tau} \frac{\| n \|_{L^4}^8}{C_N\|n\|_{L^2}^4} ds+C_{GNS}\int_0^{t\wedge T_{R}\wedge \tau}\Vert n\Vert_{L^4}^{16/3}ds.
\end{align}
{ On the set $A:=\{(\omega,t):  \Vert n\Vert_{L^4}^4(t\wedge T_{R}\wedge \tau)\geq \sqrt{C_{GNS}C_N} \Vert n(t\wedge T_{R}\wedge \tau)\Vert_{L^2}^2\Vert n(t\wedge T_{R}\wedge \tau)\Vert^{8/3}_{L^4} \}$}, we deduce from (\ref{nL4_inequality_global}) that $\Vert n\Vert^4_{L^4}({ t\wedge T_{R}\wedge \tau})\leq \Vert n\Vert_{L^4}^4(0).$  Whereas, on the set $A^c$, we have $\Vert n\Vert_{L^4}^{\frac{4}{3}}({ t\wedge T_{R}\wedge \tau})\leq  \sqrt{C_{GNS}C_N} \Vert n(t\wedge T_R\wedge \tau)\Vert_{L^2}^2.$ 
In summary, we have
\begin{align}\label{4_power_inequality_nL4_global}
\sup_{t\in(0,\tau\wedge T\wedge T_R)}\Vert n\Vert_{L^4}^4\leq \Vert n\Vert^4_{L^4}(0)+(C_{GNS}C_N)^{\frac{3}{2}}\sup_{t\in(0,\tau\wedge T\wedge T_R)} \Vert n\Vert_{L^2}^{6}.
\end{align}
Then, taking the $\frac{k}{4}$-th power on both sides of \eqref{4_power_inequality_nL4_global} for any $k>1$, we obtain
\begin{align*}
\sup_{t\in(0,\tau\wedge T\wedge T_R)}\Vert n\Vert_{L^4}^k\leq 2\Vert n\Vert^k_{L^4}(0)+2(C_{GNS}C_N)^{\frac{3k}{8}}\sup_{t\in(0,\tau\wedge T\wedge T_R)} \Vert n\Vert_{L^2}^{\frac{3k}{2}}
\end{align*}
By applying (\ref{thankstomar10tue1}),  we find for any $k>1,$ 
\begin{align}\label{suptCt4mar10tuenew}
\mathbb E\left(\sup_{t\leq T\wedge\tau \wedge T_R}\|n(t)\|^k_{L^4}\right)
\le
C(k,R,T).
\end{align}
Thanks to Lemma \ref{lemma28} and (\ref{suptCt4mar10tuenew}) {and the fact that $\|n\|_{L^1}=M$ at any time almost surely}, we obtain
\begin{align}\label{eq:2.4new}
\mathbb E(\sup_{t\leq T\wedge \tau \wedge T_R}\|\nabla c(t)\|^{k}_{L^\infty})\le &\mathbb E(\sup_{t\leq T\wedge \tau\wedge T_R}\Vert n\Vert_{L^1}^{\frac{k}{3}}\Vert n\Vert_{L^4}^{\frac{2k}{3}})\nonumber\\
=& M^{\frac{k}{3}} \mathbb E(\sup_{t\leq T\wedge \tau\wedge T_R}\|n(t)\|^{{\frac{2k}{3}}}_{L^4})
\le C(k,R,T,M).
\end{align}

We next estimate the fluid velocity ${\bf u}$. 
For that purpose, we introduce the notation $\nabla \times {\bf u}$ for the vorticity of a two-dimensional vector field ${\bf u} = ({ u}_1, u_2)$, i.e.,
\[
\nabla \times {\bf u} := \partial_{x_1} u_2 - \partial_{x_2} u_1.
\]
Now, consider the following equation for the vorticity $v$ of the fluid velocity ${\bf u}$: 
\begin{equation}
d v  + (\mathbf{u} \cdot \nabla) v dt = \Delta v dt+ \nabla \times ({n}  \nabla {c})dt+\nabla \times {\bf f}(\mathbf{u})dW_t.
\label{eq:vorticity}
\end{equation}
By using $v$-equation, we have 
\begin{align}\label{omegatwedgeRdmar10}
k\Vert v\Vert^{2(k-1)}_{L^2}{d}\int_{\mathbb R^2}& |v  |^2dx
+
2k\Vert v\Vert^{2(k-1)}_{L^2}\int_{\mathbb R^2}  |\nabla v|^2dt
\nonumber\\
\le&
-2k\Vert v\Vert^{2(k-1)}_{L^2}\int_{\mathbb R^2}  \mathbb (n\nabla c)\cdot \nabla^{\perp}v dt\nonumber\\
&+ 2k\Vert v\Vert^{2(k-1)}_{L^2} \langle  v(t), \nabla \times {\bf f}({\bf u}(t)) dW_t \rangle+k\Vert v\Vert^{2(k-1)}_{L^2}\| {\bf f}({\bf u}(t))\|_{L_2(U;D(A^{1/2}))}^2dt \nonumber\\
\le&
k\Vert v\Vert^{2(k-1)}_{L^2} \int_{\mathbb R^2}  |\nabla v |^2dx
+
C_k\Vert v\Vert^{2(k-1)}_{L^2}\|\nabla c\|_{L^\infty}^2
\int_{\mathbb R^2}  n^2dx\nonumber\\
&+ 2k\Vert v\Vert^{2(k-1)}_{L^2}\langle v(t), \nabla \times {\bf f}({\bf u}(t)) dW_t \rangle+ k\Vert v\Vert^{2(k-1)}_{L^2}\| {\bf f}({\bf u}(t))\|_{L_2(U;D(A^{1/2}))}^2dt,
\end{align}
where $C>0$ is a constant depending on $k$.

We integrate, use Young's inequality and \eqref{eq:2.4new} to obtain
\begin{align}\label{omegatwedgeRdmar101}
&\|v (t\wedge \tau
\wedge T_R)\|_{L^2}^{2k} +  k\int_0^{t \wedge T_R\wedge\tau} \Vert v(s)\Vert_{L^2}^{2(k-1)}\|\nabla v (s)\|_{L^2}^2 \, ds\nonumber\\
\le&\Vert v(0)\Vert_{L^2}^{2k}+C_k\int_0^{t\wedge T_R\wedge\tau}\Vert v(s)\Vert_{L^2}^{2(k-1)}( \Vert \nabla c\Vert_{L^\infty}^4+\Vert n\Vert_{L^2}^4)ds
+2k \int_0^{t \wedge T_R\wedge\tau}\Vert v(s)\Vert_{L^2}^{2(k-1)} \langle v(s), \nabla\times  {\bf f}({\bf u}(s)) dW_s \rangle\nonumber\\
&+ k\int_0^{t \wedge T_R\wedge\tau}\Vert v(s)\Vert_{L^2}^{2(k-1)} \| {\bf f}({\bf u}(s))\|_{L_2(U;D(A^{1/2}))}^2 \, ds.
\end{align}
 By BDG's inequality, Young's inequality and assumption (H1), one has
\begin{align}\label{omegatwedgeRdmar102}
\mathbb{E}\sup_{0 \le t \le T}
\left|
\int_0^{t \wedge T_R\wedge\tau}
\Vert v(s)\Vert^{2(k-1)}_{L^2}\left\langle
v(s),
\nabla \times {\bf f}({\bf u}(s)) \, dW_s
\right\rangle
\right|
\le&
C \,
\mathbb{E}
\left(
\int_0^{T \wedge T_R\wedge\tau}
\|v(s)\|^{4(k-1)}_{L^{2}}|\langle v(s) ,{\bf f}({\bf u}(s))\rangle|^2
\, ds
\right)^{1/2}\nonumber\\
\le C \,
\mathbb{E}\bigg[
\sup_{s\in(0,T\wedge\tau\wedge T_R)}\|v(s)\|^{2(k-1)}_{L^{2}}&
\left(\int_0^{T \wedge T_R\wedge\tau}
|\langle v(s) ,{\bf f}({\bf u}(s))\rangle|^2
\, ds
\right)^{1/2}\bigg]\nonumber\\
\le C \,
\mathbb{E}\bigg[
\sup_{s\in(0,T\wedge\tau\wedge T_R)}\|v(s)\|^{2k-1}_{L^{2}}&
\left(\int_0^{T \wedge T_R\wedge\tau}
\|{\bf f}({\bf u})(s)\|_{L_2(U;D(A^{1/2}))}^2
\, ds
\right)^{1/2}\bigg]\nonumber\\
\le \varepsilon{\mathbb{E}}( \,\sup_{s\in(0,T\wedge\tau\wedge T_R)}\Vert v(s)\Vert^{2k}_{L^2})
&+C_{\varepsilon}
\mathbb{E}
\left(
\int_0^{T \wedge T_R\wedge\tau}(1+
\| {\bf u}(s))\|_{W{^{1,2}}}^2)
\, ds
\right)^k,
\end{align}
for some positive constants $C$ and $C_{\varepsilon}$ depending on $\varepsilon>0$ small enough.  In addition, we use assumption (H1) to get
\begin{align}\label{omegatwedgeRdmar103}
 \int_0^{t \wedge T_R\wedge\tau}\Vert v(s)\Vert_{L^2}^{2(k-1)}\Vert   {\bf f}({\bf u}(s))\|_{L_2(U;D(A^{1/2}))}^2 \, ds\lesssim  &\int_0^{t \wedge T_R
\wedge\tau} \Vert v(s)\Vert_{L^2}^{2(k-1)}(1+\| {\bf u}(s)\|_{W^{1,2}}^2)ds\nonumber\\
\lesssim  \sup_{s\in(0,T\wedge\tau\wedge T_R)} \Vert v(s)\Vert_{L^2}^{2k}&+\bigg(\int_0^{T\wedge\tau\wedge T_R}(1+\| {\bf u}(s)\|_{W^{1,2}}^2)ds\bigg)^{ {k}}.
 \end{align}
By using H\"{o}lder' inequality, it follows from (\ref{omegatwedgeRdmar101}), (\ref{omegatwedgeRdmar102}) and (\ref{omegatwedgeRdmar103}) that
\begin{align*}
&\mathbb E(\sup_{0 \le t \le T\wedge\tau\wedge T_R}
\|v (t)\|_{L^2}^{2k})
+
\mathbb{E}\bigg(\int_0^{T\wedge T_R\wedge\tau}\Vert v(s)\Vert_{L^2}^{2(k-1)}
\|\nabla v (s)\|_{L^2}^{2k} \, ds\bigg)\nonumber\\
\le&\Vert v(0)\Vert_{L^{2}}^{2k}+C\mathbb E\big(\sup_{t\in(0,\tau\wedge T\wedge T_R)}\Vert \nabla c\Vert_{L^\infty}^8 \big)T+C\mathbb E\big( \sup_{t\in(0,\tau\wedge T\wedge T_R)}\Vert n\Vert_{L^2}^{2k+2} \big)T+C\mathbb E\big( \sup_{t\in(0,\tau\wedge T\wedge T_R)}\Vert n\Vert_{L^2}^{4k-1} \big)T\nonumber\\
&+C\bigg(T+RT+\mathbb E\big(\int_0^{T\wedge\tau\wedge T_R    }\Vert v\Vert_{L^2}^{2k}(s)\,ds\big)\bigg)^k\nonumber \\
\le&\Vert v(0)\Vert_{L^2}^2+C(k,R,T)+CT^k+CR^kT^k+CT^{k-1}\mathbb E\bigg(\int_0^{T\wedge\tau\wedge T_R}\|v\|_{L^2}^{2k}(s)ds\bigg).
\end{align*}
Then, by Gr\"{o}nwall inequality, one has
\begin{align}
 \mathbb{E}\sup_{0 \le s \le T\wedge\tau\wedge T_R}
\|v (s)\|_{L^2}^{2k}\leq C(k,R,T),
\end{align}
and thus, for any $k\geq 1$ we have
\begin{align}\label{nablauR_2p_2_estimate}
 \mathbb{E}\sup_{0 \le s \le T\wedge\tau\wedge T_R}
\|\nabla {\bf u}(s)\|_{L^2}^{k}\leq C(k,R,T).
\end{align}

Next, we show that for $p,q,r$ defined in \eqref{pqr}, we have
\begin{equation}
\mathbb{E}\Big[ \sup_{0 \le t \le T\wedge\tau\wedge T_R} \|\nabla {\bf u}(t, \cdot)\|_{L^r}^{\frac{pq}{p-q}} \Big] \leq C(R,T).
\end{equation}
Taking $L^r$ norm of $\nabla \bu$ in \eqref{3p2mar} we obtain,
\begin{align}\label{labelphi2}
\|\nabla {\bf u}\|_{L^{r}}
&\lesssim \Vert e^{-tA}\nabla{\bf u}_0\Vert_{L^{r}}+ \int_0^t \| e^{-(t-s)A} \mathcal{P} \nabla   ( n(s)  \nabla c(s) + {\bf u}(s)\cdot \nabla {\bf u}_R(s) ) \|_{L^r} \, ds\nonumber\\
&\qquad+\Vert\int_0^t e^{-(t-s)A}\mathcal  P\nabla {\bf f}({\bf u}_R(s))\,dW_s\Vert_{L^r} \nonumber \\
&\lesssim \Vert \nabla {\bf u}_0\Vert_{L^r}+ \int_0^t (t-s)^{-\frac{1}{2
}+\frac{1}{r}-\frac{1}{q}} \big( \| n (s) \|_{L^p} \| \nabla c(s) \|_{L^\frac{pq}{p-q}})ds\nonumber\\ 
&\qquad+\int_0^t(t-s)^{-\frac{1}{2
}+\frac{1}{r}-\frac{pq+2p-2q}{2pq}} \| \nabla {\bf u}(s) \|_{L^2} \| {\bf u}(s) \|_{L^\frac{pq}{p-q}} \big) \, ds\nonumber\\
&\qquad+\Vert\int_0^t e^{-(t-s)A}\mathcal   P\nabla {\bf f}({\bf u}_R(s))\,dW_s\Vert_{L^r}\nonumber\\
&\lesssim  \Vert\nabla {\bf u}_0\Vert_{L^r}+ \int_0^t (t-s)^{-\frac{1}{2}+\frac{1}{r}-\frac{1}{q}} \big( \| n(s) \|^2_{L^p})ds +\Vert\int_0^t e^{-(t-s)A}\mathcal  P\nabla{\bf f}({\bf u}_R(s))\,dW_s\Vert_{L^r}\nonumber\\
&\qquad+\int_0^t (t-s)^{-\frac{1}{2}+\frac{1}{r}-\frac{pq+2p-2q}{2pq}}  \| {\bf u}_R(s)\|_{L^{\frac{pq}{p-q}}} \| \nabla {\bf u}_R(s) \|_{L^2} \, ds, 
\end{align}
where $p,$ $q$ and $r$ are given in (\ref{pqr}).
Taking the supremum over $t\in(0,T\wedge\tau\wedge T_R)$ and then the expectation, we obtain 
\begin{align}
\sup_{t\in(0,\tau\wedge T\wedge T_R)}\Vert \nabla  {\bf u}\Vert_{L^r}&\lesssim  \Vert \nabla {\bf u}_0\Vert_{L^r}+ \big(\int_0^{T\wedge\tau\wedge T_R}(t-s)^{-\frac{1}{2}+\frac{1}{r}-\frac{1}{q}}  \, ds\nonumber \\
&+\int_0^{T\wedge\tau\wedge T_R} (t-s)^{-\frac{1}{2}+\frac{1}{r}-\frac{pq+2p-2q}{2pq}}  \, ds\big)\,\sup_{t\in(0,\tau\wedge T\wedge
 T_R)}  (\Vert n \Vert_{L^p}^2+\Vert {\bf u}\Vert_{L^{\frac{pq}{p-q}}}^2+\Vert \nabla {\bf u} \Vert^2_{L^2})\nonumber\\
&+\sup_{t\in(0,\tau\wedge T\wedge T_R)}\Vert\int_0^t e^{-(t-s)A}\mathcal  P\nabla {\bf f}({\bf u}_R(s))\,dW_s\Vert_{L^r},\label{nablau}
\end{align}
where we recall that $r=2+\epsilon$ with $\epsilon>0$ small enough.

We will next bound each term on the right hand side. By Gagliardo-Nirenberg-Sobolev inequality, we have for $p$ and $q$ given in (\ref{pqr}), that
\[
\|{\bf u}\|_{L^{\frac{pq}{p-q}}}
\le
C
\|\nabla {\bf u}\|_{L^{2}}^{\theta}
\|{\bf u} \|_{L^2}^{1-\theta},\qquad 0<\theta=
1-\frac{2}{ q}+\frac{2}{p}<1,
\]
which implies due to \eqref{nablauR_2p_2_estimate} that
\begin{align}\label{uR_Lp_global_estimate}
\mathbb E\bigg(\sup_{t\in(0,T\wedge\tau\wedge T_R)}\Vert {\bf u}\Vert^{\frac{2pq}{p-q}}_{L^\frac{pq}{p-q}}\bigg) \lesssim \mathbb  E\bigg(\sup_{t\in(0,T\wedge\tau\wedge
 T_R)}\Vert \nabla {\bf u}\Vert_{L^2}^{\frac{2pq}{p-q}}\bigg)R^{\frac{2(1-\theta)pq}{p-q}}\leq C(p,q,R,T).
\end{align}
For the final term, we apply BDG's inequality and Young's inequality. This gives us
\begin{align}\label{estiamateAalurmar10}
 &\mathbb E\bigg(\sup_{t\in(0, T\wedge\tau\wedge T_R)}\Vert\int_0^t e^{-(t-s)A}\mathcal  P\nabla {\bf f}({\bf u}(s))\,dW_s\Vert^{\frac{pq}{p-q}}_{L^r}\bigg)\leq C\bigg[\mathbb E\bigg(\int_0^{T\wedge\tau\wedge T_R}\Vert \nabla {\bf u}(t)\Vert_{r}^{{\frac{pq}{p-q}}}ds\bigg)\bigg]+CT^{\frac{pq}{2(p-q)}}.
 \end{align} 
 Hence, using  (\ref{Lpestimate_1and2_n}), (\ref{nablauR_2p_2_estimate}) and (\ref{estiamateAalurmar10}) in \eqref{nablau} we obtain that
 \begin{align}
&\mathbb E(\sup_{t\in(0,T\wedge\tau)}\|\nabla {\bf u}\|^{\frac{pq}{p-q}}_{L^{r}})\lesssim CT^{\frac{pq}{2(p-q)}}+C[\mathbb E(\int_0^{T\wedge\tau}\Vert \nabla{\bf u}(t)\Vert_{L^r}^{{\frac{pq}{p-q}}}ds)],
 \end{align}
 where $C>0$ is a constant depending on $\Vert {\bf u}_0\Vert_{L^r}$, $R$ and $T$.
Thus, by Gr\"{o}nwall inequality, we find
 \begin{align}\label{A1/2uR_estimate_global}
\mathbb{E}\sup_{t \le T\wedge\tau\wedge T_R}
\|\nabla  {\bf u}(t)\|_{L^{r}}^{\frac{pq}{p-q}}
\le C_{R,T } \, e^{C T},
\end{align}
where $r=2+\epsilon$ given in (\ref{pqr}) and $C_{R,T,p,q}>0$ is a constant depending on $R$, $T$.  

Combining (\ref{Lpestimate_1and2_n}) with $l=\frac{pq}{p-q}$, (\ref{uR_Lp_global_estimate}) and (\ref{A1/2uR_estimate_global}), we finish the proof of Proposition \ref{proposition1}.
\end{proof}
Now, we shall give the proof of Theorem \ref{thm1}.

 \begin{proof}[Proof of Theorem \ref{thm1}]
 { First, we will show that $\tau=\infty$ almost surely.}
Recall the definitions  \eqref{eq:tau_m}, \eqref{def:tau} and  \eqref{stopingtime1}.
First note that for any $R>0$ we can write
$$\{\tau\leq T\}= \{\tau\leq T\wedge T_R\}\cap\{T_R\geq T\} \cup \{\tau \leq T\}\cap \{T_R\leq T\}, $$
which implies that for any $T>0$
\begin{align*}
   \bP(\tau \leq T)& \leq \bP(\tau \leq T\wedge T_R) + \bP(T_R \leq T).
\end{align*}
Now using Proposition \ref{proposition1} we obtain for the first term on the right hand side that
\begin{align*}
    \bP(\tau\leq T\wedge T_R)& \leq \lim_{m\to\infty}\bP(\tau_m \leq T\wedge T_R) \\
    &\leq \lim_{m\to\infty} \bP \left( \sup_{t\in(0,T\wedge T_R)}\left(\|n(t)\|_{L^p} + \|\bu(t)\|_{L^{\frac{pq}{p-q}}}+\|\nabla\bu(t)\|_{L^r}\right) > m\right)\\
    &\leq \lim_{m\to\infty}\frac1{m^{\frac{pq}{p-q}}}\left(\mathbb E \sup_{t\in(0,T\wedge T_R)}\left(\Vert \nabla{\bf u}(t)\Vert_{L^r}^{\frac{pq}{p-q}}+\Vert {\bf u}(t)\Vert_{L^{\frac{pq}{p-q}}}^{\frac{pq}{p-q}}+\Vert n(t)\Vert_{L^{p}}^{\frac{pq}{p-q}}\right)\right)\\
    &\leq \lim_{m\to\infty}\frac1{m^{\frac{pq}{p-q}}} C_T(R) =0.
\end{align*}
Now,
for the second term, we use Chebyshev's inequality again. By using \eqref{stopingtime1} and \eqref{eq218jan} (see Lemma \ref{lem:modified_energy_growth}) we obtain for any $R>0$ that
\begin{align*}
     \bP(T_R \leq T) \leq \frac1R\bE \left(  \sup_{t\in[0,T{ \wedge \tau}]} \int_{\mathbb R^2} n(t)\ln^+ n(t) \, dx+
 \sup_{t\in[0,T{\wedge \tau}]} \| {\bf u}(t) \|^2_{L^2} \right) \leq \frac{C_{M,T}}{R}.
\end{align*}
Hence, we conclude that, for any $T>0$, we have
$$\bP(\tau\leq T)=0.$$
{Next, we will prove that the free energy bounds given in \eqref{energybounds} hold. First note that the higher regularity results obtained in Theorem \ref{thm:regularity} are now true global-in-time for $(n,\bu)$.} 
 
Now, we establish the global-in-time bound of $\int_{\mathbb R^2}n\ln(1+|x|^2)\,dx.$
 Noting that
\[
|\nabla \ln(1 + |x|^2)| \le \frac{2|x|}{1 + |x|^2} \le 1 \quad \text{and} \quad 
|\Delta \ln(1 + |x|^2)| = \frac{4}{(1 + |x|^2)^2} \le 4,
\]
we use the evolution equation (\ref{PKSNS-time-dependent})$_1$ for $n$ to obtain
\[
\begin{aligned}
{d} \int_{\mathbb{R}^2} n &\ln(1 + |x|^2) \, dx 
= - \int_{\mathbb{R}^2} {\bf u} \cdot \nabla n \ln(1 + |x|^2) \, dx \\
&\quad + \int_{\mathbb{R}^2} \Delta n\ln(1+|x|^2)\, dxdt-\int_{\mathbb R^2} \nabla\cdot(n\nabla c)\ln(1+|x|^2)dxdt\\
&= \int_{\mathbb{R}^2} n {\bf u} \cdot \nabla \ln(1 + |x|^2) \, dxdt +\int_{\mathbb R^2} n\Delta \ln(1+|x|^2)dxdt+\int n\nabla c\cdot \nabla(\ln(1+|x|^2)dxdt\\
&\le \int_{\mathbb{R}^2} n |{\bf u}| \, dxdt + 4Mdt + \int_{\mathbb{R}^2} n |\nabla c| \, dxdt \\
&\le \|n\|_{L^{\frac{pq}{pq-p+q}}} \|{\bf u}\|_{L^{\frac{pq}{p-q}}}dt + 4Mdt + \|n\|_{L^{\frac{pq}{pq-p+q}}} \|\nabla c\|_{L^{\frac{pq}{p-q}}}dt \\
&\le (\|n\|^2_{L^{\frac{pq}{pq-p+q}}}+ \|{\bf u}\|^2_{L^{\frac{pq}{p-q}}})dt + 4Mdt + (\|n\|^2_{L^{\frac{pq}{pq-p+q}}}+ \|\nabla c\|^2_{L^{\frac{pq}{p-q}}})dt \\
&\le (\|n_0\|_{L^{1}}^{\frac{2}{q}}\|n\|_{L^{p}}^{2(1-\frac{1}{q})}+ \|{\bf u}\|^2_{L^{\frac{pq}{p-q}}})dt + 4Mdt + (\|n_0\|_{L^1}^{{\frac{2}{q}}}\|n\|_{L^{p}}^{2(1-\frac{1}{q})} +\|n\|^2_{L^p})dt,
\end{aligned}
\]
where we used H\"{o}lder's inequality and Young's inequality.
Integrating over time yields
\[
\begin{aligned}
\int_{\mathbb{R}^2} n(t) \ln(1+|x|^2)\, dx 
&\le \int_{\mathbb{R}^2} n_0 \ln(1+|x|^2)\, dx  + 4Mt + C_M
\int_0^t (\|n\|_{L^{p}}^{2(1-\frac{1}{q})} + \|{\bf u}\|^2_{L^{\frac{pq}{p-q}}} +\|n\|^2_{L^p})(s)\, ds\\
&\le \int_{\mathbb{R}^2} n_0 \ln(1+|x|^2)\, dx  + 4Mt\\
&\quad + C_M
\sup_{s\le t}  \left(\|n(s)\|_{L^{p}}^{2(1-\frac{1}{q})} + \|{\bf u}(s)\|^2_{L^{\frac{pq}{p-q}}} +\|n(s)\|^2_{L^p}\right)\, t.
\end{aligned}
\]
Raising power $pq/(2(p-q))$ both sides, we obtain
\[
\begin{aligned}
\left(\int_{\mathbb{R}^2} n(t) \ln(1+|x|^2)\, dx \right)^{\frac{pq}{2(p-q)}}
&\le \left(\int_{\mathbb{R}^2} n_0 \ln(1+|x|^2)\, dx\right)^{\frac{pq}{2(p-q)}} + (4Mt)^{\frac{pq}{2(p-q)}}\\
&\quad + C \sup_{s\le t}  \left(\|n(s)\|_{L^{p}}^{(1-\frac{1}{q})\frac{pq}{p-q}} + \|{\bf u}(s)\|^{\frac{pq}{p-q}}_{L^{\frac{pq}{p-q}}} +\|n(s)\|^{\frac{pq}{p-q}}_{L^p}\right)\, t^{\frac{pq}{2(p-q)}}.
\end{aligned}
\]
Taking supremum over $t\in[0,T]$  then expectation yields
\begin{align}\label{estimate_nlog1+x}
\mathbb E\left[\sup_{t\le T}\left(\int_{\mathbb{R}^2} n(t) \ln(1+|x|^2)\, dx \right)^{k}\right]
&\le \left(\int_{\mathbb{R}^2} n_0 \ln(1+|x|^2)\, dx\right)^{k} + (4M T)^{k}\nonumber\\
&\quad + C \mathbb E\left[\sup_{s\le T}  \left(  \|{\bf u}(s)\|^{2k}_{L^{\frac{pq}{p-q}}} +\|n(s)\|^{2k}_{L^p}\right)\,\right] T^{k}\nonumber\\
&\le  \left(\int_{\mathbb{R}^2} n_0 \ln(1+|x|^2)\, dx\right)^{k} + (4M T)^{k}\nonumber\\
&\quad + C_{\Vert n_0\Vert_p,\Vert \bu_0\Vert_{\frac{pq}{p-q}},T } T^{k},
\end{align}
{where we used (\ref{base_estimate:nu})}  and $p$, $q$ are given in (\ref{pqr}).

We compute the evolution of free energy to the system \eqref{PKSNS-time-dependent}.
In light of divergence-free condition $\text{div}\, {\bf u} = 0$, we first apply the chain rule to $n\ln n$ associated with the equation for $n$ in (\ref{PKSNS-time-dependent}) and 
integrate it by parts to get
\begin{align}\label{eq1oct15}
d \int_{\mathbb{R}^2} n\ln n \, dx 
+ 4\|\nabla \sqrt{n}\|_{L^2}^2 dt 
&= \int_{\mathbb{R}^2} \nabla n\cdot \nabla c  dx dt. 
\end{align}

In addition, one has from (\ref{PKSNS-time-dependent}) that
\begin{align*}
d\int_{\mathbb R^2} nc\, dx=&-\int_{\mathbb R^2}\nabla n\cdot\nabla cdxdt+\int_{\mathbb R^2} {\bf u}\cdot n\nabla c dxdt+\int_{\mathbb R^2}n|\nabla c|^2dxdt+\int_{\mathbb R^2} nc_t
dxdt\\
=&-\int_{\mathbb R^2}\nabla n\cdot\nabla cdxdt+\int_{\mathbb R^2}{\bf u}\cdot  n\nabla cdxdt+\int_{\mathbb R^2}n|\nabla c|^2dxdt+\int_{\mathbb R^2}-\Delta c c_tdxdt\\
=&-\int_{\mathbb R^2}\nabla n\cdot\nabla cdxdt+\int_{\mathbb R^2}{\bf u}\cdot  n\nabla cdxdt+\int_{\mathbb R^2}n|\nabla c|^2dxdt+\int_{\mathbb R^2}\nabla c \cdot\nabla c_tdxdt,
\end{align*}
which implies
\begin{align}\label{eq1oct152}
d\int_{\mathbb R^2} nc\, dx-d(\frac{1}{2}\int_{\mathbb R^2}|\nabla c|^2 dx)
=&-\int_{\mathbb R^2}\nabla n\cdot\nabla cdxdt+\int_{\mathbb R^2}{\bf u}\cdot  n\nabla cdxdt+\int_{\mathbb R^2}n|\nabla c|^2dxdt.
\end{align}
Using the fact $({\bf u},({\bf u} \cdot \nabla){\bf u})_{L^2} = 0$, we compute the kinetic energy $\Vert u\Vert_{L^2}$ and apply It\^o's formula to get
\begin{align}\label{eq1oct153}
d\big(\frac{1}{2}\|{\bf u}\|_{L^2}^2\big)
= \Big( - \|\nabla {\bf u}\|_{L^2}^2 
    + \langle {\bf u}, n\nabla c \rangle 
    + \frac{1}{2}\|{\bf f}({\bf u})\|_{L_2(U;L^2)}^2 \Big)\,dt
    + \langle {\bf  u}, {\bf f}({\bf u})\,dW_t \rangle.
\end{align}
Combining (\ref{eq1oct15}), (\ref{eq1oct152}) with (\ref{eq1oct153}), one has
\begin{align}\label{rewrite2024mar8sun}
&d \int_{\mathbb{R}^2} n\ln n \, dx 
+ 4\|\nabla \sqrt{n}\|_{L^2}^2 dt -d\int_{\mathbb R^2} ncdx+d\big(\frac{1}{2}\int_{\mathbb R^2}|\nabla c|^2dx\big)+d(\frac{1}{2}\Vert {\bf u}\Vert_{L^2}^2)\nonumber\\
=&2\int\nabla n\cdot\nabla cdxdt-\int_{\mathbb R^2}n|\nabla c|^2dxdt\nonumber\\
&-\int_{\mathbb R^2}|\nabla {\bf u}|^2dxdt + \frac{1}{2}\|{\bf f}({\bf u})\|_{L_2(U;L^2)}^2 dt
    +  \sum_{j \ge 1}  
\big({\bf  u},  {\bf f}_j({\bf  u}) \big)_{L^2} \, dW^j.
\end{align}
In light of (\ref{freeenergy}), we rewrite (\ref{rewrite2024mar8sun}) as
 \begin{align}\label{dfintbeforemarsun}
 d\mathcal J+\int_{\mathbb R^2} n|\nabla(\ln  n-c)|^2dxdt+\int_{\mathbb R^2}|\nabla {\bf u}|^2dxdt=  \frac{1}{2}\|{\bf f}({\bf u})\|_{L_2(U;L^2)}^2 dt
    +  \sum_{j \ge 1}  
\big({\bf  u},  {\bf f}_j({\bf  u}) \big)_{L^2} \, dW^j.
 \end{align}
Next, we will first find a lower bound for $\mathcal{J}.$
Noting that
\[
\int_{\mathbb R^2} \frac{n} {M} \, dx \equiv  1,\text{ and}\quad M=\int_{\mathbb R^2} n_0dx<8\pi,
\]
and the fact that $-\Delta c =n$ we apply Lemma \ref{logHLSlemma} to obtain, for some $\delta_1>0$, that
\begin{align}\label{hlsineqaftermar8}
\int_{\mathbb{R}^2} n\ln n \, dx  -\frac{1}{2}\int_{\mathbb R^2} ncdx\geq \int_{\mathbb R^2} n\ln ndx-\frac{M}{8\pi}\int_{\mathbb R^2} n\ln ndx-C(M)\geq \delta_1\int_{\mathbb R^2} n\ln ndx-C(M),
\end{align}
where $C(M)$ is a deterministic constant appearing in Lemma \ref{logHLSlemma}.

We next bound $\int_{\mathbb R^2} n\ln n~dx$ from below in expectation.  By using Lemma \ref{lemm2mar9}, we have
\begin{align}\label{combinehlsafter2mar91}
\int_{\mathbb{R}^2} n \ln n \, dx
&\ge - \int_{\mathbb{R}^2} n \ln ^{-}n \, dx\ge - {2} \int_{\mathbb{R}^2} n \ln (1 + |x|^2) \, dx
     - M\ln \pi - \frac{1}{e}. 
\end{align}
Combining (\ref{hlsineqaftermar8}) with (\ref{combinehlsafter2mar91}), one finds 
\begin{align}\label{411mar9now}
&  \int_{\mathbb R^2} n\ln ndx-\frac{1}{2}\int_{\mathbb R^2}ncdx 
\geq -2\delta_1  \int_{\mathbb{R}^2} n \ln(1 + |x|^2) \, dx -M\delta_1 \ln \pi -\frac{\delta_1}{e}-C(M).
\end{align}
This implies that
$$\mathcal{J} \geq \frac12 \|\bu\|_{L^2}^2-2\delta_1  \int_{\mathbb{R}^2} n \ln(1 + |x|^2) \, dx -M\delta_1 \ln \pi -\frac{\delta_1}{e}-C(M).$$
Next, using this we will find upper bounds for the right hand side of \eqref{dfintbeforemarsun}.
In light of assumption (H1), we have from (\ref{estimate_nlog1+x}) that for $k\geq 1,$
and any $t\geq 0$ that,
\begin{align}\label{411realmar8now}
& \mathbb E\Big[\frac{1}{2}\int_0^t\|{\bf f}({\bf u})\|_{L_2(U;L^2)}^{2k}(s)\,ds\Big]\lesssim (1+\mathbb E\int_0^t\Vert {\bf u}\Vert^{2k}_{L^2}\,ds )\nonumber\\
 \lesssim&
 \mathbb E\big[\int_0^t|\mathcal J(t)|^k\,ds\big]+ C_k\bigg[2^k\delta_1^k \mathbb E\bigg[\int_0^t\bigg(\int_{\mathbb{R}^2} n \ln(1 + |x|^2) \, dx\bigg)^k\,ds \bigg] +(M\delta_1 \ln \pi)^kt +\frac{\delta_1^k}{e^k}t+C^k(M)t\bigg]\nonumber\\
\lesssim&\mathbb E\big[\int_0^t|\mathcal J(t)|^k\,ds\big]+ C_{k,\Vert n_0\Vert_p,\Vert {\bf u}_0\Vert_{\frac{pq}{p-q}},\delta_1,t }t^{k+1}+ \left(\int_{\mathbb{R}^2} n_0 \ln(1+|x|^2)\, dx\right)^{k}t.
\end{align}
By using BDG and Young's inequality, one finds
\begin{align}\label{413nowmar9mon}
\mathbb{E}\bigg[\sup_{s \le T} &\bigg|\int_{0}^s\sum_{j\geq 1}\left( {\bf u}(r),  {\bf f}_j({\bf u}(r)) \right)_{L^2} 
\, dW^j(r)\bigg|^k\bigg]
\le
\varepsilon \, \mathbb{E}\sup_{s \le T}\|{\bf u}(s)\|_{L^2}^{2k}
+
C_{\varepsilon,k}\,
\mathbb{E}\left(
\int_0^T \big(\|{\bf u}(r)\|_{L^2}^2 + 1\big)\,dr
\right)^k\nonumber\\
\leq& \varepsilon \mathbb E[\sup_{t\leq T}|\mathcal J(t)|^k]+C_{\varepsilon,k}\mathbb E\int_0^T |\mathcal J|^k(r)dr
+C_{\varepsilon,k}T^k+ C_{k,\Vert n_0\Vert_{p},\Vert {\bf u}_0\Vert_{\frac{pq}{p-q}},\delta_1,T }T^{(k+1)k}\nonumber\\
&+\left(\int_{\mathbb{R}^2} n_0 \ln(1+|x|^2)\, dx\right)^{k}T.
\end{align}
Now, take $\bE\sup_{0\leq t\leq T}$ on both sides of (\ref{dfintbeforemarsun}) and apply (\ref{411realmar8now}) and (\ref{413nowmar9mon}) to arrive at, 
\begin{align}\label{eqgronwallbefore}
 \mathbb E[\sup_{t\leq T}|\mathcal J(t)|^k]+\mathbb E\big(\int_0^{T}\mathcal G(r)dr\big)^k\lesssim &|\mathcal J(0)|^k+C_k\mathbb E\int_0^{T}|\mathcal J|^k(r)dr+C_{k,\Vert n_0\Vert_p,\Vert {\bf u}_0\Vert_{\frac{pq}{p-q}},\delta_1,T}T^{(k+1)k}\nonumber\\
&+ C_{k,\Vert n_0\Vert_p,\Vert {\bf u}_0\Vert_{\frac{pq}{p-q}},\delta_1,T }T^{k+1 }+\left(\int_{\mathbb{R}^2} n_0 \ln(1+|x|^2)\, dx\right)^{k}T,
 \end{align}
 where $C>0$ is a constant and
 $$\mathcal G(t):=\int_{\mathbb R^2}n|\nabla (\ln n-c)|^2dx+\Vert \nabla {\bf u}\Vert_{L^2}^2.$$
By applying the Gr\"{o}nwall inequality to (\ref{eqgronwallbefore}), one gets for $k\geq 1,$
\begin{align}\label{lastinequalitymar8sun}
 \mathbb E\bigg(\int_0^{T} \mathcal G(r) \, dr\bigg)^k &+\mathbb E[\sup_{t\leq T}|\mathcal J(t)|^k]
 \lesssim e^{C_kT} 
\bigg(
|\mathcal J(0)|^k+CT^{(k+1)k}
+\left(\int_{\mathbb{R}^2} n_0 \ln(1+|x|^2)\, dx\right)^{k}T
\bigg),
\end{align}
where $C>0$ depends on $T$, $k,\Vert n_0\Vert_p,\Vert {\bf u}_0\Vert_{\frac{pq}{p-q}}$.

\end{proof}

\section*{Acknowledgment}
The authors would like to acknowledge the support by the Department of Applied Mathematics at the University of Washington Seattle and the support by the National Science Foundation grant DMS-2553666 (transferred from DMS-2407197) awarded to Krutika Tawri. 

\bibliographystyle{plain}
\bibliography{ref}

\end{document}